\documentclass[a4paper,10pt,reqno]{amsart}
\usepackage{amsmath,amsfonts,amsthm,amssymb,xcolor}
\usepackage[T1]{fontenc}
\usepackage{pdfsync}
\usepackage{csquotes}
\usepackage{graphicx}
\usepackage{pstricks}
\usepackage{lmodern}
\usepackage{calc}
\usepackage{mathabx}
\usepackage{mathrsfs}
\usepackage{ulem}
\usepackage{hyperref}

\usepackage{accents}

\usepackage{enumerate}

\numberwithin{equation}{section}
\usepackage{tikz}

\newcommand{\be}{\beta}

\newcommand{\1}{\mathbf{1}}

\newcommand{\cq}{\mathcal{Q}}

\newcommand{\cob}{\color{black}}
\newcommand{\cop}{\color{blue}}

\newcommand{\R}{\mathbb R}

\newcommand{\Z}{\mathbb Z}
\newcommand{\T}{\mathbb T}

\newcommand{\cac}{\mathcal C}

\newcommand{\cf}{\mathcal F}

\newcommand{\ch}{\mathcal H}
\newcommand{\ci}{\mathcal I}
\newcommand{\cj}{\mathcal J}
\newcommand{\ck}{\mathcal K}
\newcommand{\cl}{\mathcal L}
\newcommand{\cm}{\mathcal M}

\newcommand{\cp}{\mathcal P}

\newcommand{\cpti}{\widetilde{\mathcal{P}}}

\newcommand{\al}{\alpha}

\newcommand{\ga}{\gamma}

\newcommand{\la}{\lambda}

\newcommand{\vp}{\varphi}

\newtheorem{theorem}{Theorem}[section]

\newtheorem{corollary}[theorem]{Corollary}

\newtheorem{lemma}[theorem]{Lemma}
\newtheorem{notation}[theorem]{Notation}

\newtheorem{proposition}[theorem]{Proposition}

\theoremstyle{remark}
\newtheorem{remark}[theorem]{Remark}

\pgfdeclareshape{crosscircle}
{
  \inheritsavedanchors[from=circle] 
  \inheritanchorborder[from=circle]
  \inheritanchor[from=circle]{north}
  \inheritanchor[from=circle]{north west}
  \inheritanchor[from=circle]{north east}
  \inheritanchor[from=circle]{center}
  \inheritanchor[from=circle]{west}
  \inheritanchor[from=circle]{east}
  \inheritanchor[from=circle]{mid}
  \inheritanchor[from=circle]{mid west}
  \inheritanchor[from=circle]{mid east}
  \inheritanchor[from=circle]{base}
  \inheritanchor[from=circle]{base west}
  \inheritanchor[from=circle]{base east}
  \inheritanchor[from=circle]{south}
  \inheritanchor[from=circle]{south west}
  \inheritanchor[from=circle]{south east}
  \inheritbackgroundpath[from=circle]
  \foregroundpath{
    \centerpoint%
    \pgf@xc=\pgf@x%
    \pgf@yc=\pgf@y%
    \pgfutil@tempdima=\radius%
    \pgfmathsetlength{\pgf@xb}{\pgfkeysvalueof{/pgf/outer xsep}}%
    \pgfmathsetlength{\pgf@yb}{\pgfkeysvalueof{/pgf/outer ysep}}%
    \ifdim\pgf@xb<\pgf@yb%
      \advance\pgfutil@tempdima by-\pgf@yb%
    \else%
      \advance\pgfutil@tempdima by-\pgf@xb%
    \fi%
    \pgfpathmoveto{\pgfpointadd{\pgfqpoint{\pgf@xc}{\pgf@yc}}{\pgfqpoint{-0.707107\pgfutil@tempdima}{-0.707107\pgfutil@tempdima}}}
    \pgfpathlineto{\pgfpointadd{\pgfqpoint{\pgf@xc}{\pgf@yc}}{\pgfqpoint{0.707107\pgfutil@tempdima}{0.707107\pgfutil@tempdima}}}
    \pgfpathmoveto{\pgfpointadd{\pgfqpoint{\pgf@xc}{\pgf@yc}}{\pgfqpoint{-0.707107\pgfutil@tempdima}{0.707107\pgfutil@tempdima}}}
    \pgfpathlineto{\pgfpointadd{\pgfqpoint{\pgf@xc}{\pgf@yc}}{\pgfqpoint{0.707107\pgfutil@tempdima}{-0.707107\pgfutil@tempdima}}}
  }
}
\makeatother

\definecolor{gr}{rgb}   {0.,   0.69,   0.23 }
\definecolor{bl}{rgb}   {0.,   0.5,   1. }
\definecolor{mg}{rgb}   {0.85,  0.,    0.85}
\definecolor{yl}{rgb}   {0.8,  0.7,   0.}
\definecolor{or}{rgb}  {0.7,0.2,0.2}

\colorlet{symbols}{black!90!black}
\colorlet{symbolsb}{black!90!black}
\colorlet{testcolor}{green!60!black}

\usetikzlibrary{shapes.misc}
\usetikzlibrary{shapes.symbols}
\usetikzlibrary{decorations}
\usetikzlibrary{decorations.markings}

\def\drawx{\draw[-,solid] (-3pt,-3pt) -- (3pt,3pt);\draw[-,solid] (-3pt,3pt) -- (3pt,-3pt);}
\tikzset{
	root/.style={circle,fill=testcolor,inner sep=0pt, minimum size=2mm},
	dot/.style={circle,fill=black,inner sep=0pt, minimum size=1mm},
	var/.style={circle,fill=black!10,draw=black,inner sep=0pt, minimum size=2mm},
	dotred/.style={circle,fill=black!50,inner sep=0pt, minimum size=2mm},
	generic/.style={semithick,shorten >=1pt,shorten <=1pt},
	dist/.style={ultra thick,draw=testcolor,shorten >=1pt,shorten <=1pt},
	testfcn/.style={ultra thick,testcolor,shorten >=1pt,shorten <=1pt,<-},
	testfcnx/.style={ultra thick,testcolor,shorten >=1pt,shorten <=1pt,<-,
		postaction={decorate,decoration={markings,mark=at position 0.6 with {\drawx}}}},
	kprime/.style={semithick,shorten >=1pt,shorten <=1pt,densely dashed,->},
	kprimex/.style={semithick,shorten >=1pt,shorten <=1pt,densely dashed,->,
		postaction={decorate,decoration={markings,mark=at position 0.4 with {\drawx}}}},
	kernel/.style={semithick,shorten >=1pt,shorten <=1pt,->},
	multx/.style={shorten >=1pt,shorten <=1pt,
		postaction={decorate,decoration={markings,mark=at position 0.5 with {\drawx}}}},
	kernelx/.style={semithick,shorten >=1pt,shorten <=1pt,->,
		postaction={decorate,decoration={markings,mark=at position 0.4 with {\drawx}}}},
	kernel1/.style={->,semithick,shorten >=1pt,shorten <=1pt,postaction={decorate,decoration={markings,mark=at position 0.45 with {\draw[-] (0,-0.1) -- (0,0.1);}}}},
	kernel2/.style={->,semithick,shorten >=1pt,shorten <=1pt,postaction={decorate,decoration={markings,mark=at position 0.45 with {\draw[-] (0.05,-0.1) -- (0.05,0.1);\draw[-] (-0.05,-0.1) -- (-0.05,0.1);}}}},
	kernelBig/.style={semithick,shorten >=1pt,shorten <=1pt,decorate, decoration={zigzag,amplitude=1.5pt,segment length = 3pt,pre length=2pt,post length=2pt}},
	rho/.style={dotted,semithick,shorten >=1pt,shorten <=1pt},
	renorm/.style={shape=circle,fill=white,inner sep=1pt},
	labl/.style={shape=rectangle,fill=white,inner sep=1pt},
	xi/.style={circle,fill=symbols!10,draw=symbols,inner sep=0pt,minimum size=1.2mm},
	xiblack/.style={circle,fill=symbolsb,draw=symbolsb,inner sep=0pt,minimum size=1.2mm},
	xix/.style={crosscircle,fill=symbols!10,draw=symbols,inner sep=0pt,minimum size=1.2mm},
	xib/.style={circle,fill=symbols!10,draw=symbols,inner sep=0pt,minimum size=1.6mm},
	xibx/.style={crosscircle,fill=symbols!10,draw=symbols,inner sep=0pt,minimum size=1.6mm},
	not/.style={circle,fill=symbols,draw=symbols,inner sep=0pt,minimum size=0.5mm},
	notblack/.style={circle,fill=symbolsb,draw=symbolsb,inner sep=0pt,minimum size=0.5mm},
	>=stealth,
	}
\makeatletter
\def\DeclareSymbol#1#2#3{\expandafter\gdef\csname MH@symb@#1\endcsname{\tikz[baseline=#2,scale=0.15,draw=symbols]{#3}}\expandafter\gdef\csname MH@symb@#1s\endcsname{\scalebox{0.7}{\tikz[baseline=#2,scale=0.15,draw=symbols]{#3}}}}
\def\<#1>{\csname MH@symb@#1\endcsname}
\makeatother

\DeclareSymbol{circle}{0.5}{\draw (0,0.7) node[xi] {};}
\DeclareSymbol{line}{0.5}{\draw (0,0.2) node[not] {} -- (0,1.4) node[not] {};}

\DeclareSymbol{Psi}{0.5}{\draw (0,0) node[not] {} -- (0,1.5) node[xi] {};}
\DeclareSymbol{Psiblack}{0.5}{\draw (0,0) node[notblack] {} -- (0,1.5) node[xiblack] {};}
\DeclareSymbol{Psi2}{0.5}{\draw (-0.6,1.5) node[xi] {} -- (0,0) node[not] {} -- (0.6,1.5) node[xi] {};}
\DeclareSymbol{IPsi2}{0}{\draw (0,1) -- (0.8,2.2) node[xi] {};\draw (0,-0.25) node[not] {} -- (0,1) node[not] {} -- (-0.8,2.2) node[xi] {};}
\DeclareSymbol{PsiIPsi2}{0}{\draw (0,1) -- (0.8,2.2) node[xi] {};\draw (0,-0.25) node[not] {} -- (0,1) node[not] {} -- (-0.8,2.2) node[xi] {};\draw (0,-0.25) node[not] {} -- (0.8,1) node[xi] {};}

\DeclareSymbol{IPsi2dot}{0}{\draw[-,dashed] (0,1) -- (0.8,2.2) node[xi] {};\draw[-,dashed] (0,-0.25) node[not] {} -- (0,1) node[not] {} -- (-0.8,2.2) node[xi] {};}

\date{\today}

\title{On the 1d stochastic Schr{\"o}dinger product}

\author{Aurélien Deya}
\thanks{\textit{Adress.} Université de Lorraine, CNRS, IECL, F-54000 Nancy, France; aurelien.deya@univ-lorraine.fr}

\keywords{Stochastic Schr{\"o}dinger equation, fractional noise, renormalization}

\begin{document}

\begin{abstract}
We exhibit various restrictions about the wellposedness of the Schr{\"o}dinger product
$$\cl:z \longmapsto -\imath \int_0^t e^{\imath  s {\cop \partial^2_x}}\big( z_s\cdot \Psi_s\big) ds  $$ 
where $\Psi$ refers to the so-called linear solution of the stochastic Schr{\"o}dinger problem. We focus more specifically on the case where $\Psi$ satisfies 
\begin{equation}\label{starting-equation-abstract}
(\imath \partial_t-\partial^2_x)\Psi=\dot{B}, \quad \Psi_0=0,\quad \quad t\in \R, \ x\in \mathbb{T},
\end{equation}
where $\dot{B}$ is a white noise in space with fractional time covariance of index $H>\frac12$.

\smallskip

{\cop As an consequence of our analysis, we obtain that if $H$ is close to $\frac12$ (that is $\dot{B}$ is close to a space-time white noise), then it is essentially impossible to treat the stochastic NLS problem 
\begin{equation*}
(\imath \partial_t-\partial^2_x)u= |u|^2+\dot{B}, \quad u_0=0,\quad \quad  t\in \R, \ x\in \mathbb{T},
\end{equation*}
using only a first-order expansion of the solution (\enquote{$u=\Psi+z$}).}

\end{abstract}

\maketitle

\section{Presentation of the problem}
\color{blue}
In this paper, we propose to point out some limitations in the analysis of (local) wellposedness for the stochastic quadratic Schr{\"o}dinger equation
\begin{equation}\label{general-deter-nls}
(\imath \partial_t-{\cop \partial^2_x})u=|u|^2+\dot{B}, \quad u(0,.)=0, \quad t\in \R, \ x\in \T,
\end{equation}
where $\dot{B}$ is a stochastic noise in a family to be specified. 

\smallskip

This model in fact belongs to the broader class of nonlinear stochastic Schr{\"o}dinger equations
\begin{equation}\label{general-deter-nls-gene}
(\imath \partial_t-{\cop \partial^2_x})u=\lambda\, u^{p} \overline{u}^q+\dot{B}, \quad u(0,.)=0, \quad t\in \R, \ x\in \T,
\end{equation}
where $p,q\geq 0$ are two fixe integers and $\lambda$ is a real parameter.\color{black}

\smallskip

The study of noise influence on NLS models is a recurring topic in the SPDE literature. The most widely covered situation - by far - is that of a white noise in time with suitably colored spatial covariance. Provided such a noise is regular enough, the solution of \eqref{general-deter-nls} is expected to take values in a space of functions (almost surely); the powerful It{\^o} integration tools then become available, which even opens the possibility to treat multiplicative perturbations (see e.g. \cite{debou-debu-2003,cheung-mosincat,oh-okamoto,oh-popovnicu-wang} for additive-noise models, and \cite{BRZ,debou-debu-1999,BM,hornung} for a multiplicative noise).

\smallskip

In contrast with this \enquote{functional} case, we are here interested in rougher \enquote{distributional} situations - for which literature turns out to be much more scarce. Namely, in the continuation of \cite{DFT}, we will focus on examples for which the equation can only be handled in a space of \textit{negative-order distributions}, and for which \textit{renormalization procedures} are also required.

\smallskip

In order to implement these ideas, let us consider, throughout the analysis, the case of a white-in-space fractional-in-time noise. Thus, for some (fixed) index $H\in (\frac12,1)$, $\dot{B}$ is here defined as the centered Gaussian noise with covariance given by the formula
\begin{equation}
\mathbb{E}\Big[\dot{B}(s,x)\overline{\dot{B}(t,y)}\Big]=|t-s|^{2H-2} \delta_{\{x=y\}}
\end{equation}
or otherwise stated: for all test-functions $\vp,\psi$ on $\R\times \T$,
\begin{equation}
\mathbb{E}\Big[\langle \dot{B},\vp\rangle \overline{\langle \dot{B},\psi\rangle}\Big]=\int_{\mathbb{T}} dx \int_{\R^2} \frac{ds dt}{|t-s|^{2-2H}} \, \vp(s,x)\overline{\psi(t,x)}.
\end{equation}
Such a noise can be equivalently described through the expansion
\begin{equation}\label{repres-noise}
\dot{B}(t,x)=\sum_{k\in \Z} \dot{\beta}^{(k)}_t  e^{-\imath k x},
\end{equation}
where the $\dot{\beta}^{(k)}$'s are independent (real-valued) fractional noises on $\R$, with common Hurst index $H>\frac12$.

\

Now, it is a classical fact that the expected regularity of the solution $u$ in \eqref{general-deter-nls} is prescribed by that of the associated linear problem
\begin{equation}\label{linear-pb}
(\imath \partial_t-{\cop \partial^2_x})\Psi=\dot{B}, \quad \Psi(0,.)=0, \quad t\in \R, \ x\in \T.
\end{equation}

We will thus rely on the following preliminary lemma to identify the distributional cases of interest.

\begin{lemma}\label{lem:regu-luxo}
In the above setting, the following assertions hold true.

\smallskip

\noindent
$(i)$ If $H>\frac34$, then for every $T>0$, one has $\Psi\in L^2([0,T]\times \T)$ almost surely.

\smallskip

\noindent
$(ii)$ If $H\leq \frac34$, then for every $T>0$, one has $\mathbb{E}\Big[\big\|\Psi\big\|_{L^2([0,T]\times \mathbb{T})}^2\Big] = \infty$.

\end{lemma}
The statement of item $(ii)$ corresponds to a slightly extended version of \cite[Proposition 2.1, item (ii)]{DFT} (take $H_0=H$ and $H_1=\frac12$ therein), and it can be proved with the same arguments. As for item $(i)$, we have included a sketch of its proof in Section \ref{sec:proof-regu-luxo}, for the sake of clarity.

\

Based on the result of Lemma \ref{lem:regu-luxo}, we henceforth focus on the case where $H\in (\frac12,\frac34]$, that is on situations where the solution $\Psi$ of \eqref{linear-pb}, and accordingly the solution $u$ of \eqref{general-deter-nls}, cannot be defined as functions.

\smallskip

Before we describe the so-called \textit{first-order strategy} at the core of our investigations, let us introduce a few notations and spaces that will be used throughout the paper.

\

\begin{notation}

From now on and for the rest of the paper:

\smallskip

\noindent
$\bullet$ We use the classical convention $\1_{[0,a]}:=-\1_{[a,0]}$ if $a<0$.

\smallskip

\noindent
$\bullet$ We fix a smooth symmetric function $\chi:\R\to [0,1]$ such that $\chi\equiv 1$ on $[-1,1]$ and $\text{supp}(\chi)\subset [-\frac32,\frac32]$.

\smallskip

\noindent
$\bullet$ For every function $f:\R\to \R$ and every $\la\in \R$, we denote the Fourier transform (in time) as 
$$\color{blue}\cf(f)(\la):=\int_{\R}dt \, e^{-\imath \la t}f(t) .\color{black}$$

\smallskip

\noindent
$\bullet$ We denote by $\ci_\chi$ the local integration operator, that is for all $f:\R\to\R$ and $t\in \R$,
\begin{equation}\label{not-ci-chi}
\ci_\chi f(t):= - \imath \chi(t) \int_0^t ds \, \chi(s)f(s).
\end{equation}
We also denote by $\ci_\chi(.,.)$ the Fourier kernel associated with this operator, {\cop given by
\begin{equation}\label{kernel-int-simple}
\ci_\chi(\la,\la'):=\int_{\R} dt\, e^{-\imath \la t}\chi(t)\int_{\R} ds \, e^{\imath \la' s} \chi(s)\1_{[0,t]}(s)
\end{equation}
and such that}
\begin{equation}\label{i-chi-fou}
\cf\big(\ci_\chi f\big)(\la)=\int_{\R} d\la_1 \, \ci_\chi(\la,\la_1) \cf(f)(\la_1).
\end{equation}

\noindent
$\bullet$ For every function $\vp:\mathbb{T}\to \R$ and every $k\in \mathbb{Z}$, we set
\begin{equation}\label{not:fouri-space}
\vp_k:=\int_{\T}dx \, e^{-\imath k x}\vp(x)=\frac1{2\pi}\int_0^{2\pi}dx \, e^{-\imath k x}\vp(x)  .
\end{equation}

\smallskip

\noindent
$\bullet$ \color{blue} We consider the scale of Bourgain spaces 
\begin{equation}\label{bourgain-sp}
\{Z^{c,b}, \, c\in [0,1), \, b \in (\frac12,1)\}
\end{equation}
defined through the norm
\begin{equation}\label{defi:x-b-c}
\big\|z\big\|_{Z^{c,b}}^2:=\sum_{k\in \Z}\langle k\rangle^{2c}\int_{\R} d\la  \, \langle \la\rangle^{2b} \big| \cf(z_k)(\la)\big|^2,
\end{equation}
where the notation $\langle .\rangle$ refers to $\langle \la\rangle:=(1+|\la|^2)^{\frac12}$ (see Remark \ref{rk:bourgain-spaces} for further details about the Bourgain-space terminology).

\smallskip

We will also be led to consider the spaces
\begin{equation}\label{defi:zti-c-b}
\widetilde{Z}^{c,b}:=Z^{c,0} \cap Z^{0,b},
\end{equation}
naturally associated with the norm
\begin{equation}\label{defi:zti-c,b-norm}
\big\|z\big\|_{\widetilde{Z}^{c,b}}^2:=\sum_{k\in \Z}\int_{\R} d\la  \, \big\{\langle k\rangle^{2c}+\langle \la\rangle^{2b}\big\} \big| \cf(z_k)(\la)\big|^2,
\end{equation}
noting that $Z^{c,b} \subset \widetilde{Z}^{c,b}$ for all $c\in [0,1), b\in (\frac12,1)$.
\color{black}

\smallskip

\noindent
$\bullet$ We denote by $(\Omega,\mathfrak{F},\mathbb{P})$ the complete probability space which accommodates the fractional noise $\dot{B}$ under consideration (or equivalently the one-parameter fractional noises $(\dot{\beta}^{(k)})$ in \eqref{repres-noise}).

\smallskip

\noindent
$\bullet$ For any Banach space $E$, we define
$$L^{\infty,-}(\Omega;E):=\bigcap_{p\geq 1} L^p(\Omega;E).$$

\noindent
$\bullet$ Last but not least, we denote by $\dot{B}^{(n)}$ the (spatial) regularization of $\dot{B}$ derived from \eqref{repres-noise} through the formula
\begin{equation}\label{defi:dot-b-n}
\dot{B}^{(n)}(t,x)=\sum_{k: \, \langle k\rangle \leq 2^n} \dot{\beta}^{(k)}_t  e^{-\imath k x}.
\end{equation}

\end{notation}

\

\subsection{The first-order strategy and related questions}\label{subsec:first-ord-stra}

\

\smallskip

Recall that we concentrate here on situations where the solution $u$ of \eqref{general-deter-nls} is not expected to be a function, so that there is no a priori interpretation of the product term {\cop $|u|^2$}. 

\smallskip

{\cop To overcome this issue, a natural idea is to consider a first-order expansion of the solution. This strategy has become classical in the context of stochastic parabolic equations (it is often referred to as the \enquote{Da Prato–Debussche trick}, following its introduction in \cite{daprato-debussche}), as well as in the dispersive setting with random initial data (see, for instance, \cite{bourgain, McKean}). The reader may also consult \cite{DFT, FOW} for some early applications to stochastic NLS models of the form \eqref{general-deter-nls-gene}.}

\smallskip

We propose to explain the details behind this approach at the level of the approximated equation first.

\smallskip

For $\dot{B}^{(n)}$ defined as in \eqref{defi:dot-b-n}, let $u^{(n)}$ be the (well-defined) solution of the approximated equation
\begin{equation}\label{general-deter-nls-approx}
(\imath \partial_t-{\cop \partial^2_x})u^{(n)}={\cop | u^{(n)}|^2}+\dot{B}^{(n)}, \quad u^{(n)}(0,.)=0, \quad t\in \R, \ x\in \T.
\end{equation}
With the notation in \eqref{not-ci-chi}, and setting 
\begin{equation}\label{defi-luxo}
\<Psi>^{(n)}:=\ci_\chi \big(e^{\imath  .{\cop \partial^2_x}}\dot{B}^{(n)}\big) ,  \quad v^{(n)}(t):=e^{\imath  t {\cop \partial^2_x}} u^{(n)}(t),
\end{equation}
it is easy to check that we can (locally) recast \eqref{general-deter-nls-approx} under the mild formulation
\begin{equation}\label{basic-equation0-quad}
v^{(n)}(t)= \<Psi>^{(n)}(t)+{\cop \ci_\chi \cm\big(v^{(n)},v^{(n)}\big)(t)},
\end{equation}
where the product operator ${\cop \cm}$ is defined in Fourier coordinates by the formula {\cop
\begin{equation} \label{cm}
\cm(v,w)_k(t):=\sum_{k_1} e^{\imath t\Omega_{k,k_1-k}} v_{k_1}(t) \overline{w_{k_1-k}(t)}
\end{equation}
with \enquote{resonant function} $\Omega_{k,k_1-k}$ given by
$$\Omega_{k,k_1-k}:=-k^2+k_1^2-(k-k_1)^2=2k(k_1-k).$$
}

\smallskip

\color{blue}

\begin{remark}\label{rk:bourgain-spaces}
Independently of the presence of noise, the topology most commonly used to study equation \eqref{general-deter-nls-approx} is that of the Bourgain spaces $X^{c,b}$, defined via the norm
\begin{equation}\label{bourg}
\| u\|^2_{X^{c,b}}:=\sum_k \, \langle k\rangle^{2c}\int d\la \, \langle \lambda -|k|^2\rangle^{2b}  \big|\cf(u_k)(\la)\big|^2.
\end{equation}
We adopt the same convention here, expressed this time at the level of the transformed process $v$ (see \eqref{defi-luxo}). Indeed, the norm $\big\|.\big\|_{Z^{c,b}}$ introduced in \eqref{defi:x-b-c} - and used throughout what follows - can also be written as
\begin{equation*}
\|v\|_{Z^{c,b}}^2= \|e^{- \imath t \Delta}v\|_{X^{c,b}}^2.
\end{equation*}
Recall also that the fundamental embedding
$$X^{c,b} \subset \cac(\R;H^c(\T))$$
which ensures in particular continuity of the solution at time $0$, is only guaranteed if $b>\frac12$ (see \cite[Corollary 2.10]{tao-book}). This justifies our restriction to $b>\frac12$ in \eqref{bourgain-sp}.
\end{remark}

\color{black}

\

In light of \eqref{basic-equation0-quad}, a natural first-order transformation of the problem simply consists in the consideration of the difference process 
$$z^{(n)}:=v^{(n)}-\<Psi>^{(n)},$$
which now satisfies the equation
\begin{align}
z^{(n)}&= \ci_\chi {\cop \cm\big(z^{(n)}+\<Psi>^{(n)},z^{(n)}+\<Psi>^{(n)}\big)}\nonumber\\
&{\cop =\ci_\chi \cm\big(z^{(n)},z^{(n)}\big)+\ci_\chi \cm\big(\<Psi>^{(n)},\<Psi>^{(n)}\big)+\ci_\chi \cm\big(z^{(n)},\<Psi>^{(n)}\big)+\ci_\chi \cm\big(\<Psi>^{(n)},z^{(n)}\big).}\label{equa-z-n}
\end{align}

\smallskip

The whole point of this change of perspective can be roughly summed up as follows: taking the action of the operator {\cop $\ci_\chi \cm$} into account, we \textit{hope} the difference-process $z^{(n)}$ to be sufficiently regular (and at least more regular than $\<Psi>^{(n)}$ or $v^{(n)}$) so that the product terms in \eqref{equa-z-n} can now make sense as $n\to\infty$.

\

\cop

The so-called first-order strategy is now based on the implementation of a fixed-point argument for equation \eqref{equa-z-n}, by establishing control over each term within a yet-to-be-specified space in the scale $\{Z^{c,b}\}$ (see \eqref{defi:x-b-c}).

\

\begin{remark}\label{rk:bourgain-spaces}
Independently of the presence of noise, the topology most commonly used to study equation \eqref{general-deter-nls-approx} is that of the Bourgain spaces $X^{c,b}$, defined via the norm
\begin{equation}\label{bourg}
\| u\|^2_{X^{c,b}}:=\sum_k \, \langle k\rangle^{2c}\int d\la \, \langle \lambda -|k|^2\rangle^{2b}  \big|\cf(u_k)(\la)\big|^2.
\end{equation}
We adopt the same convention here, expressed this time at the level of the transformed process $v$ (see \eqref{defi-luxo}). Indeed, the norm $\big\|.\big\|_{Z^{c,b}}$ introduced in \eqref{defi:x-b-c} - and used throughout what follows - can also be written as
\begin{equation*}
\|v\|_{Z^{c,b}}^2= \|e^{- \imath t \Delta}v\|_{X^{c,b}}^2.
\end{equation*}
Recall also that the fundamental embedding
$$X^{c,b} \subset \cac(\R;H^c(\T))$$
which ensures in particular continuity of the solution at time $0$, is only guaranteed if $b>\frac12$ (see e.g. \cite[Corollary 2.10]{tao-book}). This justifies our restriction to $b>\frac12$ in \eqref{bourgain-sp}.
\end{remark}

\color{black}

\

{\cop Returning to equation \eqref{equa-z-n}, let us now observe that the continuity of the (purely deterministic) mapping
$$Z^{c,b} \rightarrow Z^{c,b}, \quad z \mapsto \ci_\chi \cm\big(z,z\big)$$
is ensured by the sole condition $b>\frac12$. This can be seen as a consequence of Bourgain's fundamental estimates in \cite{Bou1} (see the proof of \cite[Proposition 3.5]{DFT} for details). We may therefore focus our attention on the analysis of the three stochastic components in \eqref{equa-z-n}, namely the (Schr{\"o}dinger) products involving the linear solution $\<Psi>^{(n)}$.} In this context, the following important observation immediately comes to mind.

\

\noindent
\textbf{Observation.} In order to guarantee the convergence of equation \eqref{equa-z-n}, one should \textit{at least} be able to address the following two problems:

\smallskip

\noindent
$\mathbf{(P1)}$ Does the sequence of - explicit - stochastic processes
\begin{equation}\label{multil-tree}
{\cop \ci_\chi \cm\big(\<Psi>^{(n)},\<Psi>^{(n)}\big)}
\end{equation}
converge in $L^{\infty,-}(\Omega;{\cop Z^{c,b}})$ as $n\to \infty$ (for $b,c$ to be determined) ?

\smallskip

\noindent
$\mathbf{(P2)}$ Does the stochastic Schr{\" o}dinger product operation
\begin{equation}\label{produ-ope}
{\cop \cl^{(n)}:z \longmapsto \ci_\chi \cm\big(z,\<Psi>^{(n)}\big)}
\end{equation}
converge as a \textit{random operator} from ${\cop Z^{c,b}}$ to ${\cop Z^{c,b}}$ as $n\to \infty$ (for $b,c$ to be determined) ?

\

The two above problems $\mathbf{(P1)}$ and $\mathbf{(P2)}$ will be our guidelines in the subsequent study. However, for a proper examination of these questions in the rough stochastic framework, the above formulations both need to be refined, which is the purpose of the two next sections.

\subsubsection{About renormalization}\label{subsec:renorm-cond} 
Owing to the pathwise irregularity of $\dot{B}$ (especially as $H$ gets close to $\frac12$), the convergence of the process in \eqref{multil-tree} can only be achieved through a renormalization trick, as developed in Section \ref{subsec:schro-prod-tree} below. Consequently, equation \eqref{equa-z-n} can only be handled in a renormalized sense. 

\smallskip

Nevertheless, in order not to deviate from the existing SPDE literature (especially the known results in the heat or wave settings), we shall impose this renormalization procedure to be \textit{explicit}. In other words, the transformation should only give rise to explicit renormalizing constants at the level of the approximated equation.

\smallskip

In order to achieve this objective, our strategy will obey the following two rules:

\smallskip

\noindent
$\mathbf{(C1)}$ We \textit{do allow} the use of (natural) renormalization procedures for the explicit process 
$${\cop \ci_\chi \cm\big(\<Psi>^{(n)},\<Psi>^{(n)}\big)}.$$ 

\smallskip

\noindent
$\mathbf{(C2)}$ We \textit{do not allow} any renormalization (and so any deformation) for the general product operation
$$\cl^{(n)}:z \longmapsto {\cop \ci_\chi \cm\big(z,\<Psi>^{(n)}\big)}.$$

\

\begin{remark}
Observe that the condition $\mathbf{(C2)}$ immediately rules out any \enquote{a priori deformation of the product}, such as a Skorohod-type interpretation of the problem (see e.g. \cite{CDOT2}), or the renormalization trick implemented in \cite{FOW} for the {\cop stochastic} cubic model, namely:
$$
(\imath \partial_t-{\cop \partial^2_x})u=\bigg( |u|^2-\int_{\T} |u|^2\bigg)u+\dot{B}, \quad u(0,.)=\Phi, \quad t\in \R, \ x\in \T.
$$
It is clear indeed that the correction term $\int_{\T} |u|^2$ derived from the latter transformation is not explicit, i.e. it is not explicitly defined in terms of $\dot{B}$ {\cop (this observation also applies to the renormalization used in \cite{bourgain} for a deterministic dynamics with random initial data)}.

\end{remark}

\

With these considerations in mind, and following the condition $\mathbf{(C1)}$, we can refine the formulation of the problem $\mathbf{(P1)}$ as follows:

\smallskip

\noindent
$\mathbf{(P1')}$ Can the sequence of stochastic processes
\begin{equation*}
{\cop \ci_\chi \cm\big(\<Psi>^{(n)},\<Psi>^{(n)}\big)}
\end{equation*}
be suitably renormalized so as to converge in $L^{\infty,-}(\Omega;{\cop Z^{c,b}})$ (for $b,c$ to be determined) ?

\

\subsubsection{Stochastic Schr{\" o}dinger product as a random operator}\label{subsec:sto-pro-ran-op}

\

\smallskip

Let us now go back to the formulation of the problem $\mathbf{(P2)}$, related to the control of the (Schr{\" o}dinger) product operation $\cl^{(n)}$ in \eqref{produ-ope}. Following the above condition $\mathbf{(C2)}$, we intend to tackle this product operation directly, that is without any renormalization trick. 

\smallskip

Recall that by {\cop \eqref{cm}}, one has {\cop
\begin{align}
\cl^{(n)}(z)_k&=\ci_{\chi}\Big(t\mapsto \sum_{k_1} e^{\imath t\Omega_{k,k_1-k}} z_{k_1}(t) \overline{\<Psi>^{(n)}_{k_1-k}}(t)\Big), \quad \text{with} \ \ \Omega_{k,k_1-k}=2k(k_1-k).\label{ana-l-n}
\end{align}}

\smallskip

Based on this expression, one can morally expect the desired regularization effect (in space) to stem from the integration (in time) of the exponential factor $e^{\imath t{\cop \Omega_{k,k_1-k}}}$ in \eqref{ana-l-n}. {\cop Since $\Omega_{k,k_1-k}$ vanishes for $k=0$ and $k=k_1$ only, we are naturally led to the decomposition of $\cl^{(n)}$ as a sum of the \enquote{resonant} (Schr{\" o}dinger) product operator
\begin{align}\label{cl-circ-1-1}
&\cl^{\circ,(n)}(z)_k=\sum_{k_1} \1_{\{k= 0\}\cup\{k_1=k\}}\ci_{\chi}\Big(t\mapsto z_{k_1}(t)\overline{\<Psi>^{(n)}_{k_1-k}}(t)\Big),
\end{align}
and the \enquote{non-resonant} (Schr{\" o}dinger) product
\begin{align}\label{cl-sharp-1-1}
&\cl^{\sharp,(n)}(z)_k=\1_{\{k\neq 0\}}\ci_{\chi}\bigg(t\mapsto\sum_{k_1\neq k}e^{\imath t\color{blue}\Omega_{k,k_1-k}\color{black}} z_{k_1}(t)\overline{\<Psi>^{(n)}_{k_1-k}}(t)\bigg).
\end{align}}

Let us now particularize the formulation of the problem $\mathbf{(P2)}$ to each of these two components.

\smallskip

As far as $\cl^{\circ,(n)}$ is concerned, observe that the {\cop strong degeneracy condition $\1_{\{k= 0\}\cup\{k_1=k\}}$ morally reduces the operator to two (time-dependent) vectors. in this setting, one can legitimately hope for a direct analysis of the operator and a direct treatment of the central question: }

\

\noindent
$\mathbf{(P2')}$ Does the operator norm $\big\|\cl^{\circ,(n)}\big\|_{{\cop Z^{c,b}} \to {\cop Z^{c,b}}}$ converge (in $L^{\infty,-}(\Omega)$) as $n\to\infty$ (for suitable $b,c$) ?

\

Unfortunately, due to the much higher sophistication of the non-resonant component $\cl^{\sharp,(n)}$, the evaluation of the random - implicitly defined - norm $\big\|\cl^{\sharp,(n)}\big\|_{{\cop Z^{c,b}} \to {\cop Z^{c,b}}}$ turns out to be a much more difficult task. In fact, capturing the value of an operator norm in $L({\cop Z^{c,b}},{\cop Z^{c,b}})$ (or in any other distributions scale) is known to be a rarely attainable objective in general, and this observation is all the more true in our random setting. For this reason, we shall instead focus on the analysis of a more tractable estimate of $\big\|\cl^{\sharp,(n)}\big\|_{{\cop Z^{c,b}} \to {\cop Z^{c,b}}}$.

\smallskip

In order to introduce the latter quantity, observe that $\cl^{\sharp,(n)}$ can be written in Fourier coordinates as
\begin{equation}\label{kernel-random-op}
\cf\big( \cl^{\sharp,(n)}( z)_k\big)(\la)=\sum_{k_1}\int d\la_1 \big(\ck^{(n)}_\chi\big)_{kk_1}(\la,\la_1) \cf(z_{k_1})(\la_1),
\end{equation}
{\cop where the kernel $\ck^{(n)}_\chi$ is explicitly given by
\begin{equation}\label{def:ck}
\big(\ck^{(n)}_\chi\big)_{kk_1}(\la,\la_1):=\1_{\{k\neq 0\}}\1_{\{k_1\neq k\}}\int_{\R} d\la_2 \, \overline{\cf\big(\<Psi>^{(n)}_{k_1-k}\big)(\la_2)}\ci_\chi(\la,\Omega_{k,k_1-k}+\la_1-\la_2).
\end{equation}}
Now, based on this kernel formulation, it is easily checked (see Section \ref{subsec:gene-ope-esti} for details) that for every $z\in {\cop Z^{c,b}}$,
\begin{align}
&\big\|\cl^{\sharp,(n)} (z)\big\|_{{\cop Z^{c,b}}}\leq \big\|z\big\|_{{\cop Z^{c,b}}}\cdot \cp^{(n)}_{c,b},\label{claimed-bou}
\end{align}
where $\cp^{(n)}_{c,b}$ is given by
\begin{align}
&\cp^{(n)}_{c,b}:=\sum_{k_1,k_1'}\int_{\R^2} \frac{d\la_1}{{\cop \langle k_1 \rangle^{2c}\langle \la_1\rangle^{2b}}}\frac{d\la_1'}{{\cop \langle k'_1 \rangle^{2c}\langle \la'_1\rangle^{2b}}} \bigg| \sum_k \int d\la \, {\cop \{\langle k \rangle^{2c}\langle\la\rangle^{2b}\}}\, \big(\ck^{(n)}_{\chi}\big)_{kk_1}(\la,\la_1) \overline{\big(\ck^{(n)}_{\chi}\big)_{kk_1'}(\la,\la_1')}\bigg|^2.\label{defi:cp-n-gene}
\end{align}

\smallskip

With bound \eqref{claimed-bou} in mind, the - explicit - estimate $\cp^{(n)}_{c,b}$ is thus the quantity that will serve us as a landmark in the analysis of $\cl^{\sharp,(n)}$. Along this idea, let us particularize the previous control issue $\mathbf{(P2)}$ to $\cl^{\sharp,(n)}$ through the following simplified version of the problem.  

\

\noindent
$\mathbf{(P2'')}$ Does the quantity $\cp^{(n)}_{c,b}$ converge in $L^{\infty,-}(\Omega)$ (or even in $L^{1}(\Omega)$) as $n\to\infty$ (for suitable $b,c$) ?

\

\begin{remark}
In accordance with the developments in \cite{DFT}, we consider the approximation of $\big\|\cl^{\sharp,(n)}\big\|_{{\cop Z^{c,b}} \to {\cop Z^{c,b}}}$ by $\cp^{(n)}_{c,b}$ as a full part of the first-order strategy here described. Note that similar kernel-based estimates of operator norms are also extensively used in the so-called \textit{theory of random tensors} developed by Deng, Nahmod and Yue (see e.g. item (5) of Proposition 5.1 or the proof of Proposition 6.1 in \cite{DNY}).\\
\indent We have included a short discussion about the sharpness of this approximation in Sections \ref{subsec:gene-ope-esti} and \ref{subsec:young-op} below. In particular, it is therein shown that in the Young integration setting, the consideration of the corresponding quantity $\cp^{(n)}$ allows us to recover the well-known threshold value $H=\frac12$ for the Hurst index.

\end{remark}

\smallskip

\color{blue}

To conclude this (partial) heuristic analysis of problem \eqref{general-deter-nls}, let us observe that, although the spaces $Z^{c,b}$ naturally constitute the reference topology in this context, the questions raised above could just as well have been formulated using other norms. For instance, by adopting the norm introduced in \eqref{defi:zti-c,b-norm}, problem $\mathbf{(P2'')}$ immediately becomes:

\

\noindent
$\mathbf{(P2''')}$ Does the quantity $\cpti^{(n)}_{c,b}$ defined by
\begin{align*}
&\cpti^{(n)}_{c,b}:=\sum_{k_1,k_1'}\int_{\R^2} \frac{d\la_1}{{\cop \langle k_1 \rangle^{2c}+\langle \la_1\rangle^{2b}}}\frac{d\la_1'}{{\cop \langle k'_1 \rangle^{2c}+\langle \la'_1\rangle^{2b}}} \bigg| \sum_k \int d\la \, {\cop \{\langle k \rangle^{2c}+\langle\la\rangle^{2b}\}}\, \big(\ck^{(n)}_{\chi}\big)_{kk_1}(\la,\la_1) \overline{\big(\ck^{(n)}_{\chi}\big)_{kk_1'}(\la,\la_1')}\bigg|^2
\end{align*}
converge in $L^{1}(\Omega)$ as $n\to\infty$ (for suitable $b,c$) ?

\color{black}

\

\subsection{Objective of the study}

\

\smallskip

{\cop The above-described first-order approach to \eqref{general-deter-nls}} is precisely the method that was implemented - with success - in \cite{DFT}, for a fractional noise of index $H>\frac58$. In particular, the three questions $\mathbf{(P1')}$, $\mathbf{(P2')}$ and $\mathbf{(P2'')}$ all received \textit{positive} answers in the latter situation.

\

\textit{Our aim in the present study is to show that this first-order approach is however not sufficient to cover the whole range $\frac12<H<1$}, thus advocating for more sophisticated developments in rougher situations than those treated in \cite{DFT}. This conclusion will be derived from a close examination of the challenging convergence issues $\mathbf{(P1')}$, $\mathbf{(P2')}$ and {\cop $\mathbf{(P2'')}$-$\mathbf{(P2''')}$}. 

\section{Main results}

With the above presentation of the problem in mind, we are in a position to state our main results related to the three central questions $\mathbf{(P1')}$, $\mathbf{(P2')}$ and {\cop $\mathbf{(P2'')}$-$\mathbf{(P2''')}$}.

\subsection{Problem (P1'): control of the renormalized Schr{\"o}dinger product tree}\label{subsec:schro-prod-tree}

\

\smallskip

Let us start with the examination of the tree process $\<IPsi2>^{(n)}:=\ci_\chi \cm\big(\<Psi>^{(n)},\<Psi>^{(n)}\big)$, which we shall refer to as the \textit{(Schr{\"o}dinger) product tree} in the sequel.

\smallskip

As we evoked it earlier, this quantity needs to be renormalized before we can study its convergence. To this end, we will successively rely on three classical rescaling steps:

\smallskip

\noindent
$\mathbf{(R)}$ A first partial space averaging
$$\cm(\<Psi>^{(n)},\<Psi>^{(n)}) \quad \longrightarrow \quad \cm(\<Psi>^{(n)},\<Psi>^{(n)})-\<Psi>^{(n)}\int \overline{\<Psi>^{(n)}},$$
in the spirit of Bourgain's renormalization method for the cubic NLS model (see e.g. \cite{bourgain,CO}). {\cop This transformation allows us to eliminate a first set of resonant terms (that is, terms for which the resonant function vanishes): indeed, in Fourier mode, it can be checked from \eqref{cm} that
$$\Big(\cm(\<Psi>^{(n)},\<Psi>^{(n)})-\<Psi>^{(n)}\int \overline{\<Psi>^{(n)}}\Big)_k=\sum_{k_1\neq 0} e^{\imath t\Omega_{k,k_1}} \<Psi>^{(n)}_{k+k_1}(t) \overline{\<Psi>^{(n)}_{k_1}}(t).$$
}

\smallskip

\noindent
$\mathbf{(R')}$ Then a more standard space averaging:
\small
\begin{equation}\label{defi:cm-tild}
\cm(\<Psi>^{(n)},\<Psi>^{(n)})-\<Psi>^{(n)}\int \overline{\<Psi>^{(n)}} \ \longrightarrow \ \widetilde{\cm}(\<Psi>^{(n)},\<Psi>^{(n)}):=\bigg[\cm(\<Psi>^{(n)},\<Psi>^{(n)})-\<Psi>^{(n)}\int \overline{\<Psi>^{(n)}}\bigg] -\int \bigg[\cm(\<Psi>^{(n)},\<Psi>^{(n)})-\<Psi>^{(n)}\int \overline{\<Psi>^{(n)}}\bigg],
\end{equation}
\normalsize
{\cop which achieves the removal of the resonant terms: namely, in Fourier coordinates,
\begin{equation} \label{cm-fourier-intro}
\widetilde{\cm}(\<Psi>^{(n)},\<Psi>^{(n)})_k(t):=\1_{\{k\neq 0\}}\sum_{k_1\neq 0} e^{\imath t\Omega_{k,k_1}} \<Psi>^{(n)}_{k+k_1}(t) \overline{\<Psi>^{(n)}_{k_1}}(t), \quad \text{with} \ \Omega_{k,k_1}:=2kk_1.
\end{equation}
}

\smallskip

\noindent
$\mathbf{(R'')}$ Finally, a stochastic Wick renormalization trick: 
$$\widetilde{\cm}( \<Psi>^{(n)},\<Psi>^{(n)}) \quad \longrightarrow \quad \widetilde{\cm}( \<Psi>^{(n)},\<Psi>^{(n)})-\mathbb{E}\big[\widetilde{\cm}( \<Psi>^{(n)},\<Psi>^{(n)})\big].$$

\

\begin{remark}
Observe that the rescaling terms in Steps $\mathbf{(R')}$ and $\mathbf{(R'')}$ only involve reduced quantities, i.e. quantities depending on at most two of the three parameters $(t,x,\omega)$, which can indeed be expected from any reasonable renormalization trick.

On the other hand, the rescaling term in $\mathbf{(R)}$ still appeals to the \enquote{fully-dependent} quantity $\<Psi>^{(n)}$. However, the latter can easily be turned into a linear correction drift at the level of the approximated equation, making it acceptable in the procedure (see \cite[Section 1.2]{DFT} for details). 
\end{remark}

\

As a result of the three steps above, we derive the following renormalized version of the product tree:
\begin{equation}\label{def:sq-tree-renorm-intro}
\widetilde{\<IPsi2>}^{(n)}_k(t):=\ci_\chi \Big(\widetilde{\cm}( \<Psi>^{(n)},\<Psi>^{(n)})_k(.)-\mathbb{E}\big[\widetilde{\cm}( \<Psi>^{(n)},\<Psi>^{(n)})_k(.)\big]\Big)(t),
\end{equation}
where the renormalized product $\widetilde{\cm}( \<Psi>^{(n)},\<Psi>^{(n)})$ is defined by \eqref{defi:cm-tild}.


\

Our main result about this (renormalized) product tree can now be stated as follows.

\begin{proposition}\label{prop:ice-cream-frac-intro}
Assume that {\color{blue}$H\in (\frac12,\frac34)$}. Then the following picture holds true.

\smallskip

\noindent
$(i)$ For every $0\leq c<\frac12$, one has
$$\sup_{n\geq 1} \mathbb{E}\Big[\Big\|\widetilde{\<IPsi2>}^{(n)}\Big\|_{{\color{blue}Z^{c,0}}}^2\Big] <\infty.$$

\smallskip

\noindent
$(ii)$ For every $\frac12< b <2H-\frac12$, one has
$$\sup_{n\geq 1} \mathbb{E}\Big[\Big\|\widetilde{\<IPsi2>}^{(n)}\Big\|_{{\color{blue}Z^{0,b}}}^2\Big] <\infty.$$

\smallskip

\noindent
$(iii)$ If $b = 2H-\frac12$, then
$$\mathbb{E}\Big[\Big\|\widetilde{\<IPsi2>}^{(n)}\Big\|_{{\color{blue}Z^{0,b}}}^2\Big] \stackrel{n\to\infty}{\longrightarrow} \infty.$$
{\color{blue} More generally, if $0\leq c< 1$ and $\frac12<b<1$ are such that $b+\frac{c}{2}\geq 2H-\frac12$, then
$$\mathbb{E}\Big[\Big\|\widetilde{\<IPsi2>}^{(n)}\Big\|_{Z^{c,b}}^2\Big] \stackrel{n\to\infty}{\longrightarrow} \infty.$$}
\end{proposition}

\smallskip

The above properties thus provide us with the exact time regularity of $\widetilde{\<IPsi2>}$ (namely $(2H-\frac12)-$), which, given the central role of the process in the dynamics of \eqref{general-deter-nls}, can be seen as an important result of independent interest.

\smallskip

In view of our present objective, the condition {\cop $b+\frac{c}{2}<2H-\frac12$} will be our first restriction in the application of the first-order strategy described in Section \ref{subsec:first-ord-stra}.

\

\subsection{Problem (P2'): about the Schr{\" o}dinger product operator $\cl^{\circ,(n)}$}

\

\smallskip

Let us now turn to the issues related to the Schr{\" o}dinger product operator $\cl^{(n)}$, starting from its resonant - and relatively simple - component $\cl^{\circ,(n)}$ (see \eqref{cl-circ-1-1}).

\smallskip

Our main result about $\cl^{\circ,(n)}$ will actually be derived from a careful examination of the action of the operator on suitable (Gaussian) processes. The property can be summed as follows.

\begin{proposition}\label{prop:q-sharp-optim-intro}
Assume that $H\in (\frac12,\frac34)$ and fix $b\in (\frac12,1)$. 

\smallskip

If $0\leq c<\frac32-2H$, then there exists a sequence $({\cop Y^{(n)}})_{n\geq 1}$ of random functions such that
\begin{equation}\label{bound-z-n}
\sup_{n\geq 1} \mathbb{E}\Big[ \big\|{\cop Y^{(n)}}\big\|_{{\color{blue}Z^{c,b}}}^q\Big] < \infty \quad \text{for every} \ q\geq 2, \quad \text{and}\quad \mathbb{E}\Big[ \big\| \cl^{\circ,(n)}{\cop Y^{(n)}}\big\|_{{\color{blue}L^2(\T \times \R)}}^2\Big] \stackrel{n\to\infty}{\longrightarrow} \infty.
\end{equation} 
In particular, if $0\leq c<\frac32-2H$, then for every $p>2$, one has
\begin{equation}\label{explos-op-1}
\mathbb{E}\Big[\big\|\cl^{\circ,(n)}\big\|_{{\color{blue}Z^{c,b}} \to {\color{blue}Z^{c,b}}}^p\Big] \stackrel{n\to\infty}{\longrightarrow} \infty,
\end{equation}
{\cop as well as
\begin{equation}\label{explos-op-2}
\mathbb{E}\Big[\big\|\cl^{\circ,(n)}\big\|_{{\color{blue}\widetilde{Z}^{c,b}} \to {\color{blue}\widetilde{Z}^{c,b}}}^p\Big] \stackrel{n\to\infty}{\longrightarrow} \infty,
\end{equation}
where the space $\widetilde{Z}^{c,b}$ has been introduced in \eqref{defi:zti-c-b}.}

\end{proposition}

Proposition \ref{prop:p-p-q} therefore offers a partial answer to the guideline question $\mathbf{(P2')}$. In particular here, the result gives birth to our second restriction on the spaces ${\cop Z^{c,b}}$ (or ${\cop \widetilde{Z}^{c,b}}$) involved in the first-order analysis: namely, one must have $c\geq \frac32-2H$. 

\smallskip

\color{blue}

\begin{remark}
As we shall see in the course of the proof (see Section \ref{sec:prod-restr-space}, and in particular inequality \eqref{inega-intro}), the blow-up observed in \eqref{bound-z-n} is primarily due to the divergence of the term corresponding to $k=0$ in the definition \eqref{cl-circ-1-1} of $\cl^{\circ,(n)}$. This observation partially echoes the result established in \cite[Proposition 3.2]{liu} for an analogous product on the two-dimensional torus (with a noisy input of slightly different nature).

\end{remark}

\

The combination of the restriction in Proposition \ref{bound-z-n} with the constraints arising from Proposition \ref{prop:ice-cream-frac-intro} already rules out the possibility of covering the interval $H\in (\frac12,1)$ through a first-order analysis based on the spaces  $\{Z^{c,b}: \, c\in [0,1), \, b\in (\frac12,1)\}$.

\begin{corollary}\label{coroz}
If $\frac12<H\leq \frac{7}{12}$, then for any pair $(c,b)\in [0,1)\times (\frac12,1)$, one has either
\begin{equation}\label{explosi-bourgain}
\big\|\widetilde{\<IPsi2>}^{(n)}\big\|_{Z^{c,b}}\stackrel{n\to\infty}{\longrightarrow} \infty,  \quad \text{or} \quad  \big\|\cl^{\circ,(n)}\big\|_{Z^{c,b} \to Z^{c,b}} \stackrel{n\to\infty}{\longrightarrow} \infty.
\end{equation}
In particular, for $\frac12<H\leq \frac{7}{12}$, the stochastic Schr{\"o}dinger problem \eqref{general-deter-nls} cannot be treated in the scale $\{Z^{c,b}\}$ with the first-order strategy described in Section \ref{subsec:first-ord-stra}.
\end{corollary}

\begin{proof}
For none of the two explosions in \eqref{explosi-bourgain} to happen, it is necessary that $b+\frac{c}{2}<2H-\frac12$ (Proposition \ref{prop:ice-cream-frac-intro}) and $c\geq \frac32-2H$ (Proposition \ref{prop:q-sharp-optim-intro}), which can be summed up as 
$$\frac32-2H \leq c < 4H-1-2b< 4H-2,$$
and hence one must have $H>\frac{7}{12}$.
\end{proof}

\

In light of the above observation, {\it we now turn to the broader spaces $\widetilde{Z}^{c,b}$}. Indeed, the results of Proposition \ref{prop:ice-cream-frac-intro} show that $\widetilde{\<IPsi2>} \in \widetilde{Z}^{c,b}$ for all $0\leq c <\frac12$ and $\frac12<b<2H-\frac12$. As a result, the additional constraint $c\geq \frac32-2H$ arising from Proposition \ref{prop:q-sharp-optim-intro} does not, at this stage, prevent the possibility of covering the interval $H\in (\frac12,1)$ by working within these spaces.

\smallskip

The continuity of the deterministic multiplication $z \longmapsto \ci_\chi \cm\big(z,z\big)$ from $\widetilde{Z}^{c,b}$ to $\widetilde{Z}^{c,b}$ is guaranteed by the result of Proposition \ref{prop:control-m-z-z-ztilde}, under the general condition $b>\frac12$ and $0\leq c \leq b$. Our next step is to address the non-resonant component $\cl^{\sharp,(n)}$ of the multiplication operator with $\<Psi>^{(n)}$. 

\ 

\color{black}

\subsection{{\cop Problem $\mathbf{(P2''')}$}: about the Schr{\" o}dinger product operator $\cl^{\sharp,(n)}$}

\

\smallskip

As explained in the above Section \ref{subsec:sto-pro-ran-op}, we focus here on the convergence issue for the approximation ${\cop \cpti^{(n)}_{c,b}}$ of $\big\|\cl^{\sharp,(n)}\big\|_{{\cop \widetilde{Z}^{c,b}} \to {\cop \widetilde{Z}^{c,b}}}$ given by the expression
\begin{align}
&{\cop \cpti^{(n)}_{c,b}}:=\nonumber\\
&\sum_{k_1,k_1'}\int_{\R^2} \frac{d\la_1}{\langle k_1 \rangle^{2c}+\langle \la_1\rangle^{2b}}\frac{d\la_1'}{\langle k'_1 \rangle^{2c}+\langle \la'_1\rangle^{2b}} \bigg| \sum_k \int d\la \, \{\langle k \rangle^{2c}+\langle\la\rangle^{2b}\}\, \big(\ck^{(n)}_{\chi}\big)_{kk_1}(\la,\la_1) \overline{\big(\ck^{(n)}_{\chi}\big)_{kk_1'}(\la,\la_1')}\bigg|^2,\label{defi:cp-n}
\end{align}
{\cop where the Fourier kernel $\ck^{(n)}_\chi$ has been introduced in \eqref{def:ck}.}

\begin{proposition}\label{prop:p-p-q}
Let $H\in (\frac12,\frac34)$, $b\in (\frac12,1)$ and $c\in (0,1)$. 

\smallskip

If $c\geq b- \frac14$, then one has
$$\mathbb{E}\Big[ \big|{\cop \cpti^{(n)}_{c,b}}\big|\Big] \stackrel{n\to\infty}{\longrightarrow} \infty.$$

\end{proposition}

Proposition \ref{prop:p-p-q} therefore points out our third (and last) restriction on the class of spaces ${\cop \widetilde{Z}^{c,b}}$ suitable for a first-order strategy: one must impose that $c< b- \frac14$.

\

By gathering the constraints exhibited in Propositions \ref{prop:ice-cream-frac-intro}, \ref{prop:q-sharp-optim-intro} and \ref{prop:p-p-q}, we can conclude our investigations about the limits of the first-order strategy for the stochastic Schr{\"o}dinger problem \eqref{general-deter-nls}.

\begin{corollary}\label{coro:ztilde}
If $\frac12<H\leq \frac{9}{16}$, then for any pair $(b,c)\in (\frac12,1)\times [0,1)$, one has either
\begin{equation}\label{explosi}
\big\|\widetilde{\<IPsi2>}^{(n)}\big\|_{\cop \widetilde{Z}^{c,b}}\stackrel{n\to\infty}{\longrightarrow} \infty,  \quad \big\|\cl^{\circ,(n)}\big\|_{\cop \widetilde{Z}^{c,b} \to \widetilde{Z}^{c,b}} \stackrel{n\to\infty}{\longrightarrow} \infty, \quad \text{or} \quad  {\cop \cpti^{(n)}_{c,b}} \stackrel{n\to\infty}{\longrightarrow} \infty \quad \text{in} \ L^{\infty,-}(\Omega).
\end{equation}
In particular, for $\frac12<H\leq \frac{9}{16}$, the stochastic Schr{\"o}dinger problem \eqref{general-deter-nls} cannot be treated {\cop in the scale $\{\widetilde{Z}^{c,b}\}$} with the first-order strategy described in Section \ref{subsec:first-ord-stra}.
\end{corollary}

\begin{proof}

For none of the three explosions in \eqref{explosi} to happen, it is necessary that $b<2H-\frac12$ (Proposition \ref{prop:ice-cream-frac-intro}), $c\geq \frac32-2H$ (Proposition \ref{prop:q-sharp-optim-intro}) and $c< b- \frac14$ (Proposition \ref{prop:p-p-q}), which can be summed up as 
$$\frac32-2H \leq c < b-\frac14< 2H-\frac34,$$
and hence one must have $H>\frac{9}{16}$.
\end{proof}

\

These results thus call for the development of more sophisticated methods, such as paracontrolled or random-tensor-type strategies, in order to cover the whole range $H>\frac12$ for $\dot{B}$. To be more specific, we do not expect any possible improvement regarding the constraints on the Schr{\"o}dinger product tree (Proposition \ref{prop:ice-cream-frac-intro}), and we only advocate for a more sophisticated treatment of the product operation $z\mapsto {\cop \ci_\chi \cm(z,\<Psi>^{(n)})}$. These further developments could for instance be derived from a suitable \enquote{ansatz} formulation of the problem (see \cite[Section 5.2]{DNY}), which we plan to investigate in a future study.

\

Finally, it is worth noting that there is no hope to reach the case of a space-time white noise $\dot{B}$ through the present renormalization method, as it can be seen from the following explosion result {\cop (see Section \ref{sec:explo-white})}.

\begin{proposition}\label{prop:explo-white}
Assume that $H=\frac12$, that is $\dot{B}$ is a space-time white noise on $\R\times \mathbb{T}$. Then, with the notation of Section \ref{subsec:schro-prod-tree} and for every {\cop $b > \frac12$}, it holds that
$$\mathbb{E}\Big[\Big\|\widetilde{\<IPsi2>}^{(n)}\Big\|_{{\cop Z^{0,b}}}^2\Big] \stackrel{n\to\infty}{\longrightarrow} \infty.$$
\end{proposition}

\

\color{blue}

\begin{remark}
The heuristic considerations developed in Section \ref{subsec:first-ord-stra} could of course have been expressed in the same manner for the general nonlinear model \eqref{general-deter-nls-gene}. Thus, if $u^{(n)}$ denotes the solution to the approximate equation
\begin{equation}\label{general-deter-nls-approx-gene}
(\imath \partial_t-{\cop \partial^2_x})u^{(n)}=\lambda\, (u^{(n)})^{p} (\overline{u^{(n)}})^q+\dot{B}^{(n)}, \quad u^{(n)}(0,.)=0, \quad t\in \R, \ x\in \T.
\end{equation}
and if we similarly define
\begin{equation}\label{defi-luxo-gene}
\<Psi>^{(n)}:=\ci_\chi \big(e^{\imath  .{\cop \partial^2_x}}\dot{B}^{(n)}\big) ,  \quad v^{(n)}(t):=e^{\imath  t {\cop \partial^2_x}} u^{(n)}(t),
\end{equation}
then the problem \eqref{basic-equation0-quad} naturally generalizes into the form
\begin{equation}\label{basic-equation0}
v^{(n)}(t)= \<Psi>^{(n)}(t)+\ci_\chi \cm^{(p,q)}\big(v^{(n)},\ldots,v^{(n)}\big)(t),
\end{equation} 
where the product operator $\cm^{(p,q)}$ is defined in Fourier coordinates by the formula
\begin{align} 
&\cm^{(p,q)}(v^{(1)},\dots,v^{(p)},w^{(1)},\ldots,w^{(q)})_k(t)\nonumber\\
&=\sum_{m}\sum_{\substack{k_1,\ldots,k_p\\k_1+\ldots+k_p=m}}\sum_{\substack{\ell_1,\ldots,\ell_p\\ \ell_1+\ldots+\ell_p=m-k}}  e^{\imath t\Omega_{k,\underline{k},\underline{\ell}}} v^{(1)}_{k_1}(t)\cdots v^{(p)}_{k_p}(t) \overline{w^{(1)}_{\ell_1}}(t)\cdots \overline{w^{(q)}_{\ell_q}}(t), \label{cm-fourier-gene}
\end{align}
with
\begin{equation}\label{defi-omega-gene}
\underline{k}:=(k_1,\ldots,k_p), \ \underline{\ell}:=(\ell_1,\ldots,\ell_q) \quad \text{and} \quad \Omega_{k,\underline{k},\underline{\ell}}=\Omega_{k,\underline{k},\underline{\ell}}^{(p,q)}:=-k^2+(k_1^2+\ldots+k_p^2)-(\ell_1^2+\ldots+\ell_q^2).
\end{equation}
This leads naturally to the two questions $\mathbf{(P1)}$ and $\mathbf{(P2)}$, about the control of the stochastic process
$$
\ci_\chi \cm^{(p,q)}\big(\<Psi>^{(n)},\ldots,\<Psi>^{(n)}\big)
$$
on the one hand, and of the random linear operator
$$
\cl^{(n)}:z \longmapsto \ci_\chi \cm^{(p,q)}\big(z,\<Psi>^{(n)},\ldots,\<Psi>^{(n)}\big).
$$
It turns out, however, that the answers to these two questions depend significantly on the form of the nonlinearity, in other words, on the values of the integers $p,q\geq 0$. This is readily seen by examining the behavior of the resonance function $\Omega_{k,\underline{k},\underline{\ell}}$ at the core of expression \eqref{cm-fourier-gene}. For example, if $p=0$ and $q=2$, we have
$$\Omega_{k,\underline{k},\underline{\ell}}=-k^2-\ell_1^2-\ell_2^2.$$
This expression vanishes only in the very specific case where $k=\ell_1=\ell_2=0$, which, in the context of the operator $\ci_\chi \cm^{(p,q)}$, suggests a more pronounced regularizing effect than in the case $p=q=1$, and thereby leads to conclusions that differ from those of Propositions \ref{prop:ice-cream-frac-intro}, \ref{prop:q-sharp-optim-intro}, and \ref{prop:p-p-q} (see \cite{KPV,Tao} for a comparison of the regularizing effects in these two situations).

\end{remark}

\color{black}

\

\cop

\begin{remark}
Let us recall that the first-order approach considered here was successfully implemented in \cite{DFT} in the case $H>\frac58$. On the other hand, Corollaries \ref{coroz} and \eqref{coro:ztilde} rule out its applicability for $H <\frac{9}{16}$ within the scales $\{Z^{c,b}\}$ or $\{\widetilde{Z}^{c,b}\}$. At this stage, it remains unclear whether a refinement of the arguments in \cite{DFT} would allow one to address the remaining interval $\frac{9}{16}<H<\frac58$ by means of a first-order expansion of the same kind.
\end{remark}

\color{black}

\


\

The rest of the paper is organized as follows. In Section \ref{sec:schrod-prod-tree}, we focus on the analysis of the Schr{\"o}dinger product tree and on the proof of Proposition \ref{prop:ice-cream-frac-intro} {\cop (as well as Proposition \ref{prop:explo-white})}. Sections \ref{sec:prod-restr-space} and \ref{sec:item-ii} are then devoted to the study of the (stochastic) Schr{\"o}dinger product operator $\cl^{(n)}$: Section \ref{sec:prod-restr-space} contains the proof of Proposition \ref{prop:q-sharp-optim-intro}, while Section \ref{sec:item-ii} contains the proof of Proposition \ref{prop:p-p-q}. {\cop Finally, the appendix is divided into three parts: Section \ref{sec:proof-regu-luxo} contains a (partial) proof of Lemma \ref{lem:regu-luxo}, Section \ref{sec:z-z-ztilde} is devoted to the proof of the continuity of the map $z\mapsto \ci_\chi\cm(z,z)$ in the $\widetilde{Z}^{c,b}$-topology, and Section \ref{sec:discuss-p-n} consists in a discussion about the relevance of the approximation ${\cop \cpti^{(n)}_{c,b}}$.}

\

\cop

{\it In accordance with the above-described setting and results, we fix the Hurst parameter $H\in (\frac12,\frac34)$ for the whole study, except in Section \ref{sec:explo-white} - where $H=\frac12$.}
\color{black}

\section{Convergence of the Schr{\"o}dinger product tree}\label{sec:schrod-prod-tree}

We start with the study of the convergence issue for the renormalized Schr{\"o}dinger product tree
\begin{equation}\label{def:sq-tree-renorm}
\widetilde{\<IPsi2>}^{(n)}_k(t):=\ci_\chi \Big(\widetilde{\cm}( \<Psi>^{(n)},\<Psi>^{(n)})_k(.)-\mathbb{E}\big[\widetilde{\cm}( \<Psi>^{(n)},\<Psi>^{(n)})_k(.)\big]\Big)(t),
\end{equation}
introduced in Section \ref{subsec:schro-prod-tree}. We recall in particular that the (rescaled) linear solution $\<Psi>^{(n)}$ has been introduced in \eqref{defi-luxo}, while the renormalized product operator $\widetilde{\cm}$ is defined in Fourier coordinates by
\begin{equation} \label{cm-fourier}
\widetilde{\cm}(v,w)_k(t):=\1_{\{k\neq 0\}}\sum_{k_1\neq 0} e^{\imath t\Omega_{k,k_1}} v_{k+k_1} \overline{w_{k_1}}, \quad \text{with} \ \ \Omega_{k,k_1}:=|k+k_1|^2-|k_1|^2-|k|^2=2k k_1.
\end{equation}

As a preliminary step, let us point out some useful expressions and estimates for the covariance of the process $\<Psi>^{(n)}$, at the core of expression \eqref{def:sq-tree-renorm}.

\subsection{Covariance of the linear solution}

\

\smallskip

One has by definition
\begin{align*}
\cf(\<Psi>^{(n)}_k)(\la)&=\1_{\{\langle k\rangle \leq 2^n\}}\cf\big(\ci_\chi\big( e^{-\imath . k^2}\dot{\beta}^{(k)}\big)\big)(\la)\\
\end{align*}
and we can use the integration kernel $\ci_\chi(.,.)$ (see \eqref{i-chi-fou}) to express this quantity as
\begin{align*}
\cf(\<Psi>^{(n)}_k)(\la)
&=\1_{\{\langle k\rangle \leq 2^n\}}\int d\la_1\, \ci_\chi(\la,\la_1)\cf\big(e^{-\imath . k^2}\dot{\beta}^{(k)}\big)(\la_1)\\
&=\1_{\{\langle k\rangle \leq 2^n\}}\int d\la_1\, \ci_\chi(\la,\la_1)\int dt \, e^{-\imath (\la_1+k^2) t}\dot{\beta}^{(k)}_t.
\end{align*}
Based on this expression, and since {\cop (recall that $\frac12<H<\frac34$)}
$$\mathbb{E}\Big[\dot{\beta}^{(k)}_t \dot{\beta}^{(k')}_{t'}\Big]=\1_{\{k=k'\}} |t-t'|^{2H-2}{\cop =c\, \1_{\{k=k'\}} \int_{\R}\frac{d\xi}{|\xi|^{2H-1}} e^{-\imath \xi (t-t')} },$$
we get 
\begin{align}
&\mathbb{E}\Big[ \cf(\<Psi>^{(n)}_k)(\la)\overline{\cf(\<Psi>^{(n)}_{k'})(\la')}\Big]\nonumber\\
&=\1_{\{\langle k\rangle \leq 2^n\}}\1_{\{\langle k'\rangle \leq 2^n\}}\int d\la_1d\la_1'\, \ci_\chi(\la,\la_1)\overline{\ci_\chi(\la',\la_1')}\int dt dt' \, e^{-\imath (\la_1+|k|^2) t}e^{\imath (\la'_1+|k'|^2) t'}\mathbb{E}\Big[\dot{\beta}^{(k)}_t \dot{\beta}^{(k')}_{t'}\Big]\nonumber\\
&=c\, \1_{\{k=k'\}}\1_{\{\langle k\rangle \leq 2^n\}}\int_{\R} \frac{d\xi}{|\xi|^{2H-1}}\int d\la_1d\la_1'\, \ci_\chi(\la,\la_1)\overline{\ci_\chi(\la',\la_1')}\int dt dt' \, e^{-\imath (\la_1+|k|^2+\xi) t}e^{\imath (\la'_1+|k|^2+\xi) t'}\nonumber\\
&=c\, \1_{\{k=k'\}}\1_{\{\langle k\rangle \leq 2^n\}}\int_{\R} \frac{d\xi}{|\xi|^{2H-1}} \ci_\chi(\la,-\xi-k^2)\overline{\ci_\chi(\la',-\xi-k^2)}.\label{cova-frac}
\end{align}

\

For a more detailed expression of this quantity, let us introduce the following notation.
\begin{notation}\label{notation:xi}
For all $s,t,\la\in \R$, we set
$$\Xi_t(\la):=\int_{\R} dr\, e^{\imath \la r}\chi(r)\1_{[0,t]}(r)=\int_0^t dr\, e^{\imath \la r}\chi(r),$$
\begin{equation}\label{xi-up-t}
\Xi^t(\la):=\int_{\R} dr\, e^{\imath \la r}\chi(r)\1_{[0,r]}(t)= \1_{\{t\geq 0\}} \int_t^\infty dr\, e^{\imath \la r}\chi(r)-\1_{\{t<0\}} \int_{-\infty}^t dr\, e^{\imath \la r}\chi(r).\end{equation}
\end{notation}

\

With this notation, recall that
\begin{equation}\label{ci-xi}
\ci_\chi(\la,\la')=\int_{\R} dt\, e^{-\imath \la t}\chi(t)\int_{\R} ds \, e^{\imath \la' s} \chi(s)\1_{[0,t]}(s)=\int_{\R} dt\, e^{-\imath \la t}\chi(t)\, \Xi_t(\la'),
\end{equation}
and so we can rephrase \eqref{cova-frac} as
\begin{align}
&\mathbb{E}\Big[ \cf(\<Psi>^{(n)}_k)(\la)\overline{\cf(\<Psi>^{(n)}_{k'})(\la')}\Big]\nonumber\\
&=c\, \1_{\{k=k'\}}\1_{\{\langle k\rangle \leq 2^n\}}\int dt dt' \, \chi(t) \chi(t')e^{-\imath \la t}e^{\imath \la' t'}\int_{\R} \frac{d\xi}{|\xi|^{2H-1}}  \Xi_t(-\xi-k^2) \overline{\Xi_{t'}(-\xi-k^2)}.\label{cova-refin}
\end{align}

\

Besides, based on \eqref{ci-xi}, it is readily checked {\cop (see e.g. \cite[Lemma 4.1]{DNY})} that
\begin{equation}\label{ker-i}
\big|\ci_{{\cop \chi}}(\la,\la')\big|\lesssim \frac{1}{\langle \la\rangle \langle \la-\la'\rangle},
\end{equation}
and therefore the expression in \eqref{cova-frac} leads us to the following uniform estimate (with respect to $n$): for all $k,k'\in Z$ and $\la,\la'\in \R$,
\begin{align}
\bigg|\mathbb{E}\Big[ \cf(\<Psi>^{(n)}_k)(\la)\overline{\cf(\<Psi>^{(n)}_{k'})(\la')}\Big]\bigg|&\lesssim \1_{\{k=k'\}}\frac{1}{\langle \la\rangle \langle \la'\rangle}\int_{\R} \frac{d\xi}{|\xi|^{2H-1}}\frac{1}{\langle \xi+\la+k^2\rangle}\frac{1}{\langle \xi+\la'+k^2\rangle} \nonumber\\
&\lesssim \1_{\{k=k'\}}\frac{1}{\langle \la\rangle \langle \la'\rangle}\frac{1}{\langle \la-\la'\rangle^{1-\varepsilon}} \bigg[\frac{1}{\langle \la+k^2\rangle^{2H-1}}+\frac{1}{\langle \la'+k^2\rangle^{2H-1}}\bigg],\label{cova-frac-estim}
\end{align}
where the second inequality follows from \color{blue} Lemma \ref{lem:int-singu-zero} below.\color{black}

\

With similar arguments, we obtain that
\begin{align}
\mathbb{E}\Big[ \cf(\<Psi>^{(n)}_k)(\la)\cf(\<Psi>^{(n)}_{k'})(\la')\Big]
&=c\, \1_{\{k=k'\}}\1_{\{\langle k\rangle \leq 2^n\}}\int_{\R} \frac{d\xi}{|\xi|^{2H-1}} \ci_\chi(\la,-\xi-k^2)\ci_\chi(\la',\xi-k^2),\label{cova-frac-non-conj}
\end{align}
and then, for all $k,k'\in Z$ and $\la,\la'\in \R$,
\begin{align}
\bigg|\mathbb{E}\Big[ \cf(\<Psi>^{(n)}_k)(\la)\cf(\<Psi>^{(n)}_{k'})(\la')\Big]\bigg|&\lesssim \1_{\{k=k'\}}\frac{1}{\langle \la\rangle \langle \la'\rangle}\int_{\R} \frac{d\xi}{|\xi|^{2H-1}}\frac{1}{\langle \xi+\la+k^2\rangle}\frac{1}{\langle \xi-\la'-k^2\rangle} \nonumber\\
&\lesssim \1_{\{k=k'\}}\frac{1}{\langle \la\rangle \langle \la'\rangle}\frac{1}{\langle \la+\la'+2k^2\rangle^{1-\varepsilon}} \bigg[\frac{1}{\langle \la+k^2\rangle^{2H-1}}+\frac{1}{\langle \la'+k^2\rangle^{2H-1}}\bigg],\label{cova-frac-estim-non-conj}
\end{align}
{\cop where we have used again Lemma \ref{lem:int-singu-zero} to derive the second inequality}.

\

We are now in a position to tackle the proof of our main asymptotic result.

\subsection{Proof of Proposition \ref{prop:ice-cream-frac-intro}}

\

\smallskip

First, using the integration kernel $\ci_\chi(.,.)$ (see \eqref{i-chi-fou}) and the Fourier expression of $\widetilde{\cm}$ (see \eqref{cm-fourier}), we can write the Fourier transform of the process under consideration as
\begin{align*}
&\cf\big(\ci_\chi\widetilde{\cm}( \<Psi>^{(n)},\<Psi>^{(n)})_k\big)(\la)=\int_{\R} d\la' \, \ci_\chi(\la,\la') \cf\big(\widetilde{\cm}(\<Psi>^{(n)},\<Psi>^{(n)})_k\big)(\la')\\
&=\1_{\{k\neq 0\}}\sum_{k_1\neq 0} \int_{\R} d\la' \, \ci_\chi(\la,\la') \int dt\, e^{-\imath \la' t}e^{\imath t \Omega_{k,k_1}} \<Psi>^{(n)}_{k+k_1}(t) \overline{\<Psi>^{(n)}_{k_1}(t)}\\
&=\1_{\{k\neq 0\}}\sum_{k_1\neq 0}\int d\la_1 \, \overline{\cf(\<Psi>^{(n)}_{k_1})(\la_1)} \int d\la_2 \, \cf(\<Psi>^{(n)}_{k+k_1})(\la_2) \int_{\R} d\la' \, \ci_\chi(\la,\la')\int dt\, e^{-\imath \la' t}e^{\imath t\Omega_{k,k_1}} e^{\imath t\la_2}e^{-\imath t \la_1},
\end{align*}
and thus
\begin{equation*}
\cf\big(\ci_\chi\big(\widetilde{\cm}( \<Psi>^{(n)},\<Psi>^{(n)})_k\big)\big)(\la)=\1_{\{k\neq 0\}}\sum_{k_1\neq 0}\int d\la_1d\la_2 \, \ci_\chi(\la,\Omega_{k,k_1}+\la_2-\la_1) \overline{\cf(\<Psi>^{(n)}_{k_1})(\la_1)} \cf(\<Psi>^{(n)}_{k+k_1})(\la_2).
\end{equation*}

Based on this expression, and by applying Wick's formula, we can compute
\begin{align}
&\mathbb{E}\Big[\big| \cf \big(\ci_\chi \big(\widetilde{\cm}( \<Psi>^{(n)},\<Psi>^{(n)})_k\big)(\la)\big|^2\Big]\nonumber\\
&=\1_{\{k\neq 0\}}\sum_{k_1\neq 0}\int d\la_1d\la_2 \sum_{k'_1\neq 0}\int d\la'_1d\la'_2  \ \ci_\chi(\la,\Omega_{k,k_1}+\la_2-\la_1) \overline{\ci_\chi(\la,\Omega_{k,k'_1}+\la'_2-\la'_1)}\nonumber \\
&\hspace{3cm}\mathbb{E}\Big[\overline{\cf(\<Psi>^{(n)}_{k_1})(\la_1)} \cf(\<Psi>^{(n)}_{k+k_1})(\la_2)\cf(\<Psi>^{(n)}_{k'_1})(\la'_1) \overline{\cf(\<Psi>^{(n)}_{k+k_1'})(\la'_2)}\Big]\nonumber\\
&=\Big|\mathbb{E}\big[\cf \big(\ci_\chi \big(\widetilde{\cm}( \<Psi>^{(n)},\<Psi>^{(n)})_k\big)(\la)\big]\Big|^2\nonumber\\
&\hspace{0.5cm}+\1_{\{k\neq 0\}}\sum_{k_1\neq 0}\int d\la_1d\la_2 \sum_{k'_1\neq 0}\int d\la'_1d\la'_2  \ \ci_\chi(\la,\Omega_{k,k_1}+\la_2-\la_1) \overline{\ci_\chi(\la,\Omega_{k,k'_1}+\la'_2-\la'_1)}\nonumber \\
&\hspace{3.5cm}\mathbb{E}\Big[\overline{\cf(\<Psi>^{(n)}_{k_1})(\la_1)} \cf(\<Psi>^{(n)}_{k'_1})(\la'_1) \Big]\mathbb{E}\Big[ \cf(\<Psi>^{(n)}_{k+k_1})(\la_2)\overline{\cf(\<Psi>^{(n)}_{k+k_1'})(\la'_2)}\Big]\nonumber\\
&\hspace{0.5cm}+\1_{\{k\neq 0\}}\sum_{k_1\neq 0}\int d\la_1d\la_2 \sum_{k'_1\neq 0}\int d\la'_1d\la'_2  \ \ci_\chi(\la,\Omega_{k,k_1}+\la_2-\la_1) \overline{\ci_\chi(\la,\Omega_{k,k'_1}+\la'_2-\la'_1)}\nonumber \\
&\hspace{3.5cm}\mathbb{E}\Big[\overline{\cf(\<Psi>^{(n)}_{k_1})(\la_1)} \overline{\cf(\<Psi>^{(n)}_{k+k_1'})(\la'_2)}\Big]\mathbb{E}\Big[\cf(\<Psi>^{(n)}_{k+k_1})(\la_2)\cf(\<Psi>^{(n)}_{k'_1})(\la'_1) \Big].\label{expa-ice-cream}
\end{align}

\

As far as the last term of this expansion is concerned, observe that according to \eqref{cova-frac-non-conj}, one has
\begin{align*}
&\1_{\{k\neq 0\}}\mathbb{E}\Big[\overline{\cf(\<Psi>^{(n)}_{k_1})(\la_1)} \overline{\cf(\<Psi>^{(n)}_{k+k_1'})(\la'_2)}\Big]\mathbb{E}\Big[\cf(\<Psi>^{(n)}_{k+k_1})(\la_2)\cf(\<Psi>^{(n)}_{k'_1})(\la'_1) \Big]\\
&=\1_{\{k\neq 0\}}\1_{\{k_1=k+k_1'\}} \1_{\{k+k_1=k_1'\}}  \mathbb{E}\Big[\overline{\cf(\<Psi>^{(n)}_{k_1})(\la_1)} \overline{\cf(\<Psi>^{(n)}_{k+k_1'})(\la'_2)}\Big]\mathbb{E}\Big[\cf(\<Psi>^{(n)}_{k+k_1})(\la_2)\cf(\<Psi>^{(n)}_{k'_1})(\la'_1) \Big]\\
&=\1_{\{k\neq 0\}}\1_{\{k_1=k+k_1'\}} \1_{\{2k+k_1'=k_1'\}}  \mathbb{E}\Big[\overline{\cf(\<Psi>^{(n)}_{k_1})(\la_1)} \overline{\cf(\<Psi>^{(n)}_{k+k_1'})(\la'_2)}\Big]\mathbb{E}\Big[\cf(\<Psi>^{(n)}_{k+k_1})(\la_2)\cf(\<Psi>^{(n)}_{k'_1})(\la'_1) \Big]\ =0.
\end{align*}
For the same reason,
\begin{align*}
&\mathbb{E}\Big[\overline{\cf(\<Psi>^{(n)}_{k_1})(\la_1)} \cf(\<Psi>^{(n)}_{k'_1})(\la'_1) \Big]\mathbb{E}\Big[ \cf(\<Psi>^{(n)}_{k+k_1})(\la_2)\overline{\cf(\<Psi>^{(n)}_{k+k_1'})(\la'_2)}\Big]\\
&=\1_{\{k_1=k_1'\}}\mathbb{E}\Big[\overline{\cf(\<Psi>^{(n)}_{k_1})(\la_1)} \cf(\<Psi>^{(n)}_{k_1})(\la'_1) \Big]\mathbb{E}\Big[ \cf(\<Psi>^{(n)}_{k+k_1})(\la_2)\overline{\cf(\<Psi>^{(n)}_{k+k_1})(\la'_2)}\Big].
\end{align*}

Thus, going back to \eqref{expa-ice-cream} and recalling the definition \eqref{def:sq-tree-renorm} of $\widetilde{\<IPsi2>}^{(n)}$, we deduce that
\begin{align}
&\mathbb{E}\Big[\big| \cf \big(\widetilde{\<IPsi2>}^{(n)}_k\big)(\la)\big|^2 \Big]=\mathbb{E}\Big[\big| \cf \big(\ci_\chi \big(\widetilde{\cm}( \<Psi>^{(n)},\<Psi>^{(n)})_k\big)(\la)\big|^2\Big]-\Big|\mathbb{E}\big[\cf \big(\ci_\chi \big(\widetilde{\cm}( \<Psi>^{(n)},\<Psi>^{(n)})_k\big)(\la)\big]\Big|^2\nonumber\\
&=\1_{\{k\neq 0\}}\sum_{k_1\neq 0}\int d\la_1d\la_2 \int d\la'_1d\la'_2  \ \ci_\chi(\la,\Omega_{k,k_1}+\la_2-\la_1) \overline{\ci_\chi(\la,\Omega_{k,k_1}+\la'_2-\la'_1)} \nonumber\\
&\hspace{5.5cm}\mathbb{E}\Big[\overline{\cf(\<Psi>^{(n)}_{k_1})(\la_1)} \cf(\<Psi>^{(n)}_{k_1})(\la'_1) \Big]\mathbb{E}\Big[ \cf(\<Psi>^{(n)}_{k+k_1})(\la_2)\overline{\cf(\<Psi>^{(n)}_{k+k_1})(\la'_2)}\Big].\label{cova-ice-cream}
\end{align}

\

\subsubsection{\cop Proof of Proposition \ref{prop:ice-cream-frac-intro}, item $(i)$}

\

\smallskip

Using the estimates \eqref{ker-i} and \eqref{cova-frac-estim}, we get that for all $0\leq c<\frac12$ and $\varepsilon>0$ small enough,
\begin{align*}
&\mathbb{E}\Big[\big\|\widetilde{\<IPsi2>}^{(n)}\big\|^2_{\cop Z^{c,0}}\Big]\\
&\lesssim \sum_{k\neq 0}\sum_{k_1\neq 0} \langle k\rangle^{2c}\int \frac{d\la}{\langle \la\rangle^2} \int \frac{d\la_1}{\langle \la_1\rangle}\frac{d\la_2}{\langle \la_2\rangle}  \frac{1}{\langle \la-\Omega_{k,k_1}-\la_2+\la_1\rangle} \int \frac{d\la_1'}{\langle \la_1'\rangle} \frac{d\la_2'}{\langle \la_2'\rangle}\frac{1}{\langle \la-\Omega_{k,k_1}-\la_2'+\la_1'\rangle}\\
&\lesssim \sum_{k\neq 0}\sum_{k_1\neq 0} \langle k\rangle^{2c}\int \frac{d\la}{\langle \la\rangle^2} \bigg|\int \frac{d\la_1}{\langle \la_1\rangle}\int \frac{d\la_2}{\langle \la_2\rangle}  \frac{1}{\langle \la-\Omega_{k,k_1}-\la_2+\la_1\rangle}\bigg|^2.
\end{align*}
{\cop By applying Lemma \ref{lem:gtv} below, we deduce}
\begin{align*}
&\mathbb{E}\Big[\big\|\widetilde{\<IPsi2>}^{(n)}\big\|^2_{\cop Z^{c,0}}\Big]\\
&\lesssim \sum_{k\neq 0} \sum_{k_1\neq 0}\langle k\rangle^{2c}\int \frac{d\la}{\langle \la\rangle^2} \bigg|\int \frac{d\la_1}{\langle \la_1\rangle}  \frac{1}{\langle \la-\Omega_{k,k_1}+\la_1\rangle^{1-\frac{\varepsilon}{2}}}\bigg|^2 \\
&\lesssim \sum_{k\neq 0} \sum_{k_1\neq 0}\langle k\rangle^{2c}\int \frac{d\la}{\langle \la\rangle^2} \frac{1}{\langle \la-\Omega_{k,k_1}\rangle^{2-2\varepsilon}} \lesssim \sum_{k\neq 0} \langle k\rangle^{2c} \sum_{k_1\neq 0} \frac{1}{\langle \Omega_{k,k_1}\rangle^{2-2\varepsilon}}\lesssim \sum_k \frac{1}{\langle k\rangle^{2-2c-2\varepsilon}}\sum_{k_1}\frac{1}{\langle k_1\rangle^{2-2\varepsilon}} <\infty,
\end{align*}
which already proves item $(i)$.

\

Let us now focus on the norm in ${\cop Z^{0,b}}$, which, thanks to \eqref{cova-ice-cream}, can be expanded as
\begin{align*}
&\mathbb{E}\Big[\big\|\widetilde{\<IPsi2>}^{(n)}\big\|^2_{\cop Z^{0,b}} \Big]=\sum_{k\neq 0}\int d\la \, \langle \la \rangle^{2b}\sum_{k_1\neq 0}\int d\la_1d\la_2 \int d\la'_1d\la'_2  \ \ci_\chi(\la,\Omega_{k,k_1}+\la_2-\la_1) \overline{\ci_\chi(\la,\Omega_{k,k_1}+\la'_2-\la'_1)} \\
&\hspace{6.5cm}\mathbb{E}\Big[\overline{\cf(\<Psi>^{(n)}_{k_1})(\la_1)} \cf(\<Psi>^{(n)}_{k_1})(\la'_1) \Big]\mathbb{E}\Big[ \cf(\<Psi>^{(n)}_{k+k_1})(\la_2)\overline{\cf(\<Psi>^{(n)}_{k+k_1})(\la'_2)}\Big].
\end{align*}

\

\subsubsection{\cop Proof of Proposition \ref{prop:ice-cream-frac-intro}, item $(ii)$}\label{subsec:item-ii}

Assume here that $\frac12 <b < 2H-\frac12$.

\smallskip

Using the estimates \eqref{ker-i} and \eqref{cova-frac-estim}, we get that
\begin{align*}
&\mathbb{E}\Big[\big\|\widetilde{\<IPsi2>}^{(n)}\big\|^2_{\cop Z^{0,b}} \Big]\\
&\lesssim \sum_{k\neq 0}\int \frac{d\la}{\langle \la\rangle^{2-2b}} \sum_{k_1\neq 0}\int \frac{d\la_1}{\langle \la_1\rangle} \frac{d\la_2}{\langle \la_2\rangle} \int \frac{d\la'_1}{\langle \la_1'\rangle} \frac{d\la'_2}{\langle \la_2'\rangle}   \frac{1}{\langle \la-\Omega_{k,k_1}-\la_2+\la_1\rangle}  \frac{1}{\langle \la-\Omega_{k,k_1}-\la_2'+\la_1'\rangle} \\
&\hspace{1.5cm}\bigg[ \frac{1}{\langle \la_1+k_1^2\rangle^{2H-1}}+\frac{1}{\langle \la'_1+k_1^2\rangle^{2H-1}}\bigg] \bigg[ \frac{1}{\langle \la_2+(k+k_1)^2\rangle^{2H-1}}+\frac{1}{\langle \la'_2+(k+k_1)^2\rangle^{2H-1}}\bigg]\\
&\lesssim \mathbb{A}_b+\mathbb{B}_b,
\end{align*}
with
\begin{align*}
&\mathbb{A}_b:=\sum_{k\neq 0}\sum_{k_1\neq 0}\int \frac{d\la}{\langle \la\rangle^{2-2b}} \int \frac{d\la_1}{\langle \la_1\rangle} \frac{d\la_2}{\langle \la_2\rangle}    \frac{1}{\langle \la-\Omega_{k,k_1}-\la_2+\la_1\rangle}\frac{1}{\langle \la_1+k_1^2\rangle^{2H-1}} \frac{1}{\langle \la_2+(k+k_1)^2\rangle^{2H-1}} \\
&\hspace{4.5cm} \int \frac{d\la'_1}{\langle \la_1'\rangle} \frac{d\la'_2}{\langle \la_2'\rangle} \frac{1}{\langle \la-\Omega_{k,k_1}-\la_2'+\la_1'\rangle} \\
\end{align*}
and
\begin{align*}
&\mathbb{B}_b:=\sum_{k\neq 0}\sum_{k_1\neq 0}\int \frac{d\la}{\langle \la\rangle^{2-2b}} \int \frac{d\la_1}{\langle \la_1\rangle} \frac{d\la_2}{\langle \la_2\rangle} \frac{1}{\langle \la-\Omega_{k,k_1}-\la_2+\la_1\rangle} \frac{1}{\langle \la_1+k_1^2\rangle^{2H-1}}  \\
&\hspace{4.5cm}\int \frac{d\la'_1}{\langle \la_1'\rangle} \frac{d\la'_2}{\langle \la_2'\rangle}   \frac{1}{\langle \la-\Omega_{k,k_1}-\la_2'+\la_1'\rangle}\frac{1}{\langle \la'_2+(k+k_1)^2\rangle^{2H-1}}.
\end{align*}
Since $\mathbb{A}_b$ and $\mathbb{B}_b$ no longer depend on $n$, we only need to prove that these two quantities are finite.

\

For $\mathbb{A}_b$, note first that {\cop by Lemma \ref{lem:gtv},}
$$\int \frac{d\la'_1}{\langle \la_1'\rangle} \frac{d\la'_2}{\langle \la_2'\rangle} \frac{1}{\langle \la-\Omega_{k,k_1}-\la_2'+\la_1'\rangle}  \lesssim \int \frac{d\la'_1}{\langle \la_1'\rangle}  \frac{1}{\langle \la-\Omega_{k,k_1}+\la_1'\rangle^{1-\frac{\varepsilon}{2}}}\lesssim \frac{1}{\langle \la-\Omega_{k,k_1}\rangle^{1-\varepsilon}}.$$
On the other hand, 
\begin{align*}
&\int \frac{d\la_1}{\langle \la_1\rangle} \frac{d\la_2}{\langle \la_2\rangle}    \frac{1}{\langle \la-\Omega_{k,k_1}-\la_2+\la_1\rangle}\frac{1}{\langle \la_1+k_1^2\rangle^{2H-1}} \frac{1}{\langle \la_2+(k+k_1)^2\rangle^{2H-1}} \\
&\lesssim \int \frac{d\la_1}{\langle \la_1\rangle} \frac{1}{\langle \la_1+k_1^2\rangle^{2H-1}} \int \frac{d\la_2}{\langle \la_2\rangle} \frac{1}{\langle \la_2+(k+k_1)^2\rangle^{2H-1}}\lesssim \frac{1}{\langle k_1\rangle^{4H-2-\varepsilon}}\frac{1}{\langle k+ k_1\rangle^{4H-2-\varepsilon}},
\end{align*}
{\cop where we have used again Lemma \ref{lem:gtv} to derive the last inequality.} Thus, going back to the definition of $\mathbb{A}_b$, we obtain that for $\varepsilon >0$ small enough,
\begin{align}
\mathbb{A}_b&\lesssim \sum_{k_1\neq 0}\frac{1}{\langle k_1\rangle^{4H-2-\varepsilon}}\sum_{k\neq 0}\frac{1}{\langle k+ k_1\rangle^{4H-2-\varepsilon}}\int \frac{d\la}{\langle \la\rangle^{2-2b}} \frac{1}{\langle \la-\Omega_{k,k_1}\rangle^{1-\varepsilon}}\label{a-b}\\
&\lesssim \sum_{k_1\neq 0}\frac{1}{\langle k_1\rangle^{4H-2b-2\varepsilon}}\sum_{k\neq 0}\frac{1}{\langle k+ k_1\rangle^{4H-2-\varepsilon}}\frac{1}{\langle k\rangle^{2-2b-\varepsilon}},\nonumber
\end{align}
{\cop due to Lemma \ref{lem:gtv} and the fact that $b>\frac12$.} Since $b<2H-\frac12$, one has $4H-2b>1$, and so {\cop by Lemma \ref{lem:gtv} again},
\begin{align*}
\mathbb{A}_b &\lesssim \sum_{k_1\neq 0}\frac{1}{\langle k_1\rangle^{8H-4b-1-4\varepsilon}} <\infty,
\end{align*}
for $\varepsilon >0$ small enough.

\

As far as $\mathbb{B}_b$ is concerned, one has {\cop by repeated applications of Lemma \ref{lem:gtv},}
\begin{align*}
&\int \frac{d\la_1}{\langle \la_1\rangle} \frac{d\la_2}{\langle \la_2\rangle} \frac{1}{\langle \la-\Omega_{k,k_1}-\la_2+\la_1\rangle}\frac{1}{\langle \la_1+k_1^2\rangle^{2H-1}} \\
&\lesssim \int \frac{d\la_1}{\langle \la_1\rangle} \frac{1}{\langle \la_1+k_1^2\rangle^{2H-1}} \frac{1}{\langle \la-\Omega_{k,k_1}+\la_1\rangle^{1-\varepsilon}}\\
&\lesssim \bigg( \int \frac{d\la_1}{\langle \la_1\rangle} \frac{1}{\langle \la_1+k_1^2\rangle^{4H-2}} \bigg)^{\frac12} \bigg( \int \frac{d\la_1}{\langle \la_1\rangle}  \frac{1}{\langle \la-\Omega_{k,k_1}+\la_1\rangle^{2-2\varepsilon}}\bigg)^{\frac12}\lesssim \frac{1}{\langle k_1\rangle^{4H-2-\varepsilon}}\frac{1}{\langle \la-\Omega_{k,k_1}\rangle^{\frac12}}
\end{align*}
and in the same way
\begin{align*}
&\int \frac{d\la'_1}{\langle \la_1'\rangle} \frac{d\la'_2}{\langle \la_2'\rangle}   \frac{1}{\langle \la-\Omega_{k,k_1}-\la_2'+\la_1'\rangle}\frac{1}{\langle \la'_2+(k+k_1)^2\rangle^{2H-1}}\lesssim \frac{1}{\langle k+k_1\rangle^{4H-2-\varepsilon}}\frac{1}{\langle \la-\Omega_{k,k_1}\rangle^{\frac12}},
\end{align*}
which gives
\begin{align*}
\mathbb{B}_b&\lesssim \sum_{k_1\neq 0}\frac{1}{\langle k_1\rangle^{4H-2-\varepsilon}}\sum_{k\neq 0}\frac{1}{\langle k+ k_1\rangle^{4H-2-\varepsilon}}\int \frac{d\la}{\langle \la\rangle^{2-2b}} \frac{1}{\langle \la-\Omega_{k,k_1}\rangle}.
\end{align*}
Thus, we are in the same position as in \eqref{a-b}, and we can use the same arguments to assert that $\mathbb{B}_b<\infty$.

\smallskip

This achieves to prove that for every $b$ such that $\frac12<b<2H-\frac12$, one has
$$\sup_{n\geq 1} \mathbb{E}\Big[\Big\|\widetilde{\<IPsi2>}^{(n)}\Big\|^2_{\cop Z^{0,b}}\Big] <\infty,$$
as desired.

\

\

\subsubsection{\cop Proof of Proposition \ref{prop:ice-cream-frac-intro}, item $(iii)$} Assume now that {\cop $b+\frac{c}{2} = 2H-\frac12$}.

\smallskip

Using the expression in \eqref{cova-refin}, we can write

\begin{align}
&\mathbb{E}\Big[\overline{\cf(\<Psi>^{(n)}_{k_1})(\la_1)} \cf(\<Psi>^{(n)}_{k_1})(\la'_1) \Big]\mathbb{E}\Big[ \cf(\<Psi>^{(n)}_{k+k_1})(\la_2)\overline{\cf(\<Psi>^{(n)}_{k+k_1})(\la'_2)}\Big]\nonumber\\
&=\1_{\{\langle k_1\rangle\leq 2^n\}}\int_{\R} \frac{d\xi}{|\xi|^{2H-1}} \int dt dt' \, \chi(t) \chi(t')e^{\imath \la_1 t}e^{-\imath \la_1' t'} \overline{\Xi_t(-\xi-k_1^2)} \Xi_{t'}(-\xi-k_1^2)\nonumber\\
&\hspace{0.5cm}\1_{\{\langle k+k_1\rangle\leq 2^n\}}\int_{\R} \frac{d\eta}{|\eta|^{2H-1}} \int ds ds' \, \chi(s) \chi(s')e^{-\imath \la_2 s}e^{\imath \la_2' s'} \Xi_s(-\eta-(k+k_1)^2) \overline{\Xi_{s'}(-\eta-(k+k_1)^2)}.\label{pro-1}
\end{align}
Then observe that by \eqref{ci-xi},
\begin{align}
&\int d\la_1d\la_2\,  \ci_\chi(\la,\Omega_{k,k_1}+\la_2-\la_1) e^{\imath \la_1 t}e^{-\imath \la_2 s}\nonumber\\
&=\int_{\R} dv\, e^{-\imath \la v}\chi(v) \int dr \, \1_{[0,v]}(r)\chi(r) \int d\la_1d\la_2\, e^{\imath r (\Omega_{k,k_1}+\la_2-\la_1)}e^{\imath \la_1 t}e^{-\imath \la_2 s}\nonumber\\
&=\int_{\R} dv\, e^{-\imath \la v}\chi(v)  \int dr \, \1_{[0,v]}(r)\chi(r)e^{\imath r \Omega_{k,k_1}}\bigg(\int d\la_1\, e^{-\imath r \la_1}e^{\imath \la_1 t}\bigg)\bigg(\int d\la_2\, e^{\imath r \la_2}e^{-\imath \la_2 s}\bigg)\nonumber\\
&=\delta_{\{s=t\}}\chi(t)e^{\imath t \Omega_{k,k_1}}\int_{\R} dv\, e^{-\imath \la v}\chi(v)  \1_{[0,v]}(t),\label{pro-2}
\end{align}
and in a similar way
\begin{align}
 \int d\la'_1d\la'_2  \,\overline{\ci_\chi(\la,\Omega_{k,k_1}+\la'_2-\la'_1)}e^{-\imath \la_1' t'}e^{\imath \la_2' s'}
&=\delta_{\{s'=t'\}}\chi(t')e^{-\imath t' \Omega_{k,k_1}}\int_{\R} dw\, e^{\imath \la w}\chi(w)  \1_{[0,w]}(t').\label{pro-3}
\end{align}

\

Combining \eqref{pro-1}, \eqref{pro-2} and \eqref{pro-3}, we obtain that
\begin{align*}
&\int d\la_1d\la_2 \int d\la'_1d\la'_2  \ \ci_\chi(\la,\Omega_{k,k_1}+\la_2-\la_1) \overline{\ci_\chi(\la,\Omega_{k,k_1}+\la'_2-\la'_1)} \\
&\hspace{4.5cm}\mathbb{E}\Big[\overline{\cf(\<Psi>^{(n)}_{k_1})(\la_1)} \cf(\<Psi>^{(n)}_{k_1})(\la'_1) \Big]\mathbb{E}\Big[ \cf(\<Psi>^{(n)}_{k+k_1})(\la_2)\overline{\cf(\<Psi>^{(n)}_{k+k_1})(\la'_2)}\Big]\\
&={\cop \1_{\{\langle k_1\rangle\leq 2^n\}} \1_{\{\langle k+k_1\rangle\leq 2^n\}}}\int_{\R} \frac{d\xi}{|\xi|^{2H-1}} \int_{\R} \frac{d\eta}{|\eta|^{2H-1}}\\
& \hspace{1cm}\bigg(\int_{\R} dv\, e^{-\imath \la v}\chi(v) \int dt  \, \1_{[0,v]}(t)e^{\imath t \Omega_{k,k_1}}\chi(t)^3 \overline{\Xi_t(-\xi-k_1^2)} \Xi_t(-\eta-(k+k_1)^2) \bigg)\\
&\hspace{2cm}\bigg(\int_{\R} dw\, e^{\imath \la w}\chi(w)\int  dt' \,  \1_{[0,w]}(t')e^{-\imath t' \Omega_{k,k_1}} \chi(t')^3\Xi_{t'}(-\xi-k_1^2) \overline{\Xi_{t'}(-\eta-(k+k_1)^2)} \bigg)
\end{align*}
and so, going back to \eqref{cova-ice-cream}, we get the expression
\begin{equation}\label{expression-q}
\mathbb{E}\Big[\big\|\widetilde{\<IPsi2>}^{(n)}\big\|^2_{\cop Z^{c,b}} \Big]=\sum_{k\neq 0}{\cop \langle k\rangle^{2c}} \sum_{k_1\neq 0} \1_{\{\langle k_1\rangle\leq 2^n\}} \1_{\{\langle k+k_1\rangle\leq 2^n\}}\int \frac{d\xi}{|\xi|^{2H-1}} \int \frac{d\eta}{|\eta|^{2H-1}} \int d\la \, \langle \la \rangle^{2b} \big|Q^{\Omega_{k,k_1}}_{\xi+k_1^2,\eta+(k+k_1)^2}(\la)\big|^2, 
\end{equation}
where we have set
\begin{align}
Q^L_{\beta,\beta'}(\la):=\int dt\, e^{-\imath \la t} \chi(t)\int_0^t ds \, e^{\imath s L} \chi(s)^3 \overline{\Xi_s(-\beta)}\Xi_s(-\beta').\label{defi:qlbeta}
\end{align}
In particular, 
\begin{align*}
&\mathbb{E}\Big[\big\|\widetilde{\<IPsi2>}^{(n)}\big\|^2_{\cop Z^{c,b}} \Big]
\geq \sum_{k\geq 1} {\cop \langle k\rangle^{2c}}\sum_{k_1\geq 1} \1_{\{\langle k_1\rangle\leq 2^n\}} \1_{\{\langle k+k_1\rangle\leq 2^n\}} \int \frac{d\xi}{|\xi|^{2H-1}} \int \frac{d\eta}{|\eta|^{2H-1}} \int d\la \, | \la |^{2b} \big|Q^{\Omega_{k,k_1}}_{\xi+k_1^2,\eta+(k+k_1)^2}(\la)\big|^2 \\
&\geq \sum_{k\geq 1}{\cop \langle k\rangle^{2c}} \sum_{k_1\geq 1}\1_{\{\langle k_1\rangle\leq 2^n\}} \1_{\{\langle k+k_1\rangle\leq 2^n\}}{\cop |\Omega_{k,k_1}|^{1+2b}}  \int \frac{d\xi}{|\xi|^{2H-1}} \int \frac{d\eta}{|\eta|^{2H-1}} \int d\la \, | \la |^{2b} \big|Q^{\Omega_{k,k_1}}_{\xi+k_1^2,\eta+(k+k_1)^2}(\la \Omega_{k,k_1})\big|^2 .
\end{align*}
For all $L\geq 1$, we can use an integration-by-parts argument to decompose $Q^{L}_{\beta,\beta'}(\la L)$ as
\begin{align}
Q^{L}_{\beta,\beta'}(\la L)&=\int dt\, e^{-\imath \la Lt} \chi(t)\int_0^t ds \, e^{\imath s L} \chi(s)^3 \overline{\Xi_s(-\beta)}\Xi_s(-\beta')=M^{L}_{\beta,\beta'}(\la )+R^{L}_{\beta,\beta'}(\la ),\label{decompoqlbeta}
\end{align}
where
\begin{equation}\label{defin-m-l}
M^{L}_{\beta,\beta'}(\la ):=\frac{1}{\imath \la L}\int dt\, e^{-\imath Lt(\la-1)} \chi(t)^4\,  \overline{\Xi_t(-\beta)}\Xi_t(-\beta')
\end{equation}
and
\begin{equation}\label{defin-r-l}
R^{L}_{\beta,\beta'}(\la ):=\frac{1}{\imath \la L}\int dt\, e^{-\imath \la Lt} \chi'(t)\int_0^t ds \, e^{\imath s L} \chi(s)^3 \overline{\Xi_s(-\beta)}\Xi_s(-\beta').
\end{equation}

Using this decomposition, we can write
\small
\begin{align*}
&\mathbb{E}\Big[\big\|\widetilde{\<IPsi2>}^{(n)}\big\|^2_{\cop Z^{c,b}} \Big]\geq \sum_{k\geq 1}{\cop \langle k\rangle^{2c}} \sum_{k_1\geq 1}\1_{\{\langle k_1\rangle\leq 2^n\}} \1_{\{\langle k+k_1\rangle\leq 2^n\}}{\cop |\Omega_{k,k_1}|^{1+2b}}  \int \frac{d\xi}{|\xi|^{2H-1}} \int \frac{d\eta}{|\eta|^{2H-1}} \int d\la \, | \la |^{2b} \big|M^{\Omega_{k,k_1}}_{\xi+k_1^2,\eta+(k+k_1)^2}(\la )\big|^2 \\
&\hspace{2cm}+\sum_{k\geq 1}{\cop \langle k\rangle^{2c}} \sum_{k_1\geq 1}\1_{\{\langle k_1\rangle\leq 2^n\}} \1_{\{\langle k+k_1\rangle\leq 2^n\}}{\cop |\Omega_{k,k_1}|^{1+2b}}  \int \frac{d\xi}{|\xi|^{2H-1}} \int \frac{d\eta}{|\eta|^{2H-1}} \int d\la \, | \la |^{2b} \big|R^{\Omega_{k,k_1}}_{\xi+k_1^2,\eta+(k+k_1)^2}(\la )\big|^2\\
&\hspace{1.5cm}-2\sum_{k\geq 1}{\cop \langle k\rangle^{2c}} \sum_{k_1\geq 1}{\cop |\Omega_{k,k_1}|^{1+2b}}  \int \frac{d\xi}{|\xi|^{2H-1}} \int \frac{d\eta}{|\eta|^{2H-1}} \int d\la \, | \la |^{2b} \big|M^{\Omega_{k,k_1}}_{\xi+k_1^2,\eta+(k+k_1)^2}(\la )\big| \big|R^{\Omega_{k,k_1}}_{\xi+k_1^2,\eta+(k+k_1)^2}(\la)\big|\\
&\geq \sum_{k\geq 1}{\cop \langle k\rangle^{2c}} \sum_{k_1\geq 1}\1_{\{\langle k_1\rangle\leq 2^n\}} \1_{\{\langle k+k_1\rangle\leq 2^n\}}{\cop |\Omega_{k,k_1}|^{1+2b}}  \int \frac{d\xi}{|\xi|^{2H-1}} \int \frac{d\eta}{|\eta|^{2H-1}} \int d\la \, | \la |^{2b} \big|M^{\Omega_{k,k_1}}_{\xi+k_1^2,\eta+(k+k_1)^2}(\la )\big|^2 \\
&\hspace{0.5cm}-2\sum_{k\geq 1}{\cop \langle k\rangle^{2c}} \sum_{k_1\geq 1}{\cop |\Omega_{k,k_1}|^{1+2b}}  \int \frac{d\xi}{|\xi|^{2H-1}} \int \frac{d\eta}{|\eta|^{2H-1}} \int d\la \, | \la |^{2b} \big|M^{\Omega_{k,k_1}}_{\xi+k_1^2,\eta+(k+k_1)^2}(\la )\big| \big|R^{\Omega_{k,k_1}}_{\xi+k_1^2,\eta+(k+k_1)^2}(\la )\big|.
\end{align*}
\normalsize

\

By applying Lemma \ref{lem:estim-m-r} below, we get that for all $\Omega\geq 1$, $k,\ell\geq 1$ and $\varepsilon >0$ small enough,
\small
\begin{align*}
&\int \frac{d\xi}{|\xi|^{2H-1}} \int \frac{d\eta}{|\eta|^{2H-1}} \int d\la \, | \la |^{2b} \big|M^{\Omega}_{\xi+k,\eta+\ell}(\la )\big|\big|R^{\Omega}_{\xi+k,\eta+\ell}(\la )\big|\\
&\lesssim \frac{1}{\Omega^{4-\varepsilon}}\bigg[ \int_{|\la |\geq 2} \frac{d\la}{|\la |^{4-2b}} \int \frac{d\xi}{|\xi|^{2H-1}}  \int \frac{d\eta}{|\eta|^{2H-1}}\frac{1}{\langle \xi+k\rangle}\frac{1}{\langle \eta+\ell\rangle}\\
&\hspace{0.5cm}+\int_{|\la |\leq 2} \frac{d\la}{|\la |^{2-2b} |\la-1|^{1-\varepsilon}}\int \frac{d\xi}{|\xi|^{2H-1}} \int \frac{d\eta}{|\eta|^{2H-1}}\\
&\hspace{1cm}\bigg( \frac{1}{\langle \xi+k\rangle\langle \eta+\ell\rangle}+\frac{1}{\langle \xi+(k+\Omega)\rangle\langle \eta+\ell\rangle}+\frac{1}{\langle \xi+(k+\Omega)\rangle\langle \eta+(\ell-\Omega-\xi-k)\rangle}\\
&\hspace{1.5cm}+\frac{1}{\langle \eta+(\ell-\Omega)\rangle\langle \xi+k\rangle}\bigg)+\frac{1}{\langle \eta+(\ell-\Omega)\rangle\langle \xi+(k+\Omega-\eta-\ell)\rangle}\bigg) \bigg]\, .
\end{align*}
\normalsize

\color{blue}
By using the three controls contained in Lemma \ref{lem:elementar} below (as well as the fact that $2H-1<\frac12$), we easily deduce the bound: for all $k\geq 1$, $\Omega\geq 1$, $\ell \geq \Omega+1$ and $\varepsilon >0$ small enough,
\begin{align*}
&\int \frac{d\xi}{|\xi|^{2H-1}} \int \frac{d\eta}{|\eta|^{2H-1}} \int d\la \, | \la |^{2b} \big|M^{\Omega}_{\xi+k,\eta+\ell}(\la )\big|\big|R^{\Omega}_{\xi+k,\eta+\ell}(\la )\big|\lesssim \frac{1}{\Omega^{4-\varepsilon}}\bigg[ \frac{1}{\ell^{2H-1-\varepsilon}}+\frac{1}{(\ell-\Omega)^{2H-1-\varepsilon}} \bigg].
\end{align*}

Therefore, for $\varepsilon >0$ small enough,
\small
\begin{align*}
&\sum_{k\geq 1} \sum_{k_1\geq 1}\langle k\rangle^{2c}|\Omega_{k,k_1}|^{1+2b}  \int \frac{d\xi d\eta}{|\xi|^{2H-1}|\eta|^{2H-1}} \int d\la \, | \la |^{2b} \big|M^{\Omega_{k,k_1}}_{\xi+k_1^2,\eta+(k+k_1)^2}(\la )\big|\big|R^{\Omega_{k,k_1}}_{\xi+k_1^2,\eta+(k+k_1)^2}(\la )\big|\\
 &\lesssim \sum_{k\geq 1} \langle k\rangle^{2c}\sum_{k_1\geq 1}\frac{1}{|\Omega_{k,k_1}|^{3-2b-\varepsilon}}\bigg[ \frac{1}{(k+k_1)^{4H-2-2\varepsilon}}+\frac{1}{(k^2+k_1^2)^{2H-1-\varepsilon}}\bigg]\\
& \lesssim \sum_{k\geq 1} \frac{1}{\langle k\rangle^{1+4H-2b-2c-3\varepsilon}}\sum_{k_1\geq 1}\frac{1}{\langle k_1\rangle^{3-2b-\varepsilon}} < \infty,
\end{align*}
\normalsize
due to $b<1$ and $2b+2c=(2b+c)+c=4H-1+c<4H$.

\color{black}

\

We are thus left with the analysis of
\begin{equation}\label{defi:frak-m-n}
\mathfrak{M}^{(n)}:=\sum_{k\geq 1}{\cop \langle k\rangle^{2c}} \sum_{k_1\geq 1}\1_{\{\langle k_1\rangle\leq 2^n\}} \1_{\{\langle k+k_1\rangle\leq 2^n\}}{\cop |\Omega_{k,k_1}|^{1+2b}}  \int \frac{d\xi}{|\xi|^{2H-1}} \int \frac{d\eta}{|\eta|^{2H-1}} \int d\la \, | \la |^{2b} \big|M^{\Omega_{k,k_1}}_{\xi+k_1^2,\eta+(k+k_1)^2}(\la )\big|^2 .
\end{equation}

To this end, write for all $\Omega \geq 1$,
\begin{align*}
&\int \frac{d\xi}{|\xi|^{2H-1}} \int \frac{d\eta}{|\eta|^{2H-1}} \int d\la \, | \la |^{2b} \big|M^{\Omega}_{\xi+k_1^2,\eta+(k+k_1)^2}(\la )\big|^2\\
&=\frac{1}{|\Omega|^2}\int \frac{d\xi}{|\xi|^{2H-1}} \int \frac{d\eta}{|\eta|^{2H-1}} \int \frac{d\la}{|\la|^{2-2b}} \, \bigg|\int dt\, e^{-\imath \Omega(\la-1)t} \chi(t)^4\,  \overline{\Xi_t(-\xi-k_1^2)}\Xi_t(-\eta-(k+k_1)^2)\bigg|^2\\
&=\frac{1}{|\Omega|^2}  \int dt dt'\, \chi(t)^4 \chi(t')^4 \bigg(\int \frac{d\la}{|\la|^{2-2b}} e^{-\imath \la \Omega (t-t')}\bigg) \,  e^{\imath \Omega(t-t')}  \cj_{k_1}(t,t') \overline{\cj_{k+k_1}(t,t')},
\end{align*}
where we have set
$$\cj_{\ell}(t,t'):=\int \frac{d\xi}{|\xi|^{2H-1}}\overline{\Xi_t(-\xi-\ell^2)}\Xi_{t'}(-\xi-\ell^2).$$
Then, by elementary transformations,
\begin{align}
&\frac{1}{|\Omega|^2}  \int dt dt'\, \chi(t)^4 \chi(t')^4 \bigg(\int \frac{d\la}{|\la|^{2-2b}} e^{-\imath \la \Omega (t-t')}\bigg) \,  e^{\imath \Omega(t-t')}  \cj_{k_1}(t,t') \overline{\cj_{k+k_1}(t,t')}\nonumber\\
&=\frac{1}{|\Omega|^{1+2b}}  \int \frac{dt dt'}{ |t-t'|^{2b-1}}\, \chi(t)^4 \chi(t')^4  \,  e^{\imath \Omega(t-t')}  \cj_{k_1}(t,t') \overline{\cj_{k+k_1}(t,t')}\nonumber\\
&=\frac{1}{|\Omega|^{1+2b}}  \int \frac{dt ds}{ |s|^{2b-1}}\, \chi(t)^4 \chi(t-s)^4  \,  e^{\imath \Omega s}  \cj_{k_1}(t,t-s) \overline{\cj_{k+k_1}(t,t-s)}\nonumber\\
&=\frac{1}{|\Omega|^{3}}  \int \frac{dt ds}{ |s|^{2b-1}}\, \chi(t)^4 \chi\Big(t-\frac{s}{\Omega}\Big)^4  \,  e^{\imath s}  \cj_{k_1}\Big(t,t-\frac{s}{\Omega}\Big) \overline{\cj_{k+k_1}\Big(t,t-\frac{s}{\Omega}\Big)}.\label{im}
\end{align}
Let us now expand the two quantities $\cj_{k_1}\big(t,t-\frac{s}{\Omega}\big)$ and $\overline{\cj_{k+k_1}\big(t,t-\frac{s}{\Omega}\big)}$. On the one hand,
\begin{align*}
\cj_{k_1}\Big(t,t-\frac{s}{\Omega}\Big)&=\int \frac{d\xi}{|\xi|^{2H-1}}\overline{\Xi_t(-\xi-k_1^2)}\Xi_{t-\frac{s}{\Omega}}(-\xi-k_1^2)\\
&=\int \frac{d\xi}{|\xi|^{2H-1}} \bigg( \int_0^t dr \, \chi(r)e^{\imath (\xi+k_1^2)r}\bigg) \bigg( \int_0^{t-\frac{s}{\Omega}}dr' \, \chi(r')e^{-\imath (\xi+k_1^2)r'}\bigg)\\
&=\int \frac{d\xi}{|\xi|^{2H-1}}  \int dr' \, \1_{[0,t-\frac{s}{\Omega}]}(r')\chi(r') \int dr \, \1_{[0,t]}(r)\chi(r)e^{\imath (\xi+k_1^2)(r-r')}\\
&=\int \frac{d\xi}{|\xi|^{2H-1}}  \int dr' \, \1_{[0,t-\frac{s}{\Omega}]}(r')\chi(r') \int dr \, \1_{[0,t]}(r'+r)\chi(r'+r)e^{\imath (\xi+k_1^2)r}\\
&=\frac{1}{k_1^{4H-4}}\int \frac{d\xi'}{|\xi'|^{2H-1}}\int dr \, e^{\imath r k_1^2 (\xi'+1)}\int dr' \, \1_{[0,t-\frac{s}{\Omega}]}(r')\1_{[0,t]}(r'+r) \chi(r') \chi(r'+r)\\
&=\frac{1}{k_1^{4H-2}}\int \frac{d\xi'}{|\xi'|^{2H-1}} \int dr \, e^{\imath r (\xi'+1)}\int dr' \, \1_{[0,t-\frac{s}{\Omega}]}(r')\1_{[0,t]}\Big(r'+\frac{r}{k_1^2}\Big) \chi(r') \chi\Big(r'+\frac{r}{k_1^2}\Big)\\
&=\frac{1}{k_1^{4H-2}} \int \frac{dr}{|r|^{2-2H}} \, e^{\imath r}\int dr' \, \1_{[0,t-\frac{s}{\Omega}]}(r')\1_{[0,t]}\Big(r'+\frac{r}{k_1^2}\Big) \chi(r') \chi\Big(r'+\frac{r}{k_1^2}\Big),
\end{align*}
and similarly
\begin{align*}
&\overline{\cj_{k+k_1}\Big(t,t-\frac{s}{\Omega}\Big)}\\
&=\frac{1}{(k+k_1)^{4H-2}} \int \frac{du}{|u|^{2-2H}} \, e^{-\imath u}\int du' \, \1_{[0,t-\frac{s}{\Omega}]}(u')\1_{[0,t]}\Big(u'+\frac{u}{(k+k_1)^2}\Big) \chi(u') \chi\Big(u'+\frac{u}{(k+k_1)^2}\Big).
\end{align*}

\

By injecting these two expansions into \eqref{im}, we deduce the expression
\begin{align}
&\int \frac{d\xi}{|\xi|^{2H-1}} \int \frac{d\eta}{|\eta|^{2H-1}} \int d\la \, | \la |^{2b} \big|M^{\Omega}_{\xi+k_1^2,\eta+(k+k_1)^2}(\la )\big|^2\nonumber\\
&=\frac{1}{|\Omega|^{3}}\frac{1}{k_1^{4H-2}}\frac{1}{(k+k_1)^{4H-2}}  \int dt\, \chi(t)^4 \int \frac{ds}{ |s|^{2b-1}}\,  e^{\imath s}\chi\Big(t-\frac{s}{\Omega}\Big)^4  \nonumber   \\
&\hspace{2cm}\bigg(\int \frac{dr}{|r|^{2-2H}} \, e^{\imath r}\int dr' \, \1_{[0,t-\frac{s}{\Omega}]}(r')\1_{[0,t]}\Big(r'+\frac{r}{k_1^2}\Big) \chi(r') \chi\Big(r'+\frac{r}{k_1^2}\Big)\bigg)\nonumber\\
&\hspace{2.5cm}\bigg(\int \frac{du}{|u|^{2-2H}} \, e^{-\imath u}\int du' \, \1_{[0,t-\frac{s}{\Omega}]}(u')\1_{[0,t]}\Big(u'+\frac{u}{(k+k_1)^2}\Big) \chi(u') \chi\Big(u'+\frac{u}{(k+k_1)^2}\Big)\bigg)\nonumber\\
&=\frac{1}{|\Omega|^{3}}\frac{1}{k_1^{4H-2}}\frac{1}{(k+k_1)^{4H-2}} Q_{b,H}(\Omega,k,k_1), \label{defi:q-b-h}
\end{align}
where we have set
\begin{align*}
&Q_{b,H}(\Omega,k,k_1):=\int dt\int dr' \int du' \, \chi(t)^4\chi(r')\chi(u')  \\
&\hspace{0.5cm}\bigg( \int \frac{ds}{ |s|^{2b-1}}\,  e^{\imath s}\1_{[0,t-\frac{s}{\Omega}]}(r')\1_{[0,t-\frac{s}{\Omega}]}(u')\chi\Big(t-\frac{s}{\Omega}\Big)^4\bigg)  \bigg(\int \frac{dr}{|r|^{2-2H}} \, e^{\imath r} \1_{[0,t]}\Big(r'+\frac{r}{k_1^2}\Big) \chi\Big(r'+\frac{r}{k_1^2}\Big)\bigg)\\
&\hspace{2.5cm}\bigg(\int \frac{du}{|u|^{2-2H}} \, e^{-\imath u} \1_{[0,t]}\Big(u'+\frac{u}{(k+k_1)^2}\Big)  \chi\Big(u'+\frac{u}{(k+k_1)^2}\Big)\bigg).
\end{align*}

\

Using the subsequent Lemma \ref{lem:semi-conv}, we can easily check that 
$$Q_{b,H}(\Omega_{k,k_1},k,k_1) \to q_{b,H} \quad \text{as} \ k,k_1\to\infty,$$
where
$$q_{b,H}:=\bigg(\int \frac{ds}{ |s|^{2b-1}}\,  e^{\imath s}\bigg)\bigg|\int \frac{dr}{|r|^{2-2H}} \, e^{\imath r}\bigg|^2\bigg(\int dt\, \chi(t)^8\int_0^t dr'\, \chi(r')^2 \int_0^t du' \, \chi(u')^2\bigg) .$$

\

Since $q_{b,H}>0$, there exists $K\geq 1$ large enough such that for all $k_1\geq k\geq K$, one has 
$$Q_{b,H}(\Omega_{k,k_1},k,k_1)\geq \frac{q_{b,H}}{2}.$$
Therefore, going back to \eqref{defi:frak-m-n} and \eqref{defi:q-b-h}, we get that

\begin{align*}
\mathfrak{M}^{(n)}&\gtrsim \sum_{k\geq K}\frac{1}{|k|^{2-2b{\cop -2c}}} \sum_{k_1\geq k}\1_{\{\langle k_1\rangle \leq 2^n\}}\1_{\{\langle k+k_1\rangle \leq 2^n\}} \frac{1}{|k_1|^{4H-2b}}\frac{1}{|k+k_1|^{4H-2}}.
\end{align*}
and since $K$ does not depend on $n$, we deduce that
\begin{align*}
\liminf_{n\to\infty} \mathfrak{M}^{(n)}&\gtrsim \sum_{k\geq K}\frac{1}{|k|^{2-2b{\cop -2c}}} \sum_{k_1\geq k} \frac{1}{|k_1|^{4H-2b}}\frac{1}{|k+k_1|^{4H-2}}\\
&\gtrsim \sum_{k\geq K}\frac{1}{|k|^{8H-4b{\cop -2c}-1}}\bigg(\frac{1}{k} \sum_{k_1\geq k} \frac{1}{|\frac{k_1}{k}|^{4H-2b}}\frac{1}{|1+\frac{k_1}{k}|^{4H-2}}\bigg)\\
&\gtrsim \sum_{k\geq K}\frac{1}{|k|^{8H-4b{\cop -2c}-1}}\int_1^\infty \frac{dx}{|x|^{4H-2b} |1+x|^{4H-2}}\\
&\gtrsim \sum_{k\geq K}\frac{1}{|k|^{8H-4b{\cop -2c}-1}}\int_1^\infty \frac{dx}{|x|^{8H-2b-2} }.
\end{align*}
At this point, recall that {\cop $b+\frac{c}{2}=2H-\frac12$}, and so the latter bound reduces in fact to
$$\sum_{k\geq K}\frac{1}{|k|^{8H-4b{\cop -2c}-1}}\int_1^\infty \frac{dx}{|x|^{8H-2b-2} }=\bigg(\sum_{k\geq K}\frac{1}{|k|}\bigg)\bigg(\int_1^\infty \frac{dx}{|x|^{8H-2b-2}}\bigg)=\infty,$$
which leads us to the desired conclusion
$$\lim_{n\to\infty} \mathfrak{M}^{(n)}=\infty.$$

\

\subsection{Auxiliary lemmas}

\

\smallskip

The few technical results below have been used in the proof of Proposition \ref{prop:ice-cream-frac-intro}. {\cop The first, resp. second, one is borrowed from \cite[Lemma 4.2]{GTV}, resp. \cite[Lemma 3.4]{DFT}.}

\smallskip

\cop

\begin{lemma}[\cite{GTV}]\label{lem:gtv}
Let $\al\geq \beta>0$ such that $\al+\beta>1$. Then for all $A,B\in \R$, one has
$$\max\bigg(\int_{\R} \frac{dx}{\langle x-A\rangle^\al \langle x-B\rangle^{\beta}},\sum_k \frac{1}{\langle k-A\rangle^\al \langle k-B\rangle^\beta}\bigg) \lesssim \frac{1}{\langle A-B \rangle^\ga}$$
where $\al$ is given by
$$\ga:=
\begin{cases}
\al+\beta-1 & \text{if} \ \al<1\\
\beta-\varepsilon & \text{if} \ \al=1\\
\beta & \text{if} \ \al>1,
\end{cases} $$
for every $\varepsilon >0$ small enough.

\end{lemma}

\

\begin{lemma}[\cite{DFT}]\label{lem:int-singu-zero}
Consider  parameters $\mu, \nu\in [0,1)$. Then for all $(A,B)\in \R^2$ such that $|A|\leq |B|$, and for every $0<\varepsilon<1$, it holds that
\begin{equation}\label{adapt-lem-3}
\int_{\R}\frac{d\xi}{|\xi|^\nu} \frac{1}{\langle\xi-A\rangle}\frac{1}{\langle \xi-B\rangle} \lesssim \frac{1}{\langle A\rangle^\nu}\frac{1}{\langle B-A\rangle^{1-\varepsilon}}.
\end{equation}

\end{lemma}

\

\begin{lemma}\label{lem:elementar}
Let $\al>\varepsilon>0$. Then one has
$$\sup_{A\in \R} \int_{\R} \frac{d\xi}{|\xi|^{\al}} \frac{1}{\langle \xi+A \rangle} < \infty \quad \text{and} \quad \int_{\R} \frac{d\xi}{|\xi|^{\al}} \frac{1}{\langle \xi+B \rangle} \lesssim \frac{1}{|B|^{\al-\varepsilon}} \quad \text{for every} \  B\in \R.$$
As a result, for all $A,B\in \R$ and  $0<\varepsilon<\al,\beta<\frac12$, it holds that
\begin{equation}\label{elem-3}
\int_{\R} \frac{d\xi}{|\xi|^{\al}|\xi+A|^{\beta}} \frac{1}{\langle \xi+B\rangle} \lesssim \frac{1}{|B-A|^{\beta-\varepsilon}} .
\end{equation}
\end{lemma}

\begin{proof}
It suffices to observe that for all $A\in \R$ and $\varepsilon>0$ small enough,
\begin{align*}
\int_{\R} \frac{d\xi}{|\xi|^{\al}} \frac{1}{\langle \xi+A \rangle}&\lesssim \1_{\{|A|\leq 1\}}\bigg[\int_{\{|\xi|\leq 2\}} \frac{d\xi}{|\xi|^{\al}} +\int_{\{|\xi|>2\}} \frac{d\xi}{|\xi|^{\al}} \frac{1}{| |\xi|-1 |}   \bigg]+\1_{\{|A|> 1\}}\int_{\R} \frac{d\xi}{|\xi|^{\al}} \frac{1}{|\xi+A |^{1-\varepsilon}}\\
&\lesssim \1_{\{|A|\leq 1\}}\bigg[\int_{\{|\xi|\leq 2\}} \frac{d\xi}{|\xi|^{\al}} +\int_{\{|\xi|>2\}} \frac{d\xi}{|\xi|^{1+\al}}    \bigg]+\1_{\{|A|> 1\}}\frac{1}{|A|^{\al-\varepsilon}}\int_{\R} \frac{d\xi}{|\xi|^{\al}} \frac{1}{|\xi+1 |^{1-\varepsilon}}\lesssim 1.
\end{align*}
In the same way,
$$\int_{\R} \frac{d\xi}{|\xi|^{\al}} \frac{1}{\langle \xi+B \rangle}  \lesssim \int_{\R} \frac{d\xi}{|\xi|^{\al}} \frac{1}{|\xi+B|^{1-\varepsilon} }\lesssim \frac{1}{|B|^{\al-\varepsilon}}.$$
The estimate \eqref{elem-3} then follows from an application of Cauchy-Schwarz inequality.
\end{proof}

\cob

\

\begin{lemma}\label{lem:estim-m-r}
For all $L\geq 1$ and $\be,\be'\in \R$, let $M^{L}_{\beta,\beta'}$ and $R^{L}_{\beta,\beta'}$ be the functions on $\R$ defined by \eqref{defin-m-l} and \eqref{defin-r-l}, respectively. Then the following estimates hold true:

\smallskip

\noindent
$(i)$ For all $|\la|\geq 2$, one has
$$\big|M^{L}_{\beta,\beta'}(\la ) \big|\lesssim {\cop \max\bigg(\frac{1}{L|\la |} \frac{1}{\langle \beta\rangle\langle \beta'\rangle},\frac{1}{L^2|\la |^2}\bigg)}  \quad \text{and} \quad \big|R^{L}_{\beta,\beta'}(\la ) \big|\lesssim \frac{1}{L^2|\la |^2 } \frac{1}{\langle \beta\rangle\langle \beta'\rangle}.$$

\
 
\noindent
$(ii)$ For all $|\la|\leq 2$, one has
$$\big|M^{L}_{\beta,\beta'}(\la ) \big|\lesssim {\cop \max\bigg(\frac{1}{L|\la |} \frac{1}{\langle \beta\rangle\langle \beta'\rangle},\frac{1}{L^{2-\varepsilon}}\frac{1}{|\la | |\la-1|^{1-\varepsilon}}\bigg)} 
$$
and
\begin{align}
&\big|R^{L}_{\beta,\beta'}(\la ) \big|\nonumber\\
&\lesssim \frac{1}{L^2|\la | }\bigg[ \frac{1}{\langle \beta\rangle\langle \beta'\rangle}+\frac{1}{\langle L+\beta\rangle\langle \beta'\rangle}+\frac{1}{\langle L+\beta\rangle\langle L+\beta- \beta'\rangle}+\frac{1}{\langle L-\beta'\rangle\langle \beta\rangle}+\frac{1}{\langle L-\beta'\rangle\langle L+\beta-\beta'\rangle}\bigg].\label{bou-rl}
\end{align}

\end{lemma}

\begin{proof}
\color{blue}
Note first thgat the bound
$$\big|M^{L}_{\beta,\beta'}(\la ) \big|\lesssim \frac{1}{L|\la |} \frac{1}{\langle \beta\rangle\langle \beta'\rangle}$$
in both $(i)$ and $(ii)$ is a straightforward consequence of the fact that
\begin{equation}\label{unif-bou-xi-2}
|\Xi_t(-\beta)| \lesssim \frac{1}{\langle \beta\rangle}  .
\end{equation}
\color{black}

$(i)$ If $|\la|\geq 2$, then by an elementary integration-by-parts argument, we get that
\begin{align*}
&\big|M^{L}_{\beta,\beta'}(\la ) \big|\lesssim \frac{1}{L^2|\la|^2 }\int dt\,  \Big[\big|\partial_t(\chi(.)^4)(t)\big| \big|\Xi_t(-\beta)\big| \big|\Xi_t(-\beta')\big|\\
&\hspace{3cm}+\big|\chi(t)\big|^4 \big|\partial_t(\Xi_.(-\beta))(t)\big| \big|\Xi_t(-\beta')\big|+\big|\chi(t)\big|^4 \big|\Xi_t(-\beta)\big| \big|\partial_t (\Xi_.(-\beta'))(t)\big|\Big]\lesssim \frac{1}{L^2|\la|^2 },
\end{align*}
where we have used the uniform bound
\begin{equation}\label{unif-bou-xi}
|\Xi_t(-\beta)| + |\partial_t \big(\Xi_t(-\beta))| \lesssim 1.
\end{equation}
In the same way,
\begin{align*}
&\big|R^{L}_{\beta,\beta'}(\la ) \big| \lesssim \frac{1}{L^2|\la|^2 }\bigg[\int dt\, |\chi''(t)|\int_0^t ds \, |\chi(s)|^3 |\Xi_s(-\beta)| |\Xi_s(-\beta')|\\
&\hspace{5cm} +\int dt\,  |\chi'(t)| |\chi(t)|^3 |\Xi_t(-\beta)| |\Xi_t(-\beta')| \bigg]\lesssim \frac{1}{L^2|\la|^2 }\frac{1}{\langle \beta\rangle\langle \beta'\rangle},
\end{align*}
where we have used \eqref{unif-bou-xi-2} to derive the last inequality.

\

\noindent
$(ii)$ For $M^{L}_{\beta,\beta'}(\la )$, we can use again \eqref{unif-bou-xi} and an integration-by-parts argument to write, for $|\la|\leq 2$,
\begin{align*}
\big| M^{L}_{\beta,\beta'}(\la )\big|&\lesssim \frac{1}{L|\la| }\bigg|\int dt\, e^{-\imath Lt(\la-1)} \chi(t)^4\,  \overline{\Xi_t(-\beta)}\Xi_t(-\beta')\bigg|^{1-\varepsilon}\lesssim \frac{1}{L^{2-\varepsilon}|\la| |\la-1|^{1-\varepsilon}}. 
\end{align*}

\smallskip

As for $R^{L}_{\beta,\beta'}(\la )$, observe first that
\small
\begin{align*}
&\int_0^t ds \, e^{\imath s L} \chi(s)^3 \overline{\Xi_s(-\beta)}\Xi_s(-\beta')=\frac{1}{\imath L} e^{\imath t L} \chi(t)^3 \overline{\Xi_t(-\beta)}\Xi_t(-\beta')\\
&-\frac{1}{\imath L} \int_0^t ds \, e^{\imath s L} \partial_s(\chi^3)(s) \overline{\Xi_s(-\beta)}\Xi_s(-\beta')-\frac{1}{\imath L} \int_0^t ds \, e^{\imath s(L+\beta)}\chi(s)^4 \Xi_s(-\beta')-\frac{1}{\imath L} \int_0^t ds \, e^{\imath s(L-\beta')}\chi(s)^4 \overline{\Xi_s(-\beta)},
\end{align*}
\normalsize
and so, by \eqref{unif-bou-xi-2}, we get that
\begin{align*}
&|R^{L}_{\beta,\beta'}(\la )|=\frac{1}{L |\la|}\bigg|\int dt\, e^{-\imath \la Lt} \chi'(t)\int_0^t ds \, e^{\imath s L} \chi(s)^3 \overline{\Xi_s(-\beta)}\Xi_s(-\beta')\bigg|\\
&\lesssim \frac{1}{L^2|\la|}\bigg[\frac{1}{\langle \beta\rangle\langle \beta'\rangle}+\int dt\, |\chi'(t)|\bigg|\int_0^t ds \, e^{\imath s(L+\beta)}\chi(s)^4 \Xi_s(-\beta')\bigg|+\int dt\, |\chi'(t)| \bigg|\int_0^t ds \, e^{\imath s(L-\beta')}\chi(s)^4 \overline{\Xi_s(-\beta)}\bigg|\bigg].
\end{align*}
Then, using {\cop similar integration-by-parts arguments as before}, we obtain that
\begin{align*}
\bigg|\int_0^t ds \, e^{\imath s(L+\beta)}\chi(s)^4 \Xi_s(-\beta')\bigg|\lesssim \frac{1}{\langle L+\beta\rangle}\bigg[ \frac{1}{\langle \beta'\rangle}+ \frac{1}{\langle L+\beta-\beta'\rangle}\bigg],
\end{align*}
as well as
\begin{align*}
\bigg|\int_0^t ds \, e^{\imath s(L-\beta')}\chi(s)^4 \overline{\Xi_s(-\beta)}\bigg|\lesssim \frac{1}{\langle L-\beta'\rangle} \bigg[\frac{1}{\langle \beta\rangle}+\frac{1}{\langle L+\beta-\beta'\rangle}\bigg],
\end{align*}
which yields the desired bound \eqref{bou-rl}.

\end{proof}

\

\begin{lemma}\label{lem:semi-conv}
Fix $\al\in (0,1)$.

\

\noindent
$(i)$ It holds that
\begin{align}
&\sup_{\substack{L\geq 1\\ t,r',u'\in [-2,2]}}\bigg|\int_{\R} \frac{ds}{|s|^{\al}}e^{\imath s} \1_{[0,t-\frac{s}{L}]}(r')\1_{[0,t-\frac{s}{L}]}(u')\chi\Big(t-\frac{s}{L}\Big)   \bigg|\lesssim 1\label{item-i-unif}
\end{align}
and
\begin{align}
&\sup_{\substack{L\geq 1\\ r'\in [-2,2]}}\bigg|\int_{\R} \frac{dr}{|r|^{\al}} \, e^{\imath r} \1_{[0,t]}\Big(r'+\frac{r}{L}\Big) \chi\Big(r'+\frac{r}{L}\Big)   \bigg|\lesssim 1.\label{item-i-bis-unif}
\end{align}

\

\noindent
$(ii)$ For all $L\geq 1$ and $t,r',u'\in [-2,2]$, one has
\small
\begin{align}
&\bigg|\int_{\R} \frac{ds}{|s|^{\al}}e^{\imath s} \bigg\{\1_{[0,t-\frac{s}{L}]}(r')\1_{[0,t-\frac{s}{L}]}(u')\chi\Big(t-\frac{s}{L}\Big)-\1_{[0,t]}(r')\1_{[0,t]}(u')\chi\big(t\big)\bigg\}   \bigg|\lesssim \frac{1}{L^{\al}}   \bigg\{\frac{1}{|t-r'|^{\al}}+\frac{1}{|t-u'|^{\al}}\bigg\} \label{ite-ii-i}
\end{align}
\normalsize
and
\begin{align}
&\bigg|\int_{\R} \frac{dr}{|r|^{\al}}e^{\imath r} \bigg\{\1_{[0,t]}\Big(r'+\frac{r}{L}\Big)\chi\Big(r'+\frac{r}{L}\Big)-\1_{[0,t]}(r')\chi\big(r'\big)\bigg\}   \bigg|\lesssim \frac{1}{L^{\al}}  \bigg\{\frac{1}{|t|^{\al}} +\frac{1}{|t-r'|^{\al}}+\frac{1}{|r'|^{\al}} \bigg\}.\label{ite-ii-ii}
\end{align}

\end{lemma}

\

\begin{proof}

\

\smallskip

\noindent
$(i)$ Let $L\geq 1$ and $t,r',u'\in [-2,2]$. One has, for almost every $s\in \R$,
\begin{align*}
\1_{[0,t-\frac{s}{L}]}(r')&=\1_{\{ r'\geq 0\}}\1_{\{s\leq L(t-r')\}}-\1_{\{ r'< 0\}}\1_{\{s\geq L(t-r')\}},
\end{align*}
and so we can write
\begin{align*}
&\1_{[0,t-\frac{s}{L}]}(r')\1_{[0,t-\frac{s}{L}]}(u')\\
&=\Big(\1_{\{ r'\geq 0\}}\1_{\{s\leq L(t-r')\}}-\1_{\{ r'<0\}}\1_{\{s\geq L(t-r')\}}\Big) \Big(\1_{\{ u'\geq 0\}}\1_{\{s\leq L(t-u')\}}-\1_{\{ u'< 0\}}\1_{\{s\geq L(t-u')\}}\Big)\\
&=\1_{\{ r',u'\geq 0\}}\1_{\{s\leq L[(t-r')\wedge (t-u')]\}}+\1_{\{ r',u'< 0\}}\1_{\{s\geq L[(t-r')\vee(t-u')]\}}.
\end{align*}
Therefore
\begin{align*}
&\int_{\R} \frac{ds}{|s|^{\al}}e^{\imath s}\1_{[0,t-\frac{s}{L}]}(r')\1_{[0,t-\frac{s}{L}]}(u') \chi\Big(t-\frac{s}{L}\Big)\\
&=\1_{\{ r',u'\geq 0\}}\int_{-\infty}^{L[(t-r')\wedge (t-u')]} \frac{ds}{|s|^{\al}}e^{\imath s} \chi\Big(t-\frac{s}{L}\Big)+\1_{\{ r',u'< 0\}}\int_{L[(t-r')\vee(t-u')]}^\infty \frac{ds}{|s|^{\al}}e^{\imath s} \chi\Big(t-\frac{s}{L}\Big),
\end{align*}
and we can apply the estimate \eqref{a-1-a-2} in Lemma \ref{lem:der} below to deduce \eqref{item-i-unif}.

\

We can then use the same arguments to prove \eqref{item-i-bis-unif}, by noting that
$$\1_{[0,t]}\Big(r'+\frac{r}{L}\Big)=\1_{\{t\geq 0\}} \1_{\{-r'L \leq r\leq -r'L+Lt\}}-\1_{\{t< 0\}} \1_{\{-r'L+Lt \leq r\leq -r'L\}}.$$

\

\noindent
$(ii)$ Regarding \eqref{ite-ii-i}, we naturally start with the bound
\begin{align}
&\bigg|\int_{\R} \frac{ds}{|s|^{\al}}e^{\imath s} \bigg\{\1_{[0,t-\frac{s}{L}]}(r')\1_{[0,t-\frac{s}{L}]}(u')\chi\Big(t-\frac{s}{L}\Big)-\1_{[0,t]}(r')\1_{[0,t]}(u')\chi\big(t\big)\bigg\}   \bigg|\nonumber\\
&\leq \bigg|\int_{\R} \frac{ds}{|s|^{\al}}e^{\imath s} \big\{\1_{[0,t-\frac{s}{L}]}(r')-\1_{[0,t]}(r')\big\}\1_{[0,t-\frac{s}{L}]}(u')\chi\Big(t-\frac{s}{L}\Big)  \bigg|\nonumber\\
&\hspace{1cm}+\1_{[0,t]}(r')\bigg|\int_{\R} \frac{ds}{|s|^{\al}}e^{\imath s} \big\{\1_{[0,t-\frac{s}{L}]}(u')-\1_{[0,t]}(u')\big\}\chi\Big(t-\frac{s}{L}\Big)  \bigg|\nonumber\\
&\hspace{2.5cm}+\1_{[0,t]}(r')\1_{[0,t]}(u')\bigg|\int_{\R} \frac{ds}{|s|^{\al}}e^{\imath s} \bigg\{\chi\Big(t-\frac{s}{L}\Big)-\chi\big(t\big)\bigg\}   \bigg|\, .\label{dec-three}
\end{align}

\

For the first integral in \eqref{dec-three}, observe for instance that for almost every $s\geq 0$,
\small
\begin{align}
\1_{[0,t-\frac{s}{L}]}(r')-\1_{[0,t]}(r') &=\1_{\{r'\geq 0\}}\big\{ \1_{\{t\geq r'+\frac{s}{L}\}}-\1_{\{t\geq r'\}}\big\}-\1_{\{r'< 0\}}\big\{ \1_{\{t\leq r'+\frac{s}{L}\}}-\1_{\{t\leq r'\}}\big\}\nonumber\\
&=-\big\{\1_{\{r'\geq 0\}}+\1_{\{r'< 0\}}\big\}\1_{\{r'\leq t\leq r'+\frac{s}{L}\}}=-\1_{\{r'\leq t\}}\1_{\{s\geq L(t- r')\}},\label{diff-indi}
\end{align}
\normalsize
which, combined with the previous decomposition
$$\1_{[0,t-\frac{s}{L}]}(u')=\1_{\{ u'\geq 0\}}\1_{\{s\leq L(t-u')\}}-\1_{\{ u'< 0\}}\1_{\{s\geq L(t-u')\}},$$
yields
\begin{align*}
&\big\{\1_{[0,t-\frac{s}{L}]}(r')-\1_{[0,t]}(r')\big\}\1_{[0,t-\frac{s}{L}]}(u')\\
&=- \big\{\1_{\{0\leq u'\leq r'\leq t\}}\1_{\{L(t-r')\leq s\leq L(t-u')\}}-\1_{\{r'\leq t\}}\1_{\{ u'< 0\}}\1_{\{s\geq L(t-u')\}}\1_{\{s\geq L(t- r')\}}\big\}.
\end{align*}
As a result,
\begin{align*}
&\int_0^\infty \frac{ds}{s^{\al}}e^{\imath s} \big\{\1_{[0,t-\frac{s}{L}]}(r')-\1_{[0,t]}(r')\big\}\1_{[0,t-\frac{s}{L}]}(u')\chi\Big(t-\frac{s}{L}\Big) \\
&=-\1_{\{0\leq u'\leq r'\leq t\}}\int_{L|t-r'|}^{L|t-u'|} \frac{ds}{s^{\al}}e^{\imath s} \chi\Big(t-\frac{s}{L}\Big)+\1_{\{r'\leq t\}}\1_{\{ u'< 0\}} \int_{\max(L|t-r'|,L(t-u'))}^\infty \frac{ds}{s^{\al}}e^{\imath s} \chi\Big(t-\frac{s}{L}\Big),
\end{align*}
and we can use the estimate \eqref{b} in Lemma \ref{lem:der} to assert that
$$\bigg|\int_0^\infty \frac{ds}{s^{\al}}e^{\imath s} \big\{\1_{[0,t-\frac{s}{L}]}(r')-\1_{[0,t]}(r')\big\}\1_{[0,t-\frac{s}{L}]}(u')\chi\Big(t-\frac{s}{L}\Big)\bigg| \lesssim\frac{1}{L^{\al}}  \max \bigg(\frac{1}{|t-r'|^{\al}},\frac{1}{|t-u'|^{\al}}\bigg)  ,$$
which corresponds to the desired bound.

\

The bound for the second integral in \eqref{dec-three} immediately follows from the same arguments, that is from the combination of \eqref{diff-indi} with \eqref{b}: we get here that
$$\bigg|\int_{\R} \frac{ds}{|s|^{\al}}e^{\imath s} \big\{\1_{[0,t-\frac{s}{L}]}(u')-\1_{[0,t]}(u')\big\}\chi\Big(t-\frac{s}{L}\Big)  \bigg|\lesssim \frac{1}{L^{\al}|t-u'|^{\al}} .$$

\

As for the third integral in \eqref{dec-three}, we can write
\begin{align*}
\bigg|\int_0^\infty \frac{ds}{s^\al} e^{\imath s} \Big\{\chi\big(t-\frac{s}{L}\big)-\chi(t)\Big\}\bigg|&=\bigg|\frac{\al}{\imath}\int_0^\infty \frac{ds}{s^{\al+1}} e^{\imath s}\Big\{\chi\Big(t-\frac{s}{L}\Big)-\chi(t)\Big\}+\frac{1}{\imath L} \int_0^\infty \frac{ds}{s^\al} e^{\imath s} \chi'\Big(t-\frac{s}{L}\Big)\bigg|\\
&\lesssim \frac{1}{L^\al}\int_0^\infty \frac{ds}{s^{\al+1}} \big|\chi(t-s)-\chi(t)\big|+\frac{1}{L} \int_0^{4L} \frac{ds}{s^\al}\lesssim \frac{1}{L^\al},
\end{align*}
which completes the proof of \eqref{ite-ii-i}.

\

The proof of \eqref{ite-ii-ii} stems from a similar strategy: observe for instance the decomposition, valid for almost all $t,r> 0$,
\small
\begin{align*}
&\1_{[0,t]}\Big(r'+\frac{r}{L}\Big)-\1_{[0,t]}(r')=\1_{\{-\frac{r}{L}\leq r'\leq t-\frac{r}{L}\}}-\1_{\{0\leq r'\leq t\}}  \\
&=\1_{\{0\leq t\leq \frac{r}{L}\}}\big\{ \1_{\{-\frac{r}{L}\leq r'\leq t-\frac{r}{L}\}}-\1_{\{0\leq r'\leq t\}}  \big\}+\1_{\{ t>\frac{r}{L}\}}\big\{ \1_{\{-\frac{r}{L}\leq r'\leq 0\}}-\1_{\{t-\frac{r}{L}\leq r'\leq t\}}  \big\}  \\
&= \1_{\{r\geq L|t|\}} \1_{\{r\geq -Lr'\}}\1_{\{ r\leq L(t-r')\}}-\1_{\{r\geq L|t|\}} \1_{\{0\leq r'\leq t\}} +\1_{\{r\geq L |r'|\}}\1_{\{ r<Lt\}}\1_{\{r'\leq 0\}}-\1_{\{r\geq L|t-r'|\}}\1_{\{ r<Lt\}}  \1_{\{r'\leq t\}},
\end{align*}
\normalsize
which paves the way toward the application of \eqref{b}, just as above.
\end{proof}

\

\begin{lemma}\label{lem:der}
Fix $\al\in (0,1)$. For all test-function $\vp:\R\to\R$ with support in $[-K,K]$ ($K\geq 1$) and all $\la\in \{-1,1\}$, it holds that
\begin{equation}\label{a-1-a-2}
\sup_{\substack{A_1,A_2\in [-\infty,\infty]\\L\geq 1}} \bigg|\int_{A_1}^{A_2}\frac{dr}{|r|^\al}e^{ \imath \la r} \vp\Big(\frac{r}{L}\Big)\bigg| \lesssim \|\vp\|_{\infty}+\|\vp'\|_\infty K^{1-\al},
\end{equation}
and for all $0<B_1\leq B_2\leq \infty$, $L\geq 1$,
\begin{equation}\label{b}
\bigg|\int_{B_1}^{B_2} \frac{dr}{r^\al} e^{\imath \la r}\vp\Big(\frac{r}{L}\Big)\bigg|\lesssim \frac{1}{B_1^\al} \big\|\vp\big\|_\infty+\frac{K^{1-\al}}{L^\al} \big\|\vp'\big\|_\infty.
\end{equation}

\end{lemma}

\begin{proof}
For \eqref{b}, let us write, with an elementary integration by parts,
\begin{align*}
&\int_{B_1}^{B_2} \frac{dr}{r^\al} e^{\imath \la r}\vp\Big(\frac{r}{L}\Big)\\
&=\frac{1}{\imath \la} \frac{1}{B_2^\al} e^{\imath \la B_2}\vp\Big(\frac{B_2}{L}\Big)-\frac{1}{\imath \la} \frac{1}{B_1^\al} e^{\imath \la B_1}\vp\Big(\frac{B_1}{L}\Big)+\frac{\al}{\imath \la}\int_{B_1}^{B_2} \frac{dr}{r^{\al+1}} e^{\imath \la r}\vp\Big(\frac{r}{L}\Big)-\frac{1}{\imath \la L}\int_{B_1}^{B_2} \frac{dr}{r^{\al}} e^{\imath \la r}\vp'\Big(\frac{r}{L}\Big),
\end{align*}
which, since $\text{supp}\, \vp' \subset [-K,K]$,  gives
\begin{align*}
\bigg|\int_{B_1}^{B_2} \frac{dr}{r^\al} e^{\imath \la r}\vp\Big(\frac{r}{L}\Big)\bigg| \lesssim \frac{1}{B_1^\al} \big\|\vp\big\|_\infty+\big\|\vp\big\|_\infty\int_{B_1}^\infty \frac{dr}{r^{\al+1}} +\frac{1}{ L}\int_0^{KL} \frac{dr}{r^{\al}} \Big|\vp'\Big(\frac{r}{L}\Big)\Big|\lesssim \frac{1}{B_1^\al} \big\|\vp\big\|_\infty+\frac{K^{1-\al}}{L^\al} \big\|\vp'\big\|_\infty.
\end{align*}

\

As for \eqref{a-1-a-2}, it suffices to observe that for every $A\in [0,\infty]$, one has
\begin{align*}
\bigg|\int_0^A \frac{dr}{r^\al}e^{ \imath \la r} \vp\Big(\frac{r}{L}\Big)\bigg|&\leq \1_{\{0\leq A\leq 1\}}\|\vp\|_{\infty}\int_0^1 \frac{dr}{r^\al}+\1_{\{A\geq 1\}}\bigg[\|\vp\|_{\infty}\int_0^1 \frac{dr}{r^\al} +\bigg|\int_1^A \frac{dr}{r^\al}e^{ \imath\la r}\vp\Big(\frac{r}{L}\Big)\bigg|\bigg]\\
&\lesssim \|\vp\|_\infty+\1_{\{A\geq 1\}}\bigg|\int_1^A \frac{dr}{r^\al}e^{ \imath \la r} \vp\Big(\frac{r}{L}\Big)\bigg|.
\end{align*}
We can then use \eqref{b} to (uniformly) bound the latter integral and derive \eqref{a-1-a-2}. 


\end{proof}

\

\color{blue}

\subsection{Proof of Proposition \ref{prop:explo-white}}\label{sec:explo-white}

We assume in this section (and in this section only) that $H=\frac12$. Recall that in this case
$$\mathbb{E}\Big[\dot{\beta}^{(k)}_t \dot{\beta}^{(k')}_{t'}\Big]=\1_{\{k=k'\}} \delta_{\{t=t'\}},$$
and so
\begin{align}
&\mathbb{E}\Big[ \cf(\<Psi>^{(n)}_k)(\la)\overline{\cf(\<Psi>^{(n)}_{k'})(\la')}\Big]\nonumber\\
&=\1_{\{\langle k\rangle \leq 2^n\}}\1_{\{\langle k'\rangle \leq 2^n\}}\int d\xi d\xi'\, \ci_\chi(\la,\xi)\overline{\ci_\chi(\la',\xi')}\int dt dt' \, e^{-\imath (\xi+|k|^2) t}e^{\imath (\xi'+|k'|^2) t'}\mathbb{E}\Big[\dot{\beta}^{(k)}_t \dot{\beta}^{(k')}_{t'}\Big]\nonumber\\
&=\1_{\{k=k'\}}\1_{\{\langle k\rangle \leq 2^n\}}\int d\xi d\xi'\, \ci_\chi(\la,\xi)\overline{\ci_\chi(\la',\xi')}\int dt  \, e^{-\imath (\xi+|k|^2) t}e^{\imath (\xi'+|k|^2) t}\nonumber\\
&=\1_{\{k=k'\}}\1_{\{\langle k\rangle \leq 2^n\}}\int d\xi \, \ci_\chi(\la,\xi)\overline{\ci_\chi(\la',\xi)}\nonumber\\
&=c\, \1_{\{k=k'\}}\1_{\{\langle k\rangle \leq 2^n\}}\int dt dt' \, \chi(t) \chi(t')e^{-\imath \la t}e^{\imath \la' t'}\int_{\R} d\xi\,  \Xi_t(-\xi) \overline{\Xi_{t'}(-\xi)},\nonumber
\end{align}
which immediately extends \eqref{cova-refin}. As a result, we can follow the arguments leading to \eqref{expression-q} and conclude that for every $b>\frac12$,
\begin{equation*}
\mathbb{E}\Big[\big\|\widetilde{\<IPsi2>}^{(n)}\big\|^2_{Z^{0,b}} \Big]=\sum_{k\neq 0} \sum_{k_1\neq 0} \1_{\{\langle k_1\rangle\leq 2^n\}} \1_{\{\langle k+k_1\rangle\leq 2^n\}}\int d\xi d\eta\int d\la \, \langle \la \rangle^{2b} \big|Q^{\Omega_{k,k_1}}_{\xi,\eta}(\la)\big|^2, 
\end{equation*}
where $Q^L_{\xi,\eta}(\la)$ is the quantity introduced in \eqref{defi:qlbeta}. In particular,
\begin{align*}
\mathbb{E}\Big[\big\|\widetilde{\<IPsi2>}^{(n)}\big\|^2_{Z^{0,b}} \Big]
&\geq \sum_{k\geq 1} \sum_{k_1\geq 1}\1_{\{\langle k_1\rangle\leq 2^n\}} \1_{\{\langle k+k_1\rangle\leq 2^n\}}|\Omega_{k,k_1}|^{2}  \int d\xi d\eta \int d\la \, | \la |^{2b} \big|Q^{\Omega_{k,k_1}}_{\xi,\eta}(\la \Omega_{k,k_1})\big|^2 ,
\end{align*}
and we can use the decomposition $Q^{L}_{\xi,\eta}(\la L)=M^{L}_{\xi,\eta}(\la )+R^{L}_{\xi,\eta}(\la )$ exhibited in \eqref{decompoqlbeta} to derive that 
\begin{align*}
\mathbb{E}\Big[\big\|\widetilde{\<IPsi2>}^{(n)}\big\|^2_{Z^{0,b}} \Big]&\geq \sum_{k\geq 1} \sum_{k_1\geq 1}\1_{\{\langle k_1\rangle\leq 2^n\}} \1_{\{\langle k+k_1\rangle\leq 2^n\}}|\Omega_{k,k_1}|^{2}  \int d\xi d\eta \int d\la \, | \la |^{2b} \big|M^{\Omega_{k,k_1}}_{\xi,\eta}(\la )\big|^2 \\
&\hspace{0.5cm}-2\sum_{k\geq 1} \sum_{k_1\geq 1}|\Omega_{k,k_1}|^{2}  \int d\xi d\eta \int d\la \, | \la |^{2b} \big|M^{\Omega_{k,k_1}}_{\xi,\eta}(\la )\big| \big|R^{\Omega_{k,k_1}}_{\xi,\eta}(\la )\big|.
\end{align*}

\

By applying Lemma \ref{lem:estim-m-r}, we get that for all $\Omega\geq 1$, $k,\ell\geq 1$, $\frac12<b<1$ and $\varepsilon >0$ small enough,
\begin{align*}
&\int d\xi d\eta \int d\la \, | \la |^{2b} \big|M^{\Omega}_{\xi,\eta}(\la )\big|\big|R^{\Omega}_{\xi,\eta}(\la )\big|\\
&\lesssim \int d\xi d\eta \bigg[ \frac{1}{\langle \xi\rangle^{\frac32} \langle \eta\rangle^{\frac32}}\frac{1}{\Omega^{\frac72}}\int_{\{|\la|\geq 2\}} \frac{d\la}{|\la|^{\frac72-2b}}+\frac{1}{\langle \xi\rangle^{\frac12} \langle \eta\rangle^{\frac12}} \frac{1}{\Omega^{\frac72-\varepsilon}}\bigg( \int_{\{|\la|\leq 2\}} \frac{d\la}{|\la|^{2-2b} |\la-1|^{\frac12-\varepsilon}}\bigg)\\
&\hspace{1cm}\bigg(\frac{1}{\langle \xi\rangle \langle \eta\rangle} +\frac{1}{\langle \Omega+\xi\rangle \langle \eta\rangle}+\frac{1}{\langle \Omega+\xi\rangle \langle \Omega+\xi- \eta\rangle}+\frac{1}{\langle \Omega-\eta\rangle \langle \xi\rangle}+\frac{1}{\langle \Omega-\eta\rangle \langle \Omega+\xi-\eta\rangle}\bigg)\bigg]\lesssim \frac{1}{\Omega^{\frac72-\varepsilon}},
\end{align*}
where the last estimate can be easily deduced from Lemma \ref{lem:gtv}.

\smallskip

Therefore, for $\frac12<b<1$ and $\varepsilon >0$ small enough,
\begin{align*}
&\sum_{k\geq 1} \sum_{k_1\geq 1}|\Omega_{k,k_1}|^{2}  \int d\xi d\eta \int d\la \, | \la |^{2b} \big|M^{\Omega_{k,k_1}}_{\xi,\eta}(\la )\big| \big|R^{\Omega_{k,k_1}}_{\xi,\eta}(\la )\big|\lesssim \sum_{k\geq 1} \sum_{k_1\geq 1}\frac{1}{|\Omega_{k,k_1}|^{\frac32-\varepsilon}}\lesssim \bigg( \sum_{k} \frac{1}{\langle k\rangle^{\frac32-\varepsilon}}\bigg)^2,
\end{align*}
and we are thus left with the analysis of
\begin{equation*}
\mathfrak{M}^{(n)}:=\sum_{k\geq 1} \sum_{k_1\geq 1}\1_{\{\langle k_1\rangle\leq 2^n\}} \1_{\{\langle k+k_1\rangle\leq 2^n\}} |\Omega_{k,k_1}|^{2}  \int d\xi \int d\eta \int d\la \, | \la |^{2b} \big|M^{\Omega_{k,k_1}}_{\xi,\eta}(\la )\big|^2 .
\end{equation*}
To this end, write, for $\frac12<b<1$,
\begin{align*}
&\int d\xi \int d\eta \int d\la \, | \la |^{2b} \big|M^{\Omega_{k,k_1}}_{\xi,\eta}(\la )\big|^2\\
&=\frac{1}{| \Omega_{k,k_1}|^2}\int d\xi \int d\eta \int \frac{d\la}{| \la |^{2-2b}} \bigg|\int dt\, e^{-\imath \Omega_{k,k_1}t(\la-1)} \chi(t)^4\,  \overline{\Xi_t(-\xi)}\Xi_t(-\eta)\bigg|^2\\
&=\frac{1}{| \Omega_{k,k_1}|^3}\int d\xi \int d\eta \int \frac{d\la}{| 1+\frac{\la}{\Omega_{k,k_1}} |^{2-2b}} \bigg|\int dt\, e^{-\imath t\la} \chi(t)^4\,  \overline{\Xi_t(-\xi)}\Xi_t(-\eta)\bigg|^2\stackrel{k,k_1\to \infty}{\sim} \frac{C}{| \Omega_{k,k_1}|^3},
\end{align*}
where
$$C:=\int d\xi \int d\eta \int d\la\,  \bigg|\int dt\, e^{-\imath t\la} \chi(t)^4\,  \overline{\Xi_t(-\xi)}\Xi_t(-\eta)\bigg|^2.$$
It is easy to check that $C$ is finite and strictly positive, and accordingly by picking $K\geq 1$ large enough, one has
$$
\mathfrak{M}^{(n)}\gtrsim \sum_{k\geq K} \sum_{k_1\geq K}\1_{\{\langle k_1\rangle\leq 2^n\}} \1_{\{\langle k+k_1\rangle\leq 2^n\}} \frac{1}{|\Omega_{k,k_1}|} ,
$$
which, by letting $n$ tends to infinity, achieves the proof of our statement.

\

\color{black}

\section{Proof of Proposition \ref{prop:q-sharp-optim-intro}}\label{sec:prod-restr-space}

\subsection{Proof of Proposition \ref{prop:q-sharp-optim-intro}.}

\

\smallskip

Fix $b\in (\frac12,1)$ and $c\in [0,\frac32-2H)$. For every $n\geq 1$, define ${\cop Y^{(n)}}$ through the formula
\begin{equation}\label{fourieryn}
\cf\big({\cop Y^{(n)}_{k}}\big)(\la):=\1_{\{k\neq 0\}}\frac{1}{\langle k\rangle^{3-4H}}\frac{1}{\langle \la \rangle^2} \cf\big(\<Psi>^{(n)}_{k}\big)(\la).
\end{equation}
For every $q\geq 1$, one has
\begin{align*}
&\mathbb{E}\Big[\big\| {\cop Y^{(n)}}\big\|_{{\cop Z^{c,b}}}^{2q} \Big] 
\lesssim \mathbb{E}\bigg[\bigg(\sum_k\,  \langle k\rangle^{2c}\int d\la \, \langle \la\rangle^{2b} \big|\cf\big({\cop Y^{(n)}_{k}}\big)(\la)\big|^2\bigg)^q  \bigg]\\
&\lesssim \sum_{k_1,\ldots,k_q}\,  \frac{1}{\langle k_1\rangle^{6-8H-2c}}\cdots \frac{1}{\langle k_q\rangle^{6-8H-2c}}\int \frac{d\la_1}{\langle \la_1\rangle^{4-2b}} \cdots \frac{d\la_q}{\langle \la_q\rangle^{4-2b}} \, \mathbb{E}\Big[\big|\cf\big(\<Psi>^{(n)}_{k_1}\big)(\la_1)\big|^2 \cdots \big|\cf\big(\<Psi>^{(n)}_{k_q}\big)(\la_q)\big|^2\Big]\\
&\lesssim \bigg(\sum_{k}\,  \frac{1}{\langle k\rangle^{6-8H-2c}}\int \frac{d\la}{\langle \la\rangle^{4-2b}} \, \mathbb{E}\Big[\big|\cf\big(\<Psi>^{(n)}_{k}\big)(\la)\big|^{2q}\Big]^{\frac{1}{q}} \bigg)^q\lesssim \bigg(\sum_{k}\,  \frac{1}{\langle k\rangle^{6-8H-2c}}\int \frac{d\la}{\langle \la\rangle^{4-2b}} \, \mathbb{E}\Big[\big|\cf\big(\<Psi>^{(n)}_{k}\big)(\la)\big|^{2}\Big] \bigg)^q,
\end{align*}
where we have used the fact that $\cf\big(\<Psi>^{(n)}_{k}\big)(\la)$ is a Gaussian variable to derive the last inequality.

\smallskip

Now we can use the basic covariance estimate \eqref{cova-frac-estim} to assert that
\begin{align*}
\sum_{k}\,  \frac{1}{\langle k\rangle^{6-8H-2c}}\int \frac{d\la}{\langle \la\rangle^{4-2b}} \, \mathbb{E}\Big[\big|\cf\big(\<Psi>^{(n)}_{k}\big)(\la)\big|^{2}\Big]
&\lesssim\sum_{k}\,  \frac{1}{\langle k\rangle^{6-8H-2c}}\int \frac{d\la}{\langle \la\rangle^{6-2b}}\frac{1}{\langle \la+k^2\rangle^{2H-1}},
\end{align*}
{\cop which, thanks to Lemma \ref{lem:gtv}, yields} 
\begin{align*}
\sum_{k}\,  \frac{1}{\langle k\rangle^{6-8H-2c}}\int \frac{d\la}{\langle \la\rangle^{4-2b}} \, \mathbb{E}\Big[\big|\cf\big(\<Psi>^{(n)}_{k}\big)(\la)\big|^{2}\Big]
&\lesssim\sum_{k}\,  \frac{1}{\langle k\rangle^{4-4H-2c}} \ < \ \infty,
\end{align*}
due to $c<\frac32-2H$. We have thus checked the first part of \eqref{bound-z-n}: namely, for every $q\geq 2$,
\begin{equation}\label{sup-yn}
\sup_{n\geq 0} \mathbb{E}\Big[\big\|{\cop Y^{(n)}}\big\|_{{\cop Z^{c,b}}}^q \Big] < \infty. 
\end{equation}

\

We now intend to show that {\cop the second part of \eqref{bound-z-n}, that is}
\begin{equation}
\mathbb{E}\Big[\big\| \cl^{\circ,(n)}{\cop Y^{(n)}}\big\|_{{\cop Z^{0,0}}}^2\Big] \stackrel{n\to\infty}{\longrightarrow} \infty.
\end{equation}

\smallskip

Let us start the analysis by writing, along {\cop \eqref{cl-circ-1-1}, and with expression \eqref{fourieryn} of $\cf\big({\cop Y^{(n)}_{k}}\big)(\la)$ in mind},
\begin{align*}
&\cf\big( \cl^{\circ,(n)}( {\cop Y^{(n)}})_k\big)(\la)\\
&=\sum_{k_1\neq 0}\1_{\{k=0\}\cup\{k_1=k\}}\frac{1}{\langle k_1\rangle^{3-4H}}\int d\la_1 \, \ci_\chi(\la,\la_1) \int \frac{d\la_2}{\langle \la_2 \rangle^2} \, \overline{\cf\big(\<Psi>^{(n)}_{k_1-k}\big)(\la_2-\la_1)}\cf\big(\<Psi>^{(n)}_{k_1}\big)(\la_2),
\end{align*}
and so
\begin{align}
&\mathbb{E}\Big[\big\| \cl^{\circ,(n)}{\cop Y^{(n)}}\big\|_{{\cop Z^{0,0}}}^2\Big]=\sum_k \int d\la \, \mathbb{E}\Big[\big|\cf\big( \cl^{\circ,(n)}( {\cop Y^{(n)}})_k\big)(\la)\big|^2\Big]\nonumber\\
&\geq \int d\la \, \mathbb{E}\Big[\big|\cf\big( \cl^{\circ,(n)}( {\cop Y^{(n)}})_0\big)(\la)\big|^2\Big]\label{inega-intro}\\
&\geq \int d\la \, \mathbb{E}\bigg[\bigg|\sum_{k_1\neq 0}\frac{1}{\langle k_1\rangle^{3-4H}}\int d\la_1 \, \ci_{\chi}(\la,\la_1)\int \frac{d\la_2}{\langle \la_2\rangle^2} \, \overline{\cf\big(\<Psi>^{(n)}_{k_1}\big)(\la_2-\la_1)} \cf\big(\<Psi>^{(n)}_{k_1}\big)(\la_2)\bigg|^2\bigg]\nonumber\\
&\geq \int d\la \, \bigg|\sum_{k_1\neq 0}\frac{1}{\langle k_1\rangle^{3-4H}}\int d\la_1   \, \ci_{\chi}(\la,\la_1)\int \frac{d\la_2}{\langle \la_2\rangle^2}\, \mathbb{E}\Big[\overline{\cf\big(\<Psi>^{(n)}_{k_1}\big)(\la_2-\la_1)} \cf\big(\<Psi>^{(n)}_{k_1}\big)(\la_2) \Big]\bigg|^2 \ =: \ \mathbb{I}^{(n)},\nonumber
\end{align}
where we have used Jensen's inequality to derive the last inequality.

\smallskip

It remains us to prove that
\begin{equation}
\mathbb{I}^{(n)} \stackrel{n\to\infty}{\longrightarrow} \infty.
\end{equation}
To this end, we can use the notation of the subsequent Lemma \ref{lem:decompo-covar} to write
\small
\begin{align*}
&\mathbb{I}^{(n)}=\int d\la \, \bigg|\sum_{k_1\neq 0}\1_{\{\langle k_1\rangle \leq 2^n\}}\frac{1}{\langle k_1\rangle^{3-4H}}\int d\la_1   \, \ci_{\chi}(\la,\la_1)\int \frac{d\la_2}{\langle \la_2\rangle^2}\, \bigg[\frac{1}{|k_1|^{4H-2}} M(\la_2-\la_1,\la_2)+R_{k_1}(\la_2-\la_1,\la_2)\bigg]\bigg|^2\\
&=\int d\la \, \bigg|\bigg(\sum_{{\cop k_1\neq 0}}\1_{\{\langle k_1\rangle \leq 2^n\}}\frac{1}{\langle k_1\rangle^{3-4H}}\frac{1}{|k_1|^{4H-2}}\bigg) \bigg( \int d\la_1   \, \ci_{\chi}(\la,\la_1)\int \frac{d\la_2}{\langle \la_2\rangle^2}\, M(\la_2-\la_1,\la_2)\bigg)\\
&\hspace{4cm}+\sum_{k_1\neq 0}\1_{\{\langle k_1\rangle \leq 2^n\}}\frac{1}{\langle k_1\rangle^{3-4H}}\int d\la_1   \, \ci_{\chi}(\la,\la_1)\int \frac{d\la_2}{\langle \la_2\rangle^2}\, R_{k_1}(\la_2-\la_1,\la_2)\bigg|^2.
\end{align*}
\normalsize
{\cop Using the inequalities $|a+b|^2 \geq |a|^2-2|a||b|$ and $\1_{\{\langle k_1\rangle \leq 2^n\}}\frac{1}{\langle k_1\rangle^{3-4H}}\leq 1$ (recall that $H<\frac34$), we obtain}
\begin{align}
\mathbb{I}^{(n)}&\geq |\mathbb{S}^{(n)}|^2\int d\la \, \bigg| \int d\la_1   \, \ci_{\chi}(\la,\la_1)\int \frac{d\la_2}{\langle \la_2\rangle^2}\, M(\la_2-\la_1,\la_2)\bigg|^2\nonumber\\
&\hspace{0.5cm}-2\,  |\mathbb{S}^{(n)}|\cdot \int d\la \, \bigg[ \int d\la_1   \, \big|\ci_{\chi}(\la,\la_1)\big|\int \frac{d\la_2}{\langle \la_2\rangle^2}\, \big|M(\la_2-\la_1,\la_2)\big|\bigg]\nonumber\\
&\hspace{4cm}\bigg[\sum_{k_1\neq 0}\int d\la'_1   \, \big|\ci_{\chi}(\la,\la'_1)\big|\int \frac{d\la'_2}{\langle \la'_2\rangle^2}\, \big|R_{k_1}(\la'_2-\la'_1,\la'_2)\big|\bigg].\label{lb-i-c-al}
\end{align}
where we have set
\begin{align*}
\mathbb{S}^{(n)}:=\sum_{{\cop k_1\neq 0}}\1_{\{\langle k_1\rangle \leq 2^n\}}\frac{1}{\langle k_1\rangle^{3-4H}}\frac{1}{|k_1|^{4H-2}}.
\end{align*}

\

By combining the definition \eqref{def-m} and the estimate \eqref{bound-r-integr} below, it is easy to check that the integral
$$ \int d\la \, \bigg[ \int d\la_1   \, \big|\ci_{\chi}(\la,\la_1)\big|\int \frac{d\la_2}{\langle \la_2\rangle^2}\, \big|M(\la_2-\la_1,\la_2)\big|\bigg]\bigg[\sum_{k_1\neq 0}\int d\la'_1   \, \big|\ci_{\chi}(\la,\la'_1)\big|\int \frac{d\la'_2}{\langle \la'_2\rangle^2}\, \big|R_{k_1}(\la'_2-\la'_1,\la'_2)\big|\bigg]$$
is finite. Indeed, on the one hand, it holds that
\begin{align*}
\int d\la_1\, \big|\ci_{\chi}(\la,\la_1)\big|\int \frac{d\la_2}{\langle \la_2\rangle^2}\, \big|M(\la_2-\la_1,\la_2)\big|
&\lesssim \frac{1}{\langle \la\rangle} \int  \frac{d\la_1}{\langle \la-\la_1\rangle}\int \frac{d\la_2}{\langle \la_2\rangle^3}\, \frac{1}{\langle \la_2-\la_1\rangle}\\
&\lesssim \frac{1}{\langle \la\rangle} \int  \frac{d\la_1}{\langle \la-\la_1\rangle \langle \la_1\rangle}\lesssim \frac{1}{\langle \la\rangle^{2-\varepsilon}}, 
\end{align*}
{\cop where we have used Lemma \ref{lem:gtv} to derive the last two estimates.} On the other hand,
\begin{align*}
&\sum_{k_1\neq 0}\int d\la'_1   \, \big|\ci_{\chi}(\la,\la'_1)\big|\int \frac{d\la'_2}{\langle \la'_2\rangle^2}\, \big|R_{k_1}(\la'_2-\la'_1,\la'_2)\big|\\
&\lesssim \sum_{k_1\neq 0}\int d\la'_1 \int \frac{d\la'_2}{\langle \la'_2\rangle^2}\, \big|R_{k_1}(\la'_2-\la'_1,\la'_2)\big|\lesssim \sum_{k_1\neq 0}\int d\la'_1 \int \frac{d\la'_2}{\langle \la'_2\rangle^2}\, \big|R_{k_1}(\la'_1,\la'_2)\big|\lesssim \sum_{k_1\neq 0} \frac{1}{|k_1|^{4H-2\varepsilon}}  \lesssim 1,
\end{align*}
where we have used \eqref{bound-r-integr}.

\smallskip

We can thus rephrase \eqref{lb-i-c-al} as
\begin{equation}\label{lb-i-bis}
\mathbb{I}^{(n)}\geq c_0 \, |\mathbb{S}^{(n)}|^2-c_1\,  |\mathbb{S}^{(n)}|=|\mathbb{S}^{(n)}|^2 \bigg[c_0-\frac{c_1}{|\mathbb{S}^{(n)}|}\bigg],
\end{equation}
where $c_1\geq 0$ is a finite constant and
$$c_0:=\int d\la \, \bigg| \int d\la_1   \, \ci_{\chi}(\la,\la_1)\int \frac{d\la_2}{\langle \la_2\rangle^2}\, M(\la_2-\la_1,\la_2)\bigg|^2.$$

\smallskip

Observe that $\mathbb{S}^{(n)} \stackrel{n\to\infty}{\longrightarrow} \infty$, and therefore, based on \eqref{lb-i-bis}, it only remains us to guarantee that $c_0>0$.

\smallskip

To this end, write
\begin{align*}
c_0&=\int d\la \, \bigg| \int d\la_1   \, \ci_{\chi}(\la,\la_1)\int \frac{d\la_2}{\langle \la_2\rangle^2}\, M(\la_2-\la_1,\la_2)\bigg|^2\\
&=c\int d\la \, \bigg| \int dt \, e^{-\imath \la t} \chi(t)\int d\la_1 \, \Xi_t(\la_1)\int \frac{d\la_2}{\langle \la_2\rangle^2}\int ds \, \chi(s)^2 \Xi^s(\la_2-\la_1)\overline{\Xi^s(\la_2)}\bigg|^2\\
&=c \int dt \, \chi(t)^2\bigg| \int ds \, \chi(s)^2 \int \frac{d\la_2}{\langle \la_2\rangle^2}\bigg(\int d\la_1 \, \Xi_t(\la_1)\Xi^s(\la_2-\la_1)\bigg)\overline{\Xi^s(\la_2)}\bigg|^2,
\end{align*}
{\cop where we have used the classical Fourier isometry in $L^2(\R)$ to derive the last equality.} Then
\begin{align*}
\int d\la_1 \, \Xi_t(\la_1)\Xi^s(\la_2-\la_1)&=\int dr \, \chi(r)\1_{[0,t]}(r)\int dv \, \chi(v)\1_{[0,v]}(s)\int d\la_1 \, e^{\imath r \la_1}e^{\imath v(\la_2-\la_1)}\\
&=\int dv \, \chi(v)^2\1_{[0,t]}(v)\1_{[0,v]}(s) e^{\imath v \la_2},
\end{align*}
which yields
\begin{align*}
c_0&=c \int dt \, \chi(t)^2\bigg| \int ds \, \chi(s)^2 \int dr \, \chi(r)\1_{[0,r]}(s)\int dv \, \chi(v)^2\1_{[0,t]}(v)\1_{[0,v]}(s) \int \frac{d\la_2}{\langle \la_2\rangle^2}e^{-\imath \la_2 (r-v)}\bigg|^2\\
&=c \int dt \, \chi(t)^2\bigg| \int ds \, \chi(s)^2 \int dr \, \chi(r)\1_{[0,r]}(s)\int dv \, \chi(v)^2\1_{[0,t]}(v)\1_{[0,v]}(s) e^{-|r-v|}\bigg|^2,
\end{align*}
where we have used the classical identity
$$\int \frac{d\la}{\langle \la\rangle^2}e^{-\imath \la t}=c\, e^{-|t|}.$$
Finally, since $\chi\geq 0$ and $\chi\equiv 1$ on $[-1,1]$, we obtain that
\begin{align*}
c_0&\geq c \int_0^1 dt \, \chi(t)^2\bigg| \int_0^t ds \, \chi(s)^2 \int_s^t dr \, \chi(r)\1_{[0,r]}(s)\int_s^t dv \, \chi(v)^2\1_{[0,t]}(v)\1_{[0,v]}(s) e^{-|r-v|}\bigg|^2\\
&\geq c \int_0^1 dt \, \bigg| \int_0^t ds \,  \int_s^t dr \, \int_s^t dv \, e^{-|r-v|}\bigg|^2 \, > \, 0,
\end{align*}
as desired.

\

\cop

Once endowed with the properties of $Y^{(n)}$ in \eqref{bound-z-n}, we can write for $p>1$ and $q:=\frac{p}{p-1}$,
\begin{align*}
\mathbb{E}\Big[ \big\| \cl^{\circ,(n)} Y^{(n)}\big\|_{L^2(\T \times \R)}^2\Big] &\lesssim \mathbb{E}\Big[ \big\| \cl^{\circ,(n)} Y^{(n)}\big\|_{Z^{c,b}}^2\Big]\\
&\lesssim \mathbb{E}\Big[\big\|\cl^{\circ,(n)}\big\|_{Z^{c,b} \to Z^{c,b}}^{2} \big\|{\cop Y^{(n)}}\big\|_{Z^{c,b}}^2\Big]\\
&\lesssim \mathbb{E}\Big[\big\|\cl^{\circ,(n)}\big\|_{Z^{c,b} \to Z^{c,b}}^{2p}\Big]^{\frac{1}{2p}} \mathbb{E}\Big[ \big\|{\cop Y^{(n)}}\big\|_{Z^{c,b}}^{2q}\Big]^{\frac{1}{2q}}\lesssim \mathbb{E}\Big[\big\|\cl^{\circ,(n)}\big\|_{Z^{c,b} \to Z^{c,b}}^{2p}\Big]^{\frac{1}{2p}},
\end{align*}
due to \eqref{sup-yn}. Since $\mathbb{E}\Big[ \big\| \cl^{\circ,(n)}{\cop Y^{(n)}}\big\|_{L^2(\T \times \R)}^2\Big] \to \infty$ as $n\to \infty$, we obtain the desired explosion in \eqref{explos-op-1}. The assertion \eqref{explos-op-2} then naturally follows from the same arguments, noting that
$$ \big\| \cl^{\circ,(n)}{ Y^{(n)}}\big\|_{L^2(\T \times \R)}\lesssim \big\| \cl^{\circ,(n)}{\cop Y^{(n)}}\big\|_{\widetilde{Z}^{c,b}} \quad \text{and} \quad \big\|Y^{(n)}\big\|_{\widetilde{Z}^{c,b}} \lesssim \big\| Y^{(n)}\big\|_{Z^{c,b}}.$$

\color{black}

\

\subsection{Auxiliary lemma}

\

\smallskip

The following decomposition result has been used in the proof of Proposition \ref{prop:q-sharp-optim-intro}.

\begin{lemma}\label{lem:decompo-covar}
In the setting of Proposition \ref{prop:q-sharp-optim-intro}, one has for every $(k_1,k_1')\neq (0,0)$,
\begin{equation}\label{decompo-covar}
\mathbb{E}\Big[\overline{\cf\big(\<Psi>^{(n)}_{k_1}\big)(\la)}\cf\big(\<Psi>^{(n)}_{k'_1}\big)(\la')  \Big]=\1_{\{k_1=k_1'\}}\1_{\{\langle k_1\rangle \leq 2^n\}} \bigg[\frac{1}{|k_1|^{4H-2}} M(\la,\la')+R_{k_1}(\la,\la')\bigg],
\end{equation}
where, using the notation introduced in \eqref{xi-up-t}, we have set
\begin{equation}\label{def-m}
M(\la,\la'):=c\int_{\R} ds \,  \chi(s)^2 \, \Xi^s(\la ) \overline{\Xi^s(\la')}  
\end{equation}
for some constant $c>0$, and the function $R$ is such that for all $k_1\neq 0$ and $\varepsilon >0$,
\begin{equation}\label{bound-r-integr}
\int d\la_1\int \frac{d\la_2}{\langle \la_2\rangle^2}\big|R_{k_1}(\la_1,\la_2)\big|\lesssim \frac{1}{|k_1|^{4H-2\varepsilon}}.
\end{equation}

\end{lemma}


\begin{proof}

Recall that by \eqref{cova-frac}, one has
\begin{align*}
\mathbb{E}\Big[\overline{\cf\big(\<Psi>^{(n)}_{k_1}\big)(\la)}\cf\big(\<Psi>^{(n)}_{k_1'}\big)(\la')  \Big]
&=c\, \1_{\{k_1=k_1'\}}\1_{\{\langle k_1\rangle \leq 2^n\}}\int_{\R} \frac{d\xi}{|\xi|^{2H-1}} \overline{\ci_\chi(\la,-\xi-k_1^2)}\ci_\chi(\la',-\xi-k_1^2),
\end{align*}
and thus
\begin{align}
&\mathbb{E}\Big[\overline{\cf\big(\<Psi>^{(n)}_{k_1}\big)(\la)}\cf\big(\<Psi>^{(n)}_{k_1'}\big)(\la')  \Big]=c\, \1_{\{k_1=k_1'\}}\1_{\{\langle k_1\rangle \leq 2^n\}}\nonumber\\
&\hspace{0.5cm}\int_{\R} \frac{d\xi}{|\xi|^{2H-1}} \int_{\R} dt\, e^{\imath \la t}\chi(t)\int_{\R} ds \, e^{\imath (\xi+k_1^2) s} \chi(s)\1_{[0,t]}(s)\int_{\R} dt'\, e^{-\imath \la' t'}\chi(t')\int_{\R} ds' \, e^{-\imath (\xi+k_1^2)s'} \chi(s')\1_{[0,t']}(s')\nonumber\\
&=c\, \1_{\{k_1=k_1'\}}\1_{\{\langle k_1\rangle \leq 2^n\}}\int_{\R} \frac{d\xi}{|\xi|^{2H-1}}\int_{\R} ds \int_{\R} ds' \, e^{-\imath (\xi+k_1^2) (s'-s)}\nonumber\\
&\hspace{5cm}\bigg[ \chi(s) \chi(s')\int_{\R} dt\, e^{\imath \la t}\chi(t)\1_{[0,t]}(s) \int_{\R} dt'\, e^{-\imath \la' t'}\chi(t')\1_{[0,t']}(s') \bigg]\nonumber\\
&=c\, \1_{\{k_1=k_1'\}}\1_{\{\langle k_1\rangle \leq 2^n\}}\int_{\R} \frac{d\xi}{|\xi|^{2H-1}}\int_{\R} dr \, e^{-\imath (\xi+k_1^2) r}\nonumber\\
&\hspace{3cm}\bigg[\int_{\R} ds \, \,  \chi(s) \chi(s+r)\int_{\R} dt\, e^{\imath \la t}\chi(t)\1_{[0,t]}(s) \int_{\R} dt'\, e^{-\imath \la' t'}\chi(t') \1_{[0,t']}(s+r)\bigg]\nonumber\\
&=c\, \1_{\{k_1=k_1'\}} \1_{\{\langle k_1\rangle \leq 2^n\}}\int_{\R} \frac{d\xi}{|\xi|^{2H-1}}\cf\big(A_{\la,\la'}\big)(\xi+k_1^2)\nonumber\\
&=c\, \1_{\{k_1=k_1'\}}\1_{\{\langle k_1\rangle \leq 2^n\}}\int_{\R} \frac{d\xi}{|\xi-k_1^2|^{2H-1}}\cf\big(A_{\la,\la'}\big)(\xi),\label{fonct-inter-a}
\end{align}
where we have set
$$A_{\la,\la'}(r):=\int_{\R} ds \, \,  \chi(s) \chi(s+r)\int_{\R} dt\, e^{\imath \la t}\chi(t)\1_{[0,t]}(s) \int_{\R} dt'\, e^{-\imath \la' t'}\chi(t') \1_{[0,t']}(s+r).$$
With this notation, let us define
$$M(\la,\la'):=c\int_{\R} d\xi\, \cf\big(A_{\la,\la'}\big)(\xi)=c \, A_{\la,\la'}(0)$$
and
$$R_{k_1}(\la,\la'):=c\int_{\R}d\xi\, \bigg[ \frac{1}{|\xi-k_1^2|^{2H-1}}-\frac{1}{|k_1^2|^{2H-1}}\bigg]\cf\big(A_{\la,\la'}\big)(\xi),$$
which, going back to \eqref{fonct-inter-a}, yields the decomposition \eqref{decompo-covar}.

\

In order to check \eqref{bound-r-integr}, let us write
\begin{align}
\big|R_{k_1}(\la,\la')\big|\lesssim &\int_{-\infty}^{\frac12 k_1^2}d\xi\, \bigg| \frac{1}{|\xi-k_1^2|^{2H-1}}-\frac{1}{|k_1^2|^{2H-1}}\bigg| \big|\cf\big(A_{\la,\la'}\big)(\xi)\big|\nonumber\\
&\hspace{1cm}+\int_{\frac12 k_1^2}^{\infty}d\xi\, \bigg| \frac{1}{|\xi-k_1^2|^{2H-1}}-\frac{1}{|k_1^2|^{2H-1}}\bigg| \big|\cf\big(A_{\la,\la'}\big)(\xi)\big|.\label{bou-r-0}
\end{align}
On the one hand, for every $\xi\in (-\infty, \frac12 k_1^2)$, one has $|\xi-k_1^2|\geq \frac12 k_1^2$, and so
\begin{align*}
\bigg| \frac{1}{|\xi-k_1^2|^{2H-1}}-\frac{1}{|k_1^2|^{2H-1}}\bigg|&\lesssim \bigg| \frac{1}{|\xi-k_1^2|^{2H-1}}-\frac{1}{|k_1^2|^{2H-1}}\bigg|^\varepsilon\bigg| \frac{1}{|\xi-k_1^2|^{2H-1}}-\frac{1}{|k_1^2|^{2H-1}}\bigg|^{1-\varepsilon}\\
&\lesssim \frac{1}{|k_1^2|^{(2H-1)\varepsilon}} \frac{1}{|k_1^2|^{2H(1-\varepsilon)}} |\xi|^{1-\varepsilon}\lesssim \frac{|\xi|^{1-\varepsilon}}{|k_1|^{4H-2\varepsilon}},
\end{align*}
which entails
\begin{equation}
\int_{-\infty}^{\frac12 k_1^2}d\xi\, \bigg| \frac{1}{|\xi-k_1^2|^{2H-1}}-\frac{1}{|k_1^2|^{2H-1}}\bigg| \big|\cf\big(A_{\la,\la'}\big)(\xi)\big|\lesssim \frac{1}{|k_1|^{4H-2\varepsilon}}\int_{\R}d\xi\, |\xi|^{1-\varepsilon}\big|\cf\big(A_{\la,\la'}\big)(\xi)\big|.\label{bou-r-1}
\end{equation}
On the other hand,
\begin{align}
\int_{\frac12 k_1^2}^{\infty}d\xi\, \bigg| \frac{1}{|\xi-k_1^2|^{2H-1}}-\frac{1}{|k_1^2|^{2H-1}}\bigg| \big|\cf\big(A_{\la,\la'}\big)(\xi\big)\big|&=|k_1|^{4-4H}\int_{\frac12}^{\infty }d\xi\, \bigg| \frac{1}{|\xi-1|^{2H-1}}-1\bigg| \big|\cf\big(A_{\la,\la'}\big)(k_1^2\xi)\big|.\label{bou-r-2}
\end{align}
By injecting \eqref{bou-r-1} and \eqref{bou-r-2} into \eqref{bou-r-0}, we obtain that
\begin{align}
\big|R_{k_1}(\la,\la')\big|&\lesssim \frac{1}{|k_1|^{4H-2\varepsilon}}\int_{\R}d\xi\, |\xi|^{1-\varepsilon}\big|\cf\big(A_{\la,\la'}\big)(\xi)\big|+|k_1|^{4-4H}\int_{\frac12}^{\infty }d\xi\, \bigg| \frac{1}{|\xi-1|^{2H-1}}-1\bigg| \big|\cf\big(A_{\la,\la'}\big)(k_1^2\xi)\big|.\label{major-r-1}
\end{align}
In order to estimate $\cf\big(A_{\la,\la'}\big)$ in the above integrals, observe that
\begin{align*}
A_{\la,\la'}(r)&=\int_{\R} ds \, \,  \chi(s) \chi(s+r)\int_{\R} dt\, e^{\imath \la t}\chi(t)\1_{[0,t]}(s) \int_{\R} dt'\, e^{-\imath \la' t'}\chi(t') \1_{[0,t']}(s+r)\\
&=\int_{\R} ds \, \bigg[\chi(r-s) \int_{\R} dt'\, e^{-\imath \la' t'}\chi(t')\1_{[0,t']}(r-s)\bigg]\bigg[  \chi(-s) \int_{\R} dt\, e^{\imath \la t}\chi(t)\1_{[0,t]}(-s)\bigg]\\
&=\big( f_{\la'} \ast g_\la\big)(r),
\end{align*}
with
$$f_{\la'}(s):=\chi(s) \int_{\R} dt'\, e^{-\imath \la' t'}\chi(t')\1_{[0,t']}(s) \quad \text{and} \quad g_{\la}(s):=\chi(-s) \int_{\R} dt\, e^{\imath \la t}\chi(t)\1_{[0,t]}(-s) .$$
As a result,
$$\big|\cf\big(A_{\la,\la'}\big)(\xi)\big| \lesssim  \big|\cf\big(f_{\la'}\big)(\xi)\big|  \big|\cf\big(g_{\la}\big)(\xi)\big| .$$
Then 
\begin{align*}
\big|\cf\big(f_{\la'}\big)(\xi)\big|&=\bigg| \int_{\R} ds\, e^{-\imath \xi s}\chi(s) \int_{\R} dt'\, e^{-\imath \la' t'}\chi(t')\1_{[0,t']}(s) \bigg|\\
&=\bigg|\int_{\R} dt'\, e^{-\imath \la' t'}\chi(t')\int_0^{t'} ds\, e^{-\imath \xi s}\chi(s)\bigg|\ \lesssim \ \frac{1}{\langle \la'\rangle}\frac{1}{\langle \xi+\la'\rangle},
\end{align*}
{\cop where the last estimate can be easily derived from integration-by-parts arguments}. In the same way,
\begin{align*}
\big|\cf\big(g_{\la}\big)(\xi)\big|&\lesssim \frac{1}{\langle \la\rangle}\frac{1}{\langle \xi+\la\rangle},
\end{align*}
which yields
$$\big|\cf\big(A_{\la,\la'}\big)(\xi)\big|\lesssim \frac{1}{\langle \la\rangle}\frac{1}{\langle \la'\rangle} \frac{1}{\langle \xi+\la\rangle}\frac{1}{\langle \xi+\la'\rangle}.$$
Therefore
\begin{align*}
\int_{\R}d\xi\, |\xi|^{1-\varepsilon}\big|\cf\big(A_{\la,\la'}\big)(\xi)\big|&\lesssim \frac{1}{\langle \la \rangle \langle \la'\rangle}\int_{\R}d\xi\, \frac{|\xi|^{1-\varepsilon}}{\langle \xi+\la'\rangle \langle \xi+\la\rangle},
\end{align*}
while
\begin{align*}
\int_{\frac12}^{\infty }d\xi\, \bigg| \frac{1}{|\xi-1|^{2H-1}}-1\bigg| \big|\cf\big(A_{\la,\la'}\big)(k_1^2\xi)\big|&\lesssim \frac{1}{\langle \la\rangle}\frac{1}{\langle \la'\rangle} \int_{\frac12}^{\infty }d\xi\, \bigg| \frac{1}{|\xi-1|^{2H-1}}-1\bigg|\frac{1}{\langle \la+k_1^2\xi\rangle}\frac{1}{\langle \la'+k_1^2\xi\rangle}.
\end{align*}
Going back to \eqref{major-r-1}, we obtain the estimate
\begin{align*}
&\big|R_{k_1}(\la,\la')\big|\\
&\lesssim \frac{1}{\langle \la\rangle}\frac{1}{\langle \la'\rangle}\bigg[\frac{1}{|k_1|^{4H-2\varepsilon}}\int_{\R}d\xi\, \frac{|\xi|^{1-\varepsilon}}{\langle \xi+\la'\rangle \langle \xi+\la\rangle}+| k_1|^{4-4H}\int_{\frac12}^{\infty }d\xi\, \bigg| \frac{1}{|\xi-1|^{2H-1}}-1\bigg|\frac{1}{\langle \la+k_1^2\xi\rangle}\frac{1}{\langle \la'+k_1^2\xi\rangle}\bigg].
\end{align*}
Once endowed with this bound, the derivation of \eqref{bound-r-integr} easily follows. Namely, for $k_1\neq 0$,
\begin{align*}
\int d\la_1\int \frac{d\la_2}{\langle \la_2\rangle^2}\big|R_{k_1}(\la_1,\la_2)\big|
&\lesssim \frac{1}{|k_1|^{4H-2\varepsilon}}\int_{\R}d\xi\, |\xi|^{1-\varepsilon}\int \frac{d\la_1}{\langle \la_1\rangle} \frac{1}{\langle \la_1+\xi\rangle }\int \frac{d\la_2}{\langle \la_2\rangle^3}\frac{1}{ \langle \la_2+\xi\rangle}\\
&+ \frac{1}{|k_1|^{4H-4}}\int_{\frac12}^{\infty }d\xi\, \bigg| \frac{1}{|\xi-1|^{2H-1}}-1\bigg|\int \frac{d\la_1}{\langle \la_1\rangle} \frac{1}{\langle \la_1+k_1^2\xi\rangle }\int \frac{d\la_2}{\langle \la_2\rangle^3}\frac{1}{ \langle \la_2+k_1^2\xi\rangle},
\end{align*}
{\cop and by applying Lemma \ref{lem:gtv}, we deduce, for $k_1\neq 0$,}
\begin{align*}
\int d\la_1\int \frac{d\la_2}{\langle \la_2\rangle^2}\big|R_{k_1}(\la_1,\la_2)\big|
&\lesssim \frac{1}{|k_1|^{4H-2\varepsilon}}\int_{\R}d\xi\, \frac{|\xi|^{1-\varepsilon}}{\langle \xi\rangle^{2-\frac{\varepsilon}{2}}}+\frac{1}{|k_1|^{4H-4}}\int_{\frac12}^{\infty }d\xi\, \bigg| \frac{1}{|\xi-1|^{2H-1}}-1\bigg| \frac{1}{\langle k_1^2\xi\rangle^{2-\varepsilon} }\\
&\lesssim \frac{1}{|k_1|^{4H-2\varepsilon}} \bigg[\int_{\R}d\xi\, \frac{|\xi|^{1-\varepsilon}}{\langle \xi\rangle^{2-\frac{\varepsilon}{2}}}+\int_{\frac12}^{\infty }\frac{d\xi}{|\xi|^{2-\varepsilon} }\, \bigg| \frac{1}{|\xi-1|^{2H-1}}-1\bigg| \bigg] \lesssim  \frac{1}{|k_1|^{4H-2\varepsilon}}.
\end{align*}

\end{proof}

\section{Proof of Proposition \ref{prop:p-p-q}}\label{sec:item-ii}

\subsection{Notation}

\

\smallskip

As a preliminary observation, note that due to \eqref{cova-frac}, the covariance of the kernel $\ck^{(n)}$ satisfies
\begin{align}
& \mathbb{E} \Big[ \big(\ck^{(n)}_{\chi}\big)_{kk_1}(\la,\la_1) \overline{\big(\ck^{(n)}_{\chi}\big)_{k'k'_1}(\la',\la'_1)}\Big]\nonumber\\
&=\1_{\{k,k'\neq 0\}}\1_{\{k_1\neq k\}}\1_{\{k_1'\neq k'\}}\nonumber\\
&\hspace{1.5cm}\int d\la_2 d\la_2' \, \ci_\chi(\la,\Omega_{k,k_1-k}+\la_1-\la_2)\overline{\ci_\chi(\la',\Omega_{k',k'_1-k'}+\la'_1-\la'_2)}\mathbb{E}\Big[ \overline{\cf(\<Psi>^{(n)}_{k_1-k})(\la_2)}\cf(\<Psi>^{(n)}_{k'_1-k'})(\la_2')\Big]\nonumber\\
&=\1_{\{ k_1-k=k_1'-k'\}}\mathbb{E} \Big[ \big(\ck^{(n)}_{\chi}\big)_{kk_1}(\la,\la_1) \overline{\big(\ck^{(n)}_{\chi}\big)_{k'k'_1}(\la',\la'_1)}\Big].\label{support-k}
\end{align}

\smallskip

For future reference, let us also rephrase the above expression (when $k=k',k_1=k_1'$) using \eqref{cova-frac}, which gives
\begin{align}
& \mathbb{E} \Big[ \big(\ck^{(n)}_{\chi}\big)_{kk_1}(\la,\la_1) \overline{\big(\ck^{(n)}_{\chi}\big)_{kk_1}(\la',\la'_1)}\Big]\nonumber\\
&=\1_{\{k\neq 0\}}\1_{\{k_1\neq k\}}\nonumber\\
&\hspace{1cm}\int d\la_2 d\la_2' \, \ci_\chi(\la,\Omega_{k,k_1-k}+\la_1-\la_2)\overline{\ci_\chi(\la',\Omega_{k,k_1-k}+\la'_1-\la'_2)}\mathbb{E}\Big[ \overline{\cf(\<Psi>^{(n)}_{k_1-k})(\la_2)}\cf(\<Psi>^{(n)}_{k_1-k})(\la_2')\Big]\nonumber\\
&=\1_{\{k\neq 0\}}\1_{\{k_1\neq k\}}\1_{\{\langle k_1-k\rangle \leq 2^n\}}\int_{\R} \frac{d\xi}{|\xi|^{2H-1}}\, A^\xi_{(k,\la)}(k_1,\la_1)\overline{A^\xi_{(k,\color{blue}\la'\color{black})}(k_1,\la_1')}\label{e-k-k-2-ter},
\end{align}
where we have set
\begin{align*}
A^\xi_{(k,\la)}(k_1,\la_1):=\int d\la_2 \, \ci_\chi(\la,\Omega_{k,k_1-k}+\la_1-\la_2)\overline{\ci_\chi(\la_2,-\xi-(k_1-k)^2)}.
\end{align*}

\subsection{Proof of Proposition \ref{prop:p-p-q}.}

\

We fix $b,c\in (0,1)$ such that $b-c\leq \frac14$ and consider the related quantity ${\cop \cpti^{(n)}_{c,b}}$ in \eqref{defi:cp-n}. Using Jensen's inequality, we immediately obtain that

\begin{align*}
&\mathbb{E}\Big[{\cop \cpti^{(n)}_{c,b}}\Big]\\
&=\sum_{k_1,k_1'\in \Z}\int_{\R^2} \frac{d\la_1}{\langle k_1 \rangle^{2c}+\langle \la_1\rangle^{2b}}\frac{d\la_1'}{\langle k'_1 \rangle^{2c}+\langle \la'_1\rangle^{2b}} \mathbb{E}\bigg[\bigg| \sum_{k\in \Z}\int_{\R} d\la\, \{ \langle k \rangle^{2c}+\langle \la\rangle^{2b}\} \,\big(\ck^{(n)}_{\chi}\big)_{kk_1}(\la,\la_1) \overline{\big(\ck^{(n)}_{\chi}\big)_{kk_1'}(\la,\la_1')}\bigg|^2\bigg]\\
&\geq \sum_{k_1,k_1'\in \Z}\int_{\R^2} \frac{d\la_1}{\langle k_1 \rangle^{2c}+\langle \la_1\rangle^{2b}}\frac{d\la_1'}{\langle k'_1 \rangle^{2c}+\langle \la'_1\rangle^{2b}} \bigg| \sum_{k\in \Z}\int_{\R} d\la\, \{ \langle k \rangle^{2c}+\langle \la\rangle^{2b}\} \,\mathbb{E}\bigg[ \big(\ck^{(n)}_{\chi}\big)_{kk_1}(\la,\la_1) \overline{\big(\ck^{(n)}_{\chi}\big)_{kk_1'}(\la,\la_1')}\bigg]\bigg|^2\\
&\hspace{3cm}=:\mathfrak{P}^{(n)}_{c,b}.
\end{align*}

Then, thanks to \eqref{support-k}, we can assert that
\begin{align*}
&\mathfrak{P}^{(n)}_{c,b}=\sum_{k_1\in \Z}\int_{\R^2} \frac{d\la_1}{\langle k_1 \rangle^{2c}+\langle \la_1\rangle^{2b}}\frac{d\la_1'}{\langle k_1 \rangle^{2c}+\langle \la'_1\rangle^{2b}} \bigg|\sum_{k\in \Z}\int_{\R} d\la\, \{ \langle k \rangle^{2c}+\langle \la\rangle^{2b}\} \,\mathbb{E}\bigg[ \big(\ck^{(n)}_{\chi}\big)_{kk_1}(\la,\la_1) \overline{\big(\ck^{(n)}_{\chi}\big)_{kk_1}(\la,\la_1')}\bigg]\bigg|^2\\
&=\sum_{k_1\in \Z}\int_{\R^2} \frac{d\la_1}{\langle k_1 \rangle^{2c}+\langle \la_1\rangle^{2b}}\frac{d\la_1'}{\langle k_1 \rangle^{2c}+\langle \la'_1\rangle^{2b}}\sum_{k,k'\in \Z}\int_{\R^2} d\la d\la'\, \{ \langle k \rangle^{2c}+\langle \la\rangle^{2b}\}\{ \langle k' \rangle^{2c}+\langle \la'\rangle^{2b}\} \\
&\hspace{2cm}\mathbb{E}\bigg[ \big(\ck^{(n)}_{\chi}\big)_{kk_1}(\la,\la_1) \overline{\big(\ck^{(n)}_{\chi}\big)_{kk_1}(\la,\la_1')}\bigg]\mathbb{E}\bigg[ \big(\ck^{(n)}_{\chi}\big)_{k'k_1}(\la',\la_1) \overline{\big(\ck^{(n)}_{\chi}\big)_{k'k_1}(\la',\la_1')}\bigg],
\end{align*}
which, combined with \eqref{e-k-k-2-ter}, yields that
\begin{align*}
\mathfrak{P}^{(n)}_{c,b}&=\sum_{k,k'\neq 0}\int_{\R^2} d\la d\la'\, \{ \langle k \rangle^{2c}+\langle \la\rangle^{2b}\}\{ \langle k' \rangle^{2c}+\langle \la'\rangle^{2b}\} \sum_{k_1\notin \{k,k'\}}\1_{\{\langle k_1-k\rangle \leq 2^n\}}\1_{\{\langle k_1-k'\rangle \leq 2^n\}} \\
&\hspace{2cm}\int_{\R} \frac{d\la_1}{\langle k_1 \rangle^{2c}+\langle \la_1\rangle^{2b}}\int_{\R}\frac{d\la_1'}{\langle k_1 \rangle^{2c}+\langle \la'_1\rangle^{2b}}\int_{\R} \frac{d\xi}{|\xi|^{2H-1}}\, A^\xi_{(k,\la)}(k_1,\la_1)\overline{A^\xi_{(k,\la)}(k_1,\la_1')}\\
&\hspace{4cm}\int_{\R} \frac{d\xi'}{|\xi'|^{2H-1}}\, \overline{A^{\xi'}_{(k',\la')}(k_1,\la_1)}A^{\xi'}_{(k',\la')}(k_1,\la_1')\\
&=\sum_{k,k'\neq 0}\int_{\R^2} d\la d\la'\, \{ \langle k \rangle^{2c}+\langle \la\rangle^{2b}\}\{ \langle k' \rangle^{2c}+\langle \la'\rangle^{2b}\} \sum_{k_1\notin \{k,k'\}}\1_{\{\langle k_1-k\rangle \leq 2^n\}}\1_{\{\langle k_1-k'\rangle \leq 2^n\}}\\
&\hspace{2cm}\int_{\R} \frac{d\xi}{|\xi|^{2H-1}}\int_{\R} \frac{d\xi'}{|\xi'|^{2H-1}}\bigg|\int_{\R} \frac{d\la_1}{\langle k_1 \rangle^{2c}+\langle \la_1\rangle^{2b}}A^\xi_{(k,\la)}(k_1,\la_1)\overline{A^{\xi'}_{(k',\la')}(k_1,\la_1)}\bigg|^2.
\end{align*}
As a result,
\begin{align*}
\mathfrak{P}^{(n)}_{c,b}
&\geq \sum_{k\neq 0}\langle k \rangle^{4c}\sum_{k_1\neq k}\1_{\{\langle k_1-k\rangle \leq 2^n\}} \\
&\hspace{1.5cm}\int_{\R^2} d\la d\la'\int_{\R} \frac{d\xi}{|\xi|^{2H-1}}\int_{\R} \frac{d\xi'}{|\xi'|^{2H-1}}\bigg|\int_{\R} \frac{d\la_1}{\langle k_1 \rangle^{2c}+\langle \la_1\rangle^{2b}}A^\xi_{(k,\la)}(k_1,\la_1)\overline{A^{\xi'}_{(k,\la')}(k_1,\la_1)}\bigg|^2\\
&\geq \sum_{k\neq 0}\langle k \rangle^{4c}\sum_{k_1\neq k}\1_{\{\langle k_1-k\rangle \leq 2^n\}} \int_{\R} \frac{d\la_1}{\langle k_1 \rangle^{2c}+\langle \la_1\rangle^{2b}}\int_{\R} \frac{d\la'_1}{\langle k_1 \rangle^{2c}+\langle \la'_1\rangle^{2b}} \\
&\hspace{1cm}\bigg[\int_{\R} \frac{d\xi}{|\xi|^{2H-1}}\int_{\R} d\la\, A^\xi_{(k,\la)}(k_1,\la_1)\overline{A^\xi_{(k,\la)}(k_1,\la'_1)}\bigg]\bigg[\int_{\R} \frac{d\xi'}{|\xi'|^{2H-1}} \int_{\R}d\la'\, \overline{A^{\xi'}_{(k,\la')}(k_1,\la_1)}A^{\xi'}_{(k,\la')}(k_1,\la'_1)\bigg]\\
&\geq \sum_{k\neq 0}\langle k \rangle^{4c}\sum_{k_1\neq k}\1_{\{\langle k_1-k\rangle \leq 2^n\}} \int_{\R} \frac{d\la_1}{\langle k_1 \rangle^{2c}+\langle \la_1\rangle^{2b}}\int_{\R} \frac{d\la'_1}{\langle k_1 \rangle^{2c}+\langle \la'_1\rangle^{2b}} \\
&\hspace{5cm}\bigg|\int_{\R} \frac{d\xi}{|\xi|^{2H-1}}\int_{\R} d\la\, A^\xi_{(k,\la)}(k_1,\la_1)\overline{A^\xi_{(k,\la)}(k_1,\la'_1)}\bigg|^2.
\end{align*}
In turn, this entails that
\begin{align}
&\mathfrak{P}^{(n)}_{c,b}
\geq \sum_{ 2\leq k\leq 2^n} \langle k \rangle^{4c}\int_{{\cop 2k}-\frac12}^{{\cop 2k}+\frac12} \frac{d\la_1}{\langle k-1\rangle^{2c}+\langle \la_1\rangle^{2b}}\int_{{\cop 2k}-\frac12}^{{\cop 2k}+\frac12} \frac{d\la_1'}{\langle k-1 \rangle^{2c}+\langle \la'_1\rangle^{2b}}\nonumber\\
&\hspace{5cm}\bigg|\int_{\R} \frac{d\xi}{|\xi|^{2H-1}}\int_{\R} d\la\, A^\xi_{(k,\la)}(k-1,\la_1)\overline{A^\xi_{(k,\la)}(k-1,\la'_1)}\bigg|^2\nonumber\\
&\gtrsim \sum_{ 2\leq k\leq 2^n} \frac{\langle k \rangle^{4c}}{\langle k\rangle^{4c}+\langle k\rangle^{4b}}\int_{{\cop 2k}-\frac12}^{{\cop 2k}+\frac12} d\la_1 \int_{{\cop 2k}-\frac12}^{{\cop 2k}+\frac12}d\la_1'\,\bigg|\int_{\R} \frac{d\xi}{|\xi|^{2H-1}}\int_{\R} d\la\, A^\xi_{(k,\la)}(k-1,\la_1)\overline{A^\xi_{(k,\la)}(k-1,\la'_1)}\bigg|^2\nonumber\\
&\gtrsim \sum_{ 2\leq k\leq 2^n} \frac{\langle k \rangle^{4c}}{\langle k\rangle^{4c}+\langle k\rangle^{4b}}\int_{-\frac12}^{\frac12} d\la_1 \int_{-\frac12}^{\frac12}d\la_1'\,\bigg|\int_{\R} \frac{d\xi}{|\xi|^{2H-1}}\int_{\R} d\la\, A^\xi_{(k,\la)}(k-1,\la_1 {\cop +2k})\overline{A^\xi_{(k,\la)}(k-1,\la'_1 {\cop + 2k})}\bigg|^2.\label{correct-i}
\end{align}
At this point, observe that due to $\Omega_{k,-1}={\cop -2k}$, one has in fact
\begin{align*}
A^\xi_{(k,\la)}(k-1,\la_1 {\cop +2k})&=\int d\la_2 \, \ci_\chi(\la,\Omega_{k,-1}+\la_1 {\cop +2k}-\la_2)\overline{\ci_\chi(\la_2,-\xi-1)}\\
&=\int d\la_2 \, \ci_\chi(\la,\la_1-\la_2)\overline{\ci_\chi(\la_2,-\xi-1)}=:B(\xi,\la,\la_1),
\end{align*}
and therefore, going back to \eqref{correct-i}, we deduce that
\begin{align*}
\mathfrak{P}^{(n)}_{c,b}
&\gtrsim \bigg(\sum_{ 2\leq k\leq 2^n} \frac{\langle k \rangle^{4c}}{\langle k\rangle^{4c}+\langle k\rangle^{4b}}\bigg)\bigg(\int_{-\frac12}^{\frac12} d\la_1 \int_{-\frac12}^{\frac12}d\la_1'\,\bigg|\int_{\R} \frac{d\xi}{|\xi|^{2H-1}}\int_{\R} d\la\, B(\xi,\la,\la_1)\overline{B(\xi,\la,\la_1')}\bigg|^2\bigg).
\end{align*}

Recall now that $b-c\leq \frac14$. Thus, in order to ensure that $\mathfrak{P}^{(n)}_{c,b} \stackrel{n\to\infty}{\longrightarrow}\infty$, we only need to guarantee that
\begin{equation}\label{last-gar}
\int_{-\frac12}^{\frac12} d\la_1 \int_{-\frac12}^{\frac12}d\la_1'\,\bigg|\int_{\R} \frac{d\xi}{|\xi|^{2H-1}}\int_{\R} d\la\, B(\xi,\la,\la_1)\overline{B(\xi,\la,\la_1')}\bigg|^2 >0.
\end{equation}
To this end, observe for instance that for $(\la_1,\la_1')=(0,0)$,
\begin{align*}
&\int_{\R} \frac{d\xi}{|\xi|^{2H-1}}\int_{\R} d\la\, B(\xi,\la,0)\overline{B(\xi,\la,0)}=\int_{\R} \frac{d\xi}{|\xi|^{2H-1}}\int_{\R} d\la\, \big|B(\xi,\la,0)\big|^2,
\end{align*}
and then, for $(\xi,\la)=(-1,0)$,
\begin{align*}
B(-1,0,0)&=\int d\la_2 \, \ci_\chi(0,-\la_2)\overline{\ci_\chi(\la_2,0)}\\
&=\int d\la_2\, \bigg(\int_{\R} dt\,  \chi(t)\int_0^t ds \, e^{\cop -\imath \la_2s}\chi(s)\bigg) \bigg(\int_{\R} dt' \, e^{\imath \la_2 t'} \chi(t')\int_0^{t'} ds' \,\chi(s')\bigg)\\
&={\cop \int_{\R} dt\,  \chi(t)\int_0^t ds \, \chi(s)\int_{\R} dt' \,  \chi(t')\int_0^{t'} ds' \,\chi(s')\bigg(\int d\la_2\, e^{ -\imath \la_2(s-t')}\bigg) }\\
&={\cop \int_{\R} dt\,  \chi(t)\int_0^t ds \, {\cop\chi(s)^2}\int_0^{{\cop s}} ds' \,\chi(s')  >0.}
\end{align*}

\smallskip

For obvious continuity reasons, we deduce that \eqref{last-gar} is satisfied, and accordingly one has
\begin{equation*}
\mathfrak{P}^{(n)}_{c,b} \stackrel{n\to\infty}{\longrightarrow}\infty,
\end{equation*}
as desired.

\bigskip

\appendix

\section{Proof of Lemma \ref{lem:regu-luxo}, item $(i)$}\label{sec:proof-regu-luxo}

We know that the linear solution $\Psi$ is explicitly given by the convolution formula
$$\Psi(t)=\int_0^t e^{-\imath {\cop \partial^2_x} (t-s)}\dot{B}(s),$$
and hence, using the notation introduced in \eqref{not:fouri-space},
\begin{align*}
\mathbb{E}\Big[ \big\| \Psi\big\|_{L^2([0,T]\times \mathbb{T})}^2\Big]&=\sum_k \int_0^T \mathbb{E}\Big[\big| \Psi_k(t)\big|^2\Big] \, dt\\
&=\sum_k \int_0^T \mathbb{E}\bigg[\bigg| \int_0^t e^{-\imath k^2 s}\dot{\beta}^{(k)}_{t-s} \, ds\bigg|^2\bigg] \, dt=\sum_k \int_0^T dt \int_0^t ds\int_0^t ds' e^{-\imath k^2 (s-s')}  \mathbb{E}\big[\dot{\beta}^{(k)}_{t-s} \dot{\beta}^{(k)}_{t-s'} \big].
\end{align*}
Recall that the fractional covariance is given by $ \mathbb{E}\big[\dot{\beta}^{(k)}_s \dot{\beta}^{(k)}_{s'} \big]=|s-s'|^{2H-2}$, which gives
\begin{align}
\mathbb{E}\Big[ \big\| \Psi\big\|_{L^2([0,T]\times \mathbb{T})}^2\Big]&=\sum_k \int_0^T dt\int_0^t ds\int_0^t ds' \, \frac{e^{-\imath k^2 (s-s')}}{|s-s'|^{2-2H}}=\sum_k \int_0^T dt\int_0^t ds\int_{-s}^{t-s} dr\, \frac{e^{\imath k^2 r}}{|r|^{2-2H}}\nonumber\\
&=\sum_k \int_0^T dt\int_0^t ds\int_{-s}^{0} dr\,  \frac{e^{\imath k^2 r}}{|r|^{2-2H}}+\sum_k \int_0^T dt\int_0^t ds\int_{0}^{t-s} dr\, \frac{e^{\imath k^2 r}}{|r|^{2-2H}}\nonumber\\
&=\sum_k \int_0^T dt\int_0^t ds\int_{0}^{s} dr\, \frac{e^{-\imath k^2 r}}{|r|^{2-2H}}+\sum_k \int_0^T dt\int_0^t ds\int_{0}^{s} dr \,\frac{e^{\imath k^2 r}}{|r|^{2-2H}}\nonumber\\
&=2\sum_k \int_0^T dt\int_0^t ds\int_{0}^{s} dr\, \frac{\cos(k^2r)}{r^{2-2H}}\nonumber\\
&=2\int_0^T dt\int_0^t ds\int_{0}^{s} \frac{dr}{r^{2-2H}}+2\sum_{k\neq 0}\frac{1}{|k|^{4H-2}} \int_0^T dt\int_0^t ds\int_{0}^{k^2 s} dr \, \frac{\cos(r)}{r^{2-2H}}\nonumber\\
&\lesssim 1+2\, \bigg|\sum_{k\neq 0}\frac{1}{|k|^{4H-2}} \int_0^T dt\int_0^t ds\int_{0}^{k^2 s} dr \, \frac{\cos(r)}{r^{2-2H}}\bigg|.\label{zero}
\end{align}
Now observe that
\begin{align}
&\bigg|\sum_{k\neq 0}\frac{1}{|k|^{4H-2}} \int_0^T dt\int_0^t ds\int_{0}^{k^2 s} dr\,  \frac{\cos(r)}{r^{2-2H}}\bigg|\nonumber\\
&\lesssim \sum_{k\neq 0:|k|\leq \frac{1}{\sqrt{T}}}\frac{1}{|k|^{4H-2}} \int_0^T dt\int_0^t ds\int_{0}^{1} \frac{dr}{r^{2-2H}}+\sum_{k\neq 0:|k|\geq \frac{1}{\sqrt{T}}}\frac{1}{|k|^{4H-2}} \int_0^T dt\int_0^T ds\, \bigg|\int_{0}^{k^2 s} dr \, \frac{\cos(r)}{r^{2-2H}}\bigg|\nonumber\\
&\lesssim \sum_{k\neq 0}\frac{1}{|k|^{4H-2}}+\sum_{k\neq 0:|k|\geq \frac{1}{\sqrt{T}}}\frac{1}{|k|^{4H-2}} \int_0^T dt \int_0^T ds\, \bigg|\int_{0}^{k^2 s} dr \, \frac{\cos(r)}{r^{2-2H}}\bigg|.\label{un}
\end{align}
Then
\begin{align}
&\sum_{k\neq 0:|k|\geq \frac{1}{\sqrt{T}}}\frac{1}{|k|^{4H-2}} \int_0^T dt\int_0^T ds\, \bigg|\int_{0}^{k^2 s} dr\,  \frac{\cos(r)}{r^{2-2H}}\bigg|\nonumber\\
&\lesssim \sum_{k\neq 0:|k|\geq \frac{1}{\sqrt{T}}}\frac{1}{|k|^{4H-2}} \int_0^T dt\int_0^{\frac{1}{k^2}} ds\int_{0}^{1} \frac{dr}{r^{2-2H}}+\sum_{k\neq 0:|k|\geq \frac{1}{\sqrt{T}}}\frac{1}{|k|^{4H-2}} \int_0^T dt\int_{\frac{1}{k^2}}^T ds\bigg[\int_{0}^{1} \frac{dr}{r^{2-2H}}+\bigg|\int_{1}^{k^2 s} dr\, \frac{\cos(r)}{r^{2-2H}}\bigg|\bigg]\nonumber\\
&\lesssim \sum_{k\neq 0}\frac{1}{|k|^{4H-2}} +\sum_{k\neq 0:|k|\geq \frac{1}{\sqrt{T}}}\frac{1}{|k|^{4H-2}} \int_0^T dt\int_{\frac{1}{k^2}}^T ds\,\bigg|\int_{1}^{k^2 s} dr\, \frac{\cos(r)}{r^{2-2H}}\bigg|\label{deux}
\end{align}
and for every $s>\frac{1}{k^2}$, we can use an integration-by-parts argument to derive that
\begin{align}
\bigg|\int_{1}^{k^2 s} dr\, \frac{\cos(r)}{r^{2-2H}}\bigg|&\lesssim 1+\int_{1}^{k^2 s}dr\, \frac{|\sin(r)|}{r^{3-2H}}\lesssim 1+\int_{1}^{\infty} \frac{dr}{r^{3-2H}}\lesssim 1.\label{trois}
\end{align}
Injecting \eqref{un}, \eqref{deux} and \eqref{trois} into \eqref{zero}, we can conclude that
\begin{align*}
\mathbb{E}\Big[ \big\| \Psi\big\|_{L^2([0,T]\times \mathbb{T})}^2\Big]&\lesssim 1+\sum_{k\neq 0}\frac{1}{|k|^{4H-2}} <\infty,
\end{align*}
due to the assumption $H>\frac34$.

\

\color{blue}

\section{Control of $\ci_\chi\cm(z,z)$}\label{sec:z-z-ztilde}

\begin{proposition}\label{prop:control-m-z-z-ztilde}
For all $\frac12 < b<\frac34$ and $0<c\leq b$, one has
\begin{equation}\label{control-m-z-z}
\big\|  \ci_\chi \cm(z,w)\big\|_{\widetilde{Z}^{c,b}}\lesssim \| z\|_{\widetilde{Z}^{c,b}}\| w\|_{\widetilde{Z}^{c,b}}.
\end{equation}
\end{proposition}

For the sake of clarity, we will treat the controls in $\widetilde{Z}^{0,b}$ and $\widetilde{Z}^{c,0}$ separately.

\subsection{Control in $\widetilde{Z}^{0,b}$}

\

\smallskip

\begin{proposition}\label{prop:m-z-z}
For all $\frac12 < b<\frac34$, one has
\begin{equation}\label{m-z-z}
\big\| \ci_\chi \cm(z,w)\big\|_{\widetilde{Z}^{0,b}}\lesssim \| z\|_{\widetilde{Z}^{0,b}}\| w\|_{\widetilde{Z}^{0,b}}.
\end{equation}

\end{proposition}

\begin{proof}

With the notation in \eqref{i-chi-fou} and \eqref{cm}, we can first write
\begin{align*}
&\cf\big( \ci_{\chi} \cm(z,w)_k\big)(\la)=\int_{\R} d\la' \, \ci_\chi(\la,\la') \cf(\cm(z,w)_k)(\la')\\
&=\sum_{k_1} \int_{\R} d\la' \, \ci_\chi(\la,\la') \int dt\, e^{-\imath \la' t}e^{\imath t \Omega_{k,k_1}} z_{k+k_1}(t) \overline{w_{k_1}(t)}\\
&=\sum_{k_1}\int d\la_1 \, \overline{\cf(w_{k_1})(\la_1)} \int d\la_2 \, \cf(z_{k+k_1})(\la_2) \int_{\R} d\la' \, \ci_\chi(\la,\la')\int dt\, e^{-\imath \la' t}e^{\imath t\Omega_{k,k_1}} e^{\imath t\la_2}e^{-\imath t \la_1} 
\end{align*}
and so
\begin{align}
&\cf\big( \ci_{\chi} \cm(z,w)_k\big)(\la)=\sum_{k_1}\int d\la_1 \, \overline{\cf(w_{k_1})(\la_1)} \int d\la_2 \, \cf(z_{k+k_1})(\la_2) \ci_\chi(\la,\Omega_{k,k_1}+\la_2-\la_1).\label{conven-exp}
\end{align}
Based on this expression, one has
\begin{align*}
&\big\| \ci_\chi \cm(z,w)\big\|_{\widetilde{Z}^{0,b}}^2\\
&=\sum_k\int d\la \, \langle \la\rangle^{2b}\bigg| \sum_{k_1}\int d\la_1 \, \overline{\cf(w_{k_1})(\la_1)} \int d\la_2 \, \cf(z_{k+k_1})(\la_2) \ci_\chi(\la,\Omega_{k,k_1}+\la_2-\la_1)\bigg|^2\\
&=\sum_{k\neq 0}\int d\la \, \langle \la\rangle^{2b}\bigg| \sum_{k_1}\int d\la_1 \, \overline{\cf(w_{k_1})(\la_1)} \int d\la_2 \, \cf(z_{k+k_1})(\la_2) \ci_\chi(\la,\Omega_{k,k_1}+\la_2-\la_1)\bigg|^2\\
&\hspace{1cm}+\int d\la \, \langle \la\rangle^{2b}\bigg| \sum_{k_1}\int d\la_1 \, \overline{\cf(w_{k_1})(\la_1)} \int d\la_2 \, \cf(z_{k_1})(\la_2) \ci_\chi(\la,\la_2-\la_1)\bigg|^2\ =:\ I +II.
\end{align*}

\

\noindent
\textit{Estimate for $I$.} One has
\begin{align*}
&I=\sum_{k\neq 0}\int d\la \, \langle \la\rangle^{2b}\bigg| \sum_{k_1}\int d\la_1 \int d\la_2 \, \overline{\cf(w_{k_1})(\la_1)}  \cf(z_{k+k_1})(\la_2) \ci_\chi(\la,\Omega_{k,k_1}+\la_2-\la_1)\bigg|^2\\
&=\sum_{k\neq 0}\bigg[ \sum_{k_1,k_1'}\int d\la_1 d\la'_1\int d\la_2d\la'_2\\
&\hspace{1cm}  \Big(\big[\langle\la_1\rangle^{b}\overline{\cf(w_{k_1})(\la_1)}\big]\big[ \langle\la_2\rangle^{b} \cf(z_{k+k_1})(\la_2)\big]\big[\langle\la'_1\rangle^{b}\cf(w_{k'_1})(\la'_1)\big]\big[ \langle\la'_2\rangle^{b} \overline{\cf(z_{k+k_1'})(\la'_2)}\big]\Big)\\
&\hspace{1.2cm}\bigg(\frac{1}{\langle\la_1\rangle^{b}} \frac{1}{\langle\la'_1\rangle^{b}} \frac{1}{\langle\la_2\rangle^{b}}\frac{1}{\langle\la'_2\rangle^{b}}\int d\la \, \langle \la\rangle^{2b}\ci_\chi(\la,\Omega_{k,k_1}+\la_2-\la_1) \overline{\ci_\chi(\la,\Omega_{k,k'_1}+\la'_2-\la'_1)}\bigg)\bigg],
\end{align*}
and so, using the Cauchy-Schwarz inequality with respect to $(k_1,k_1',\la_1,\la_1',\la_2,\la_2')$, we get
\begin{align*}
&I\leq \sum_{k\neq 0}\bigg( \sum_{k_1,k_1'}\int d\la_1 d\la'_1\int d\la_2d\la'_2\,  \Big|\big[\langle\la_1\rangle^{b}\cf(w_{k_1})(\la_1)\big]\big[ \langle\la_2\rangle^{b} \cf(z_{k+k_1})(\la_2)\big]\\
&\hspace{7cm} \big[\langle\la'_1\rangle^{b}\cf(w_{k'_1})(\la'_1)\big]\big[ \langle\la'_2\rangle^{b} \cf(z_{k+k_1'})(\la'_2)\big]\Big|^2\bigg)^{\frac12}\\
&\bigg(\sum_{k_1,k_1'}\int \frac{d\la_1}{\langle\la_1\rangle^{2b}} \frac{d\la'_1}{\langle\la'_1\rangle^{2b}}\int \frac{d\la_2}{\langle\la_2\rangle^{2b}}\frac{d\la'_2}{\langle\la'_2\rangle^{2b}}\, \bigg|\int d\la \, \langle \la\rangle^{2b}\ci_\chi(\la,\Omega_{k,k_1}+\la_2-\la_1) \overline{\ci_\chi(\la,\Omega_{k,k'_1}+\la'_2-\la'_1)}\bigg|^2\bigg)^{\frac12}.
\end{align*}
Setting 
$$\mathbf{C}:=\sup_{k\neq 0} \sup_{\la_1,\la_2\in \R} \sup_{\la'_1,\la'_2\in \R}\sum_{k_1,k_1'} \bigg(\int d\la \, \langle \la\rangle^{2b}|\ci_\chi(\la,\Omega_{k,k_1}+\la_2-\la_1)| |\ci_\chi(\la,\Omega_{k,k'_1}+\la'_2-\la'_1)|\bigg)^2$$
and using the fact that $b>\frac12$, we obtain that
\begin{align*}
I&\leq \mathbf{C}^{\frac12}\cdot  \sum_{k}\bigg( \sum_{k_1}\int d\la_1\, \langle\la_1\rangle^{2b}\big|\cf(w_{k_1})(\la_1)\big|^2 \int d\la_2\,   \langle\la_2\rangle^{2b} \big|\cf(z_{k+k_1})(\la_2)\big|^2\bigg) \leq \mathbf{C}^{\frac12} \, \|z\|_{\widetilde{Z}^{0,b}}^2\|w\|_{\widetilde{Z}^{0,b}}^2.
\end{align*}
Therefore it remains us to prove that $\mathbf{C}<\infty$.

\smallskip

To this end, let us first use \eqref{ker-i} to write, for all $k\neq 0$ and $\la_1,\la_2\in \R$,
\begin{align}
&\sum_{k_1,k_1'} \bigg(\int d\la \, \langle \la\rangle^{2b}|\ci_\chi(\la,\Omega_{k,k_1}+\la_2-\la_1)| |\ci_\chi(\la,\Omega_{k,k'_1}+\la'_2-\la'_1)|\bigg)^2\nonumber\\
&\lesssim \sum_{k_1,k_1'} \bigg(\int \frac{d\la}{\langle \la\rangle^{2-2b}}\frac{1}{\langle \la-\Omega_{k,k_1}-\la_2+\la_1\rangle}\frac{1}{\langle \la-\Omega_{k,k'_1}-\la'_2+\la'_1\rangle}\bigg)^2\nonumber\\
&\lesssim \int \frac{d\la}{\langle \la\rangle^{2-2b}}\int \frac{d\la'}{\langle \la'\rangle^{2-2b}}\bigg[\sum_{k_1}\frac{1}{\langle \la-\Omega_{k,k_1}-\la_2+\la_1\rangle}\frac{1}{\langle \la'-\Omega_{k,k_1}-\la_2+\la_1\rangle}\bigg]\nonumber\\
&\hspace{5cm}\bigg[\sum_{k_1'}\frac{1}{\langle \la-\Omega_{k,k'_1}-\la'_2+\la'_1\rangle}\frac{1}{\langle \la'-\Omega_{k,k'_1}-\la'_2+\la'_1\rangle}\bigg].\label{interm-product}
\end{align}
Then recall that $\Omega_{k,k_1}=2kk_1$, and so
\begin{align*}
&\sum_{k_1}\frac{1}{\langle \la-\Omega_{k,k_1}-\la_2+\la_1\rangle}\frac{1}{\langle \la'-\Omega_{k,k_1}-\la_2+\la_1\rangle}\\
&=\sum_{k_1}\frac{1}{\langle \la-2kk_1-\la_2+\la_1\rangle}\frac{1}{\langle \la'-2kk_1-\la_2+\la_1\rangle}\leq \sum_{k_1}\frac{1}{\langle \la+k_1-\la_2+\la_1\rangle}\frac{1}{\langle \la'+k_1-\la_2+\la_1\rangle},
\end{align*}
where we have used the fact that $k\neq 0$ to derive the last inequality. We are here in a position to apply Lemma \ref{lem:gtv} and assert that for every $0<\varepsilon<1$,
$$\sum_{k_1}\frac{1}{\langle \la-\Omega_{k,k_1}-\la_2+\la_1\rangle}\frac{1}{\langle \la'-\Omega_{k,k_1}-\la_2+\la_1\rangle} \lesssim \frac{1}{\langle \la-\la'\rangle^{1-\varepsilon}},$$
where the proportional constant does not depend on $k$, $\la_1$ or $\la_2$. Injecting this uniform bound into \eqref{interm-product} now yields
\begin{align*}
&\sum_{k_1,k_1'} \bigg(\int d\la \, \langle \la\rangle^{2b}|\ci_\chi(\la,\Omega_{k,k_1}+\la_2-\la_1)| |\ci_\chi(\la,\Omega_{k,k'_1}+\la'_2-\la'_1)|\bigg)^2\lesssim \int \frac{d\la}{\langle \la\rangle^{2-2b}}\int \frac{d\la'}{\langle \la'\rangle^{2-2b}}\frac{1}{\langle \la-\la'\rangle^{2-2\varepsilon}},
\end{align*}
uniformly over $k\neq 0$ and $\la_1,\la_2,\la'_1,\la'_2\in \R$.

\smallskip

Finally, we can again rely on Lemma \ref{lem:gtv} to deduce that
\begin{align*}
&\mathbf{C}\lesssim \int \frac{d\la}{\langle \la\rangle^{4-4b}},
\end{align*}
and since $b<\frac34$, this achieves to prove that $\mathbf{C}<\infty$.

\

\noindent
\textit{Estimate for $II$.} Using \eqref{ker-i}, we can write
\begin{align*}
&II=\int d\la \, \langle \la\rangle^{2b}\bigg| \sum_{k_1}\int d\la_1 \, \overline{\cf(w_{k_1})(\la_1)} \int d\la_2 \, \cf(z_{k_1})(\la_2) \ci_\chi(\la,\la_2-\la_1)\bigg|^2\\
&\lesssim \int \frac{d\la}{\langle \la\rangle^{2-2b}}\bigg[ \sum_{k_1}\int d\la_1\int d\la_2 \bigg( \big[ \langle \la_1\rangle^b|\cf(w_{k_1})(\la_1)|\big]\big[\langle \la_2\rangle^b | \cf(z_{k_1})(\la_2)|\big]\bigg) \bigg( \frac{1}{\langle \la_1\rangle^b} \frac{1}{\langle \la_2\rangle^b}\frac{1}{\langle \la-(\la_1-\la_2)\rangle}\bigg)\bigg]^2,
\end{align*}
and then, by the Cauchy-Schwarz inequality (with respect to $(\la_1,\la_2)$), we obtain
\begin{align*}
II&\lesssim \bigg(\int \frac{d\la}{\langle \la\rangle^{2-2b}}\frac{d\la'_1}{\langle \la'_1\rangle^{2b}} \frac{d\la'_2}{\langle \la'_2\rangle^{2b}}\frac{1}{\langle \la-(\la'_1-\la'_2)\rangle^2}\bigg)\\
&\hspace{3cm}\bigg[ \sum_{k_1}\bigg(\int d\la_1 \, \langle \la_1\rangle^{2b}|\cf(w_{k_1})(\la_1)|^2\bigg)^{\frac12}\bigg(\int d\la_2 \, \langle \la_2\rangle^{2b}|\cf(z_{k_1})(\la_2)|^2\bigg)^{\frac12} \bigg]^2\\
&\lesssim \bigg(\int \frac{d\la}{\langle \la\rangle^{2-2b}}\frac{d\la'_1}{\langle \la'_1\rangle^{2b}} \frac{d\la'_2}{\langle \la'_2\rangle^{2b}}\frac{1}{\langle \la-(\la'_1-\la'_2)\rangle^2}\bigg) \, \|z\|_{\widetilde{Z}^{0,b}}^2\|w\|_{\widetilde{Z}^{0,b}}^2.
\end{align*}
Applying  Lemma \ref{lem:gtv} twice gives us successively (recall that $1<2b<\frac32$)
\begin{align*}
&\int \frac{d\la}{\langle \la\rangle^{2-2b}}\frac{d\la'_1}{\langle \la'_1\rangle^{2b}} \frac{d\la'_2}{\langle \la'_2\rangle^{2b}}\frac{1}{\langle \la-(\la'_1-\la'_2)\rangle^2}\lesssim \int \frac{d\la}{\langle \la\rangle^{2-2b}}\frac{d\la'_1}{\langle \la'_1\rangle^{2b}} \frac{1}{\langle \la-\la'_1\rangle^{2b}}\lesssim \int \frac{d\la}{\langle \la\rangle^{2}} <\infty,
\end{align*}
which achieves to prove that $\mathbf{C}<\infty$.
\end{proof}

\smallskip

\subsection{Control in $\widetilde{Z}^{c,0}$}

\begin{proposition}\label{prop:z-z-c}
For all $\frac12<b<1$ and $0\leq c\leq b$, one has
\begin{equation}
\big\|\ci_\chi \cm(z,w)\big\|_{\widetilde{Z}^{c,0}}\lesssim  \|z\|_{\widetilde{Z}^{c,b}}\|w\|_{\widetilde{Z}^{c,b}}.
\end{equation}

\end{proposition}

\begin{proof}
Using the expression in \eqref{conven-exp}, one has
\begin{align}
& \big\|\ci_\chi \cm(z,w)\big\|_{\widetilde{Z}^{c,0}}^2\nonumber\\
&\leq \sum_k \langle k\rangle^{2c}  \sum_{k_1,k_1'}\int d\la_1 d\la_1'\int d\la_2 d\la_2'\,  |\cf\big(w_{k_1}\big)(\la_1)| |\cf\big(z_{k+k_1}\big)(\la_2)| |\cf\big(w_{k'_1}\big)(\la'_1)| |\cf\big(z_{k+k_1'}\big)(\la'_2)| \nonumber\\
&\hspace{3cm}\int d\la\,  \big|\ci_{\chi}(\la,\Omega_{k,k_1}+\la_2-\la_1)\big| \big|\ci_{\chi}(\la,\Omega_{k,k'_1}+\la'_2-\la'_1)\big| .\label{fir-prod-x-c}
\end{align}
Then, by \eqref{ker-i},
\begin{align*}
&\int d\la\,  \big|\ci_{\chi}(\la,\Omega_{k,k_1}+\la_2-\la_1)\big| \big|\ci_{\chi}(\la,\Omega_{k,k'_1}+\la'_2-\la'_1)\big|  \\
&\lesssim  \int \frac{d\la }{\langle \la\rangle^{2}} \frac{1}{\langle \la-\Omega_{k,k_1}-\la_2+\la_1\rangle}\frac{1}{\langle \la-\Omega_{k,k'_1}-\la'_2+\la'_1\rangle} \\
&\lesssim \bigg(\int \frac{d\la }{\langle \la\rangle^2} \frac{1}{\langle \la-\Omega_{k,k_1}-\la_2+\la_1\rangle^2}\bigg)^{\frac12}\bigg(\int \frac{ d\la'}{\langle \la'\rangle^2}\frac{1}{\langle \la'-\Omega_{k,k'_1}-\la'_2+\la'_1\rangle^2}\bigg)^{\frac12}\\
& \lesssim \frac{1}{\langle \Omega_{k,k_1}+\la_2-\la_1\rangle}\frac{1}{\langle \Omega_{k,k'_1}+\la'_2-\la'_1\rangle},
\end{align*}
due to Lemma \ref{lem:gtv}. Injecting this estimate into \eqref{fir-prod-x-c}, we obtain that
\begin{align}
&\big\|  \ci_\chi \cm(z,w)  \big\|_{\widetilde{Z}^{c,0}}^2\lesssim \sum_k \langle k\rangle^{2c} \bigg[ \sum_{k_1}\int d\la_1 \int d\la_2 \,  |\cf\big(w_{k_1}\big)(\la_1)| |\cf\big(z_{k+k_1}\big)(\la_2)|\frac{1}{\langle \Omega_{k,k_1}+\la_2-\la_1\rangle}\bigg]^2.
\end{align}
Let us bound the latter sum as follows:
\begin{align*}
&\sum_k \langle k\rangle^{2c} \bigg[ \sum_{k_1}\int d\la_1 \int d\la_2 \,  |\cf\big(w_{k_1}\big)(\la_1)| |\cf\big(z_{k+k_1}\big)(\la_2)|\frac{1}{\langle \Omega_{k,k_1}+\la_2-\la_1\rangle}\bigg]^2\lesssim I+II+III,
\end{align*}
with
$$I:=\sum_{k\neq 0} \langle k\rangle^{2c} \bigg[ \sum_{k_1\neq 0}\int d\la_1 \int d\la_2 \,  |\cf\big(w_{k_1}\big)(\la_1)| |\cf\big(z_{k+k_1}\big)(\la_2)|\frac{1}{\langle \Omega_{k,k_1}+\la_2-\la_1\rangle}\bigg]^2,$$
$$II:=\sum_{k\neq 0} \langle k\rangle^{2c} \bigg[\int d\la_1 \int d\la_2 \,  |\cf\big(w_{0}\big)(\la_1)| |\cf\big(z_{k}\big)(\la_2)|\frac{1}{\langle \la_2-\la_1\rangle}\bigg]^2,$$
and
$$III:=\bigg[ \sum_{k_1}\int d\la_1 \int d\la_2 \,  |\cf\big(w_{k_1}\big)(\la_1)| |\cf\big(z_{k_1}\big)(\la_2)|\frac{1}{\langle \la_2-\la_1\rangle}\bigg]^2.$$

\

\noindent
\textit{Control of $I$.} For this term, using Lemma \ref{lem:gtv} twice, we have
\begin{align*}
&I=\sum_{k\neq 0} \langle k\rangle^{2c}\bigg[ \sum_{k_1\neq 0}\int d\la_1 d\la_2 \, \Big[\langle \la_1\rangle^b\langle \la_2\rangle^b |\cf\big(w_{k_1}\big)(\la_1)| |\cf\big(z_{k+k_1}\big)(\la_2)|\Big]\bigg[\frac{1}{\langle \la_1\rangle^b\langle \la_2\rangle^b}\frac{1}{\langle \Omega_{k,k_1}+\la_2-\la_1\rangle}\bigg]\bigg]^2\\
&\lesssim \sum_{k\neq 0} \langle k\rangle^{2c}\bigg( \sum_{k_1 }\int d\la_1 d\la_2 \, \langle \la_1\rangle^{2b}\langle \la_2\rangle^{2b} |\cf\big(w_{k_1}\big)(\la_1)|^2 |\cf\big(z_{k+k_1}\big)(\la_2)|^2\bigg)\\
&\hspace{8cm}\bigg(\sum_{k_1'\neq 0}\int \frac{d\la_1 d\la_2}{\langle \la_1\rangle^{2b}\langle \la_2\rangle^{2b}}\frac{1}{\langle \Omega_{k,k'_1}+\la_2-\la_1\rangle^{2}}\bigg)\\
&\lesssim \sum_{k\neq 0} \langle k\rangle^{2c}\bigg( \sum_{k_1 }\int d\la_1 d\la_2 \, \langle \la_1\rangle^{2b}\langle \la_2\rangle^{2b} |\cf\big(w_{k_1}\big)(\la_1)|^2 |\cf\big(z_{k+k_1}\big)(\la_2)|^2\bigg)\bigg(\sum_{k_1'\neq 0}\frac{1}{\langle \Omega_{k,k'_1}\rangle^{2b}}\bigg)\\
&\lesssim \sum_{k\neq 0} \langle k\rangle^{2(c-b)} \sum_{k_1 }\int d\la_1 d\la_2 \, \langle \la_1\rangle^{2b}\langle \la_2\rangle^{2b} |\cf\big(w_{k_1}\big)(\la_1)|^2 |\cf\big(z_{k+k_1}\big)(\la_2)|^2\\
&\lesssim \sum_{k} \sum_{k_1 }\int d\la_1 d\la_2 \, \langle \la_1\rangle^{2b}\langle \la_2\rangle^{2b} |\cf\big(w_{k_1}\big)(\la_1)|^2 |\cf\big(z_{k+k_1}\big)(\la_2)|^2 \ \lesssim \ \big\|z\big\|_{\widetilde{Z}^{0,b}}^2\big\|w\big\|_{\widetilde{Z}^{0,b}}^2,
\end{align*}
where we have used the assumption $c-b\leq 0$.

\

\noindent
\textit{Control of $II$.} One has in this case
\begin{align*}
II&=\sum_{k\neq 0} \langle k\rangle^{2c} \bigg[\int d\la_2 \,  |\cf\big(z_{k}\big)(\la_2)| \bigg(\int d\la_1 \, |\cf\big(w_{0}\big)(\la_1)|\frac{1}{\langle \la_2-\la_1\rangle}\bigg)\bigg]^2\\
&\lesssim \sum_{k\neq 0} \langle k\rangle^{2c} \bigg(\int d\la'_2 \,  |\cf\big(z_{k}\big)(\la'_2)|^2\bigg) \bigg(\int d\la_2\, \bigg(\int d\la_1 \, |\cf\big(w_{0}\big)(\la_1)|\frac{1}{\langle \la_2-\la_1\rangle}\bigg)^2\bigg)\\
&\lesssim \big\|z\big\|_{\widetilde{Z}^{c,0}}^2 \bigg[\int d\la_2\, \bigg(\int d\la_1 \, \big[\langle \la_1\rangle^b|\cf\big(w_{0}\big)(\la_1)|\big]\bigg[\frac{1}{\langle \la_1\rangle^b\langle \la_2-\la_1\rangle}\bigg]\bigg)^2\bigg]\\
&\lesssim \big\|z\big\|_{\widetilde{Z}^{c,0}}^2 \big\|w\big\|_{\widetilde{Z}^{0,b}}^2 \int\frac{d\la_2 d\la_1}{\langle \la_1\rangle^{2b}\langle \la_2-\la_1\rangle^{2}}\  \lesssim \ \big\|z\big\|_{\widetilde{Z}^{c,0}}^2 \big\|w\big\|_{\widetilde{Z}^{0,b}}^2 .
\end{align*}

\

\noindent
\textit{Control of $III$.} Finally, for this term, we have 
\begin{align*}
III&=\bigg[ \int d\la_1  d\la_2 \, \bigg(\sum_{k_1} \langle\la_1\rangle^b |\cf\big(w_{k_1}\big)(\la_1)|\cdot \langle \la_2\rangle^b|\cf\big(z_{k_1}\big)(\la_2)|\bigg)\bigg(\frac{1}{\langle\la_1\rangle^b \langle \la_2\rangle^b\langle \la_2-\la_1\rangle}\bigg)\bigg]^2\\
&\leq  \bigg[ \int d\la_1  d\la_2 \, \bigg(\sum_{k_1} \langle\la_1\rangle^b |\cf\big(w_{k_1}\big)(\la_1)|\cdot \langle \la_2\rangle^b |\cf\big(z_{k_1}\big)(\la_2)|\bigg)^2\bigg]\bigg[\int\frac{d\la_1d\la_2}{\langle\la_1\rangle^{2b} \langle \la_2\rangle^{2b}\langle \la_2-\la_1\rangle^{2}}\bigg]\\
&\lesssim \bigg[ \int d\la_1  d\la_2 \, \bigg(\sum_{k_1} \langle\la_1\rangle^{2b} |\cf\big(w_{k_1}\big)(\la_1)|^2\bigg)\bigg(\sum_{k'_1} \langle\la_2\rangle^{2b} |\cf\big(z_{k'_1}\big)(\la_2)|^2\bigg)\bigg]\ \lesssim \ \big\|z\big\|_{\widetilde{Z}^{0,b}}^2\big\|w\big\|_{\widetilde{Z}^{0,b}}^2.
\end{align*}

\end{proof}

\color{black}

\section{Estimate of the product-operator norm}\label{sec:discuss-p-n}

We gather here a few elements to justify our consideration of the quantity ${\cop \cpti^{(n)}_{c,b}}$ in \eqref{defi:cp-n} as an estimate of the operator norm $\big\|\cl^{\sharp,(n)}\big\|_{{\cop \widetilde{Z}^{c,b}} \to {\cop \widetilde{Z}^{c,b}}}$.

\subsection{General operator estimates in Besov spaces}\label{subsec:gene-ope-esti}

Given a general kernel $K\in L^2(\R^d \times \R^d)$, define the operator $\cl_K: L^2(\R^d)\to \cl^2(\R^d)$ by
\begin{equation}\label{kernel-pattern}
\cf(\cl_K f)(\la)=\int_{\R^d} d\la_1 \, K(\la,\la_1)(\cf f)(\la_1), \quad \la\in \R^d.
\end{equation}
For each $\underline{b}=(b_1,\ldots,b_d)$, consider the anisotropic Besov space $\ch^{\underline{b}}$ on $\R^d$ related to the Fourier multiplier
$$\langle \la\rangle_{\underline{b}}:=\bigg(\sum_{i=1}^d |\la_i|^{2b_i}\bigg)^{\frac12}.$$


\

If $K$ has no specific a priori structure, then we are essentially confined to very general transformations toward the operator norm of $\cl_K$ (from $\ch^{\underline{b}}$ to $\ch^{\underline{b}'}$): namely,
\begin{align}
\big\| \cl_K f\big\|_{\ch^{\underline{b}'}}^2&=\int d\la \, \langle \la\rangle_{\underline{b}'}^2  \bigg|\int d\la_1  \, K(\la,\la_1)(\cf f)(\la_1)\bigg|^2\nonumber\\
&=\int d\la \, \langle \la\rangle_{\underline{b}'}^2  \bigg|\int d\la_1  \, \frac{1}{\langle \la_1\rangle_{\underline{b}}}K(\la,\la_1)\big(\langle \la_1\rangle_{\underline{b}}(\cf f)(\la_1)\big)\bigg|^2\label{cs-basic}\\
&= \int d\la_1 d\la_1' \, \bigg[ \frac{1}{\langle \la_1\rangle_{\underline{b}}}\frac{1}{\langle \la_1'\rangle_{\underline{b}}}\int d\la \, \langle \la\rangle_{\underline{b}'}^2 \, K(\la,\la_1)\overline{K(\la,\la_1')}\bigg] \Big[\big( \langle \la_1\rangle_{\underline{b}} (\cf f)(\la_1)\big) \big( \langle\la_1'\rangle_{\underline{b}}\overline{(\cf f)(\la_1')}\big)\Big]\nonumber\\
&\leq \bigg( \int \frac{d\la_1}{\langle \la_1\rangle_{\underline{b}}^{2}}\frac{d\la_1'}{\langle \la_1'\rangle_{\underline{b}}^{2}} \,\bigg|\int d\la \, \langle \la\rangle_{\underline{b}'}^2\, K(\la,\la_1)\overline{K(\la,\la_1')}\bigg|^2\bigg)^{\frac12} \, \big\|f\big\|_{\ch^{\underline{b}}}^2,\nonumber
\end{align}
which provides us with the general bound
\begin{equation}\label{estim-cal-p}
\big\|\cl_K\big\|_{\ch^{\underline{b}}\to \ch^{\underline{b}'}} \leq \cp_{\underline{b},\underline{b}'}^{\frac14}, \quad \text{where} \ \cp_{\underline{b},\underline{b}'}:= \int \frac{d\la_1}{\langle \la_1\rangle_{\underline{b}}^{2}}\frac{d\la_1'}{\langle \la_1'\rangle_{\underline{b}}^{2}} \,\bigg|\int d\la \, \langle \la\rangle_{\underline{b}'}^2\, K(\la,\la_1)\overline{K(\la,\la_1')}\bigg|^2.
\end{equation}

\

\begin{remark}
The operator norm $\big\|\cl_K\big\|_{\ch^{\underline{b}}\to \ch^{\underline{b}'}}$ can be more basically - but less sharply - bounded by the Hilbert-Schmidt norm of $\cl_K$, that is the quantity 
$$\cq_{\underline{b},\underline{b}'}^{\frac12}:=\bigg(\int d\la \, \langle \la\rangle_{\underline{b}'}^2\int \frac{d\la_1}{\langle \la_1\rangle_{\underline{b}}^{2}} \big|K(\la,\la_1)\big|^2\bigg)^{\frac12},$$
as it can be seen from a straightforward application of Cauchy-Schwarz inequality in \eqref{cs-basic}.

\smallskip

A classical example where the three quantities $\big\|\cl_K\big\|_{\ch^{\underline{b}}\to \ch^{\underline{b}'}}$, $\cp_{\underline{b},\underline{b}'}^{\frac14}$ and $\cq_{\underline{b},\underline{b}'}^{\frac12}$ can be compared is given by the elementary kernel $K(t,s):=\1_{\{0\leq s\leq t\leq 1\}}$ in $L^2(\R)$, or more exactly by the integration operator
$$\cl f(t):=\1_{[0,1]}(t)\int_0^t f(s) \, ds.$$
In this case, it can be checked that
$$\big\|\cl\big\|_{{L^2}\to L^2}=\frac{2}{\pi}\approx 0.637, \quad \cp_{0,0}^{\frac14}=\Big(\frac{1}{6}\Big)^{\frac14} \approx 0.639,  \quad \text{and} \quad  \cq_{0,0}^{\frac12}=\frac{1}{\sqrt{2}}\approx 0.707.$$
\end{remark}

\

\subsection{Example: Young integration as random operator}\label{subsec:young-op}

As an additional justification of our focus on ${\cop \cpti^{(n)}_{c,b}}$, let us consider the (well-known) case of the fractional integration operator.

\smallskip

Thus, consider a fractional noise $\dot{B}$ on $\R$ with index $H\in (0,1)$, as well as a smooth approximation $\dot{B}^{(n)}$ given for instance by
\begin{equation}\label{appro-wn}
\dot{B}^{(n)}(t):=\int_{\{|\xi|\leq 2^n\}} {\cop \frac{\xi}{|\xi|^{H+\frac12}}}e^{-\imath t\xi}\, \widehat{W}(d\xi) ,
\end{equation}
where $\widehat{W}$ stands for the Fourier transform of a Wiener process $W$ on $\R$.

\smallskip

Then define the (local) fractional integration operator in a standard way: for any regular $z:\R\to \R$ and $t\in \R$, 
\begin{equation}
\big(\cl^{(n)}z\big)(t):=\chi(t)\int_0^t ds \, \chi(s) z(s) \dot{B}^{(n)}(s),
\end{equation}
where $\chi: \R \to \R$ stands for localizing cut-off function, that is smooth, positive and compactly-supported.

\smallskip

In Fourier mode, the operator $\cl^{(n)}$ can easily be written along the pattern of \eqref{kernel-pattern}, that is as
\begin{align*}
\cf\big( \cl^{(n)}z\big)(\la)=\int d\la_1 \, K^{(n)}(\la ,\la_1) \, (\cf z)(\la_1),
\end{align*}
{\cop with $K^{(n)}$ explicitly given by
\begin{equation}\label{kernel-integration-simple}
K^{(n)}(\la,\la_1):=\int_{\R} dt \, e^{-\imath \la t} \chi(t)\int_0^t ds \,\chi(s) e^{\imath \la_1 s} \dot{B}^{(n)}(s).
\end{equation}}
Setting
$$\big\|z\big\|_{X^b}^2:=\int d\la \, \langle \la\rangle^{2b} \big|\cf(z)\big|^2$$
and applying \eqref{estim-cal-p}, we deduce that
\begin{equation}\label{defi:p-n-y}
\big\| \cl^{(n)}\big\|_{X^b \to X^b} \lesssim \big(\cp_b^{(n)})^{\frac14} ,\quad \text{with} \ \ \cp_b^{(n)}:=\int \frac{d\la_1}{\langle \la_1\rangle^{2b}}\frac{d\la_1'}{\langle \la_1'\rangle^{2b}} \,\bigg|\int d\la \, \langle \la\rangle^{2b}\, K^{(n)}(\la,\la_1)\overline{K^{(n)}(\la,\la_1')}\bigg|^2.
\end{equation}

Focusing exclusively on this explicit quantity $\cp_b^{(n)}$ allows us to recover the classical \enquote{Young} dichotomy of fractional integration theory.

\begin{proposition}\label{prop:young-p}
In the above setting, the following picture holds true:

\smallskip

\noindent
$(i)$ If $H>\frac12$, then for every $b\in (\frac12,H)$, one has
$$\sup_{n} \, \mathbb{E}\Big[\cp_b^{(n)}\Big] < \infty. $$

\smallskip

\noindent
$(ii)$ If $H= \frac12$, then for every $b\in \R$, one has
$$\mathbb{E}\Big[\cp_b^{(n)}\Big] \stackrel{n\to \infty}{\longrightarrow} \infty.$$
\end{proposition}

\cop

\begin{proof}
We only give an account of the phenomenon occurring at the Brownian threshold - that is, concerning the proof of item $(ii)$. Thus, we fix $H=\frac12$.

\smallskip

Suppose, aiming for a contradiction, that there exists $b\in \R$ such that
\begin{equation}\label{futur-contradic}
\sup_{n\geq 0} \mathbb{E}\Big[\cp_b^{(n)}\Big] <\infty.
\end{equation}

\

Let us first prove that one has necessarily $b< 1$. Indeed, for $z:=\chi \in \bigcap_{b\in \R} X^b$,
\begin{align*}
\mathbb{E}\Big[\cp_{1}^{(n)}\Big]& \gtrsim \mathbb{E}\Big[ \big\| \cl^{(n)}(\chi)\big\|_{X^{1}}^2\Big] \geq \mathbb{E}\Big[ \big\| \partial_t \big(\cl^{(n)}(\chi)\big)\big\|_{L^2}^2\Big]\\
&\geq\int_{\R} dt \,\mathbb{E}\bigg[ \Big| \chi(t)^3 \dot{B}^{(n)}(t)+\chi'(t)\int_0^t ds \, \chi(s)^2\dot{B}^{(n)}(s)\Big|^2\bigg]\\
&\geq\int_{\R} dt \, \chi(t)^6 \mathbb{E}\Big[ \big|\dot{B}^{(n)}(t)\big|^2\Big]+2\, Re\bigg( \int_{\R} dt \,  \chi(t)^3 \chi'(t)\int_0^t ds \, \chi(s)^2\mathbb{E}\Big[\dot{B}^{(n)}(t)\overline{\dot{B}^{(n)}(s)}\Big]\bigg).
\end{align*}
Using the expression
\begin{equation}\label{expr-cov-bro}
\mathbb{E}\Big[ \dot{B}^{(n)}(s)\overline{\dot{B}^{(n)}(s')}\Big]=\mathbb{E}\Big[ \dot{B}^{(n)}(s)\dot{B}^{(n)}(s')\Big]=c \int_{\{|\xi|\leq 2^n \}}d\xi\, e^{-\imath \xi(s-s')},
\end{equation}
we can then decompose the second term as
\small
\begin{align*}
&2\, Re\bigg( \int_{\R} dt \,  \chi(t)^3 \chi'(t)\int_0^t ds \, \chi(s)^2\mathbb{E}\Big[\dot{B}^{(n)}(t)\overline{\dot{B}^{(n)}(s)}\Big]\bigg)\\
&= 2 \,  Re\bigg( \int_{\{|\xi|\leq 2^n\}} d\xi \int_{\R} dt \,  e^{-\imath \xi t}\chi(t)^3 \chi'(t)\int_0^t ds \, e^{\imath \xi s} \chi(s)^2\bigg)\\
&=2\, \int_{\{|\xi|\leq 1\}} d\xi \int_{\R} dt \,  e^{-\imath \xi t}\chi(t)^3 \chi'(t)\int_0^t ds \, e^{\imath \xi s} \chi(s)^2 -2\imath \int_{\{1\leq |\xi|\leq 2^n\}} \frac{d\xi}{\xi} \int_{\R} dt \,  e^{-\imath \xi t}\partial_t(\chi^3 \chi')(t) \int_0^t ds \, e^{\imath \xi s} \chi(s)^2,
\end{align*}
\normalsize
which allows us to assert that
\begin{align*}
&\sup_{n\geq 1} \bigg|2 \, Re\bigg( \int_{\R} dt \,  \chi(t)^3 \chi'(t)\int_0^t ds \, \chi(s)^2\mathbb{E}\Big[\dot{B}^{(n)}(t)\overline{\dot{B}^{(n)}(s)}\Big]\bigg) \bigg|\lesssim 1+\int_1^\infty \frac{d\xi}{\xi^2} < \infty.
\end{align*}
On the other hand, one has of course
\begin{align*}
&\int_{\R} dt \, \chi(t)^6 \mathbb{E}\Big[ \big|\dot{B}^{(n)}(t)\big|^2\Big]=\bigg(\int_{\R} dt \, \chi(t)^6\bigg) \bigg(\int_{\{|\xi|\leq 2^n\}}d\xi \bigg) \stackrel{n\to\infty}{\longrightarrow} \infty,
\end{align*}
which achieves to prove that 
$$\mathbb{E}\Big[\cp_{1}^{(n)}\Big] \stackrel{n\to\infty}{\longrightarrow} \infty,$$
and accordingly, if \eqref{futur-contradic} holds true, one must have $b< 1$.

\

Assume now that $\frac12\leq  b < 1$. Still focusing on the test-function $z:=\chi$, we can write using \eqref{expr-cov-bro}
\begin{align*}
&\mathbb{E}\Big[\cp_{b}^{(n)}\Big] \gtrsim \mathbb{E} \Big[\big\| \cl^{(n)}(\chi)\big\|_{X^b}^2 \Big]\geq \int_1^\infty d\la \, |\la|^{2b} \mathbb{E}\Big[ \big|\cf\big(\cl^{(n)}(\chi)\big)(\la)\big|^2\Big]\\
&\geq \int_1^\infty d\la \, |\la|^{2b}\int dt dt' \, e^{-\imath \la (t-t')}\chi(t)\chi(t') \int_0^t ds \int_0^{t'} ds' \, \chi(s)^2\chi(s')^2 \mathbb{E}\Big[ \dot{B}^{(n)}(s)\overline{\dot{B}^{(n)}(s')}\Big]\\
&\geq \int_1^\infty d\la \, |\la|^{2b}\int _{\{|\xi| \leq 2^n\}} d\xi \, \bigg(\int dt \, e^{-\imath \la t} \chi(t) \int_0^t ds \, \chi(s)^2 e^{-\imath s\xi}\bigg) \bigg(\int dt' \, e^{\imath \la t'} \chi(t') \int_0^{t'} ds' \, \chi(s')^2 e^{\imath s'\xi}\bigg),
\end{align*}
which, by elementary integrations by parts, entails
\begin{align*}
&\mathbb{E} \Big[\big\| \cl^{(n)}(\chi)\big\|_{X^b}^2 \Big]\gtrsim M^{(n)}_b+\int_1^\infty \frac{d\la}{|\la|^{2-2b}}\int _{\{|\xi| \leq 2^n\}} d\xi \, R(\la,\xi),
\end{align*}
where
\begin{align*}
M^{(n)}_b:=\int_1^\infty \frac{d\la}{|\la|^{2-2b}}\int _{\{|\xi| \leq 2^n\}} d\xi \, \bigg(\int dt \, e^{-\imath (\la+\xi) t} \chi(t)^3 \bigg) \bigg(\int dt' \, e^{\imath (\la+\xi) t'} \chi(t')^3 \bigg)
\end{align*}
and $R$ satisfies (thanks to \eqref{ker-i})
\begin{align*}
|R(\la,\xi)|\lesssim \frac{1}{\langle \la \rangle \langle \la+\xi\rangle^2}.
\end{align*}
Based on the latter estimate, it is clear that
\begin{align*}
\sup_{n\geq 1} \bigg|\int_1^\infty \frac{d\la}{|\la|^{2-2b}}\int _{\{|\xi| \leq 2^n\}} d\xi \, R(\la,\xi)\bigg|\lesssim \int_1^\infty \frac{d\la}{|\la|^{3-2b}}\int _{\R} \frac{d\xi}{\langle \xi\rangle^2} <\infty,
\end{align*}
due to $\frac12\leq b<1$. The contradiction now comes from the fact that
\begin{align*}
M^{(n)}_b\stackrel{n\to\infty}{\longrightarrow}\int_1^\infty \frac{d\la}{|\la|^{2-2b}}\bigg(\int dt \, \chi(t)^6 \bigg),
\end{align*}
which can only be finite if $b<\frac12$.

\

Next, to rule out the case $b<\frac14$, let us start from the inequality
\begin{align*}
\cp_b^{(n)}\gtrsim \int_1^\infty \frac{d\la_1}{|\la_1|^{2b}}\int_1^\infty \frac{d\la_1'}{|\la_1'|^{2b}} \,\bigg|\int_1^2 d\la \, |\la|^{2b}\, K^{(n)}(\la,\la_1)\overline{K^{(n)}(\la,\la_1')}\bigg|^2
\end{align*}
which yields
\begin{align}
\mathbb{E}\Big[\cp_b^{(n)}\Big]\gtrsim \int_1^\infty \frac{d\la_1}{|\la_1|^{2b}}\int_1^\infty \frac{d\la_1'}{|\la_1'|^{2b}} \,\bigg|\int_1^2 d\la \, |\la|^{2b}\, \mathbb{E}\Big[ K^{(n)}(\la,\la_1)\overline{K^{(n)}(\la,\la_1')}\Big]\bigg|^2. \label{pn-14}
\end{align}
Combining \eqref{kernel-integration-simple} and \eqref{expr-cov-bro}, we obtain that for all $\la,\la',\la_1,\la_1'\in \R$,
\begin{align}
&\mathbb{E}\Big[K^{(n)}(\la,\la_1)\overline{K^{(n)}(\la',\la_1')}\Big]\nonumber\\
&=c \int_{\{|\xi|\leq 2^n \}}d\xi\, \bigg(\int dt \, e^{-\imath \la t}\chi(t)\int_0^t ds \, e^{\imath (\la_1-\xi) s} \chi(s)\bigg) \bigg( \int dt' \, e^{\imath \la' t'}\chi(t')\int_0^{t'} ds' \, e^{-\imath (\la_1'-\xi) s'} \chi(s')\bigg).\label{covari-k-k}
\end{align}
From this expression, we can see in particular that
\begin{align}
&\mathbb{E}\Big[K^{(n)}(\la,\la_1)\overline{K^{(n)}(\la',\la_1')}\Big] \stackrel{n\to \infty}{\longrightarrow} F(\la,\la',\la_1'-\la_1),\label{intr-lim-f}
\end{align}
where
\begin{align*}
&F(\la,\la',\beta):=c\int dt \, e^{-\imath \la t}\chi(t)\int dt' \, e^{\imath \la' t'}\chi(t')\int_0^t ds \, e^{-\imath \beta  s}  \chi(s)^2\1_{[0,t']}(s) .
\end{align*}
Going back to \eqref{pn-14}, we deduce that
\begin{align*}
\sup_{n\geq 1}\mathbb{E}\Big[\cp_b^{(n)}\Big]\gtrsim \int_1^\infty \frac{d\la_1}{|\la_1|^{2b}}\int_1^\infty \frac{d\la_1'}{|\la_1'|^{2b}} \big| G_b(\la_1'-\la_1)\big|^2, \quad \text{with} \ G_b(\beta):=\int_1^2 d\la \, |\la|^{2b}\, F(\la,\la,\beta) .
\end{align*} 
Therefore
\begin{align*}
\sup_{n\geq 1}\mathbb{E}\Big[\cp_b^{(n)}\Big]&\gtrsim \int_1^\infty \frac{d\la_1}{|\la_1|^{2b}}\int_{\la_1+1}^{\la_1+2} \frac{d\la_1'}{|\la_1'|^{2b}} \big| G_b(\la_1'-\la_1)\big|^2\\
&\gtrsim \int_{1}^{2} d\rho\, \big| G_b(\rho)\big|^2 \int_1^\infty \frac{d\la_1}{|\la_1|^{2b}|\la_1+\rho|^{2b}}\gtrsim \int_{1}^{2} d\rho\, \big| G_b(\rho)\big|^2 \int_1^\infty \frac{d\la_1}{|\la_1|^{2b}|\la_1+2|^{2b}},
\end{align*} 
and in light of the hypothesis \eqref{futur-contradic}, we can conclude that $b>\frac14$.

\

It only remains us to exclude the possibility that $\frac14< b < \frac12$. To this end, we start from
\begin{align*}
\cp_b^{(n)}\gtrsim \widetilde{\cp_b}^{(n)}:=\int_1^\infty \frac{d\la_1}{|\la_1|^{2b}}\int_1^\infty \frac{d\la_1'}{|\la_1'|^{2b}} \,\bigg|\int_1^\infty d\la \, |\la|^{2b}\, K^{(n)}(\la,\la_1)\overline{K^{(n)}(\la,\la_1')}\bigg|^2.
\end{align*}
By applying Wick's formula, we obtain the decomposition
\begin{align*}
\mathbb{E}\Big[\widetilde{\cp_b}^{(n)}\Big]&=\int_1^\infty \frac{d\la_1}{|\la_1|^{2b}}\int_1^\infty \frac{d\la_1'}{|\la_1'|^{2b}} \int_1^\infty d\la \, |\la|^{2b} \int_1^\infty d\la' \, |\la'|^{2b}\, \mathbb{E}\Big[K^{(n)}(\la,\la_1)\overline{K^{(n)}(\la,\la_1')}\overline{K^{(n)}(\la',\la_1)}K^{(n)}(\la',\la_1')\Big]\\
&=I^{(n)}_b+II^{(n)}_b+III^{(n)}_b,
\end{align*}
with
\begin{align*}
I^{(n)}_b:=\int_1^\infty \frac{d\la_1}{|\la_1|^{2b}}\int_1^\infty \frac{d\la_1'}{|\la_1'|^{2b}}\bigg| \int_1^\infty d\la \, |\la|^{2b} \, \mathbb{E}\Big[K^{(n)}(\la,\la_1)\overline{K^{(n)}(\la,\la_1')}\Big]\bigg|^2,
\end{align*}
\begin{align*}
II^{(n)}_b:=\int_1^\infty \frac{d\la_1}{|\la_1|^{2b}}\int_1^\infty \frac{d\la_1'}{|\la_1'|^{2b}} \int_1^\infty d\la \, |\la|^{2b} \int_1^\infty d\la' \, |\la'|^{2b}\, \mathbb{E}\Big[K^{(n)}(\la,\la_1)\overline{K^{(n)}(\la',\la_1)}\Big]\mathbb{E}\Big[\overline{K^{(n)}(\la,\la_1')}K^{(n)}(\la',\la_1')\Big],
\end{align*}
\begin{align*}
III^{(n)}_b:=\int_1^\infty \frac{d\la_1}{|\la_1|^{2b}}\int_1^\infty \frac{d\la_1'}{|\la_1'|^{2b}} \int_1^\infty d\la \, |\la|^{2b} \int_1^\infty d\la' \, |\la'|^{2b}\, \mathbb{E}\Big[K^{(n)}(\la,\la_1)K^{(n)}(\la',\la_1')\Big] \mathbb{E}\Big[\overline{K^{(n)}(\la,\la_1')}\overline{K^{(n)}(\la',\la_1)}\Big].
\end{align*}
Note in particular that $I^{(n)}_b\geq 0$, and so
\begin{align}
\mathbb{E}\Big[\widetilde{\cp_b}^{(n)}\Big]&\geq II^{(n)}_b+III^{(n)}_b.\label{iin-iiin}
\end{align}
Similarly to \eqref{covari-k-k}, one has
\begin{align*}
&\mathbb{E}\Big[K^{(n)}(\la,\la_1)K^{(n)}(\la',\la_1')\Big]\\
&=c \int_{\{|\xi|\leq 2^n \}}d\xi\, \bigg(\int dt \, e^{-\imath \la t}\chi(t)\int_0^t ds \, e^{\imath (\la_1-\xi) s} \chi(s)\bigg) \bigg( \int dt' \, e^{-\imath \la' t'}\chi(t')\int_0^{t'} ds' \, e^{\imath (\la_1'+\xi) s'} \chi(s')\bigg),
\end{align*}
from which we easily deduce, using \eqref{ker-i} and Lemma \eqref{lem:gtv}, 
\begin{align*}
&\Big|\mathbb{E}\Big[K^{(n)}(\la,\la_1)K^{(n)}(\la',\la_1')\Big]\Big|\\
&\lesssim  \int_{\R}d\xi\, \bigg|\int dt \, e^{-\imath \la t}\chi(t)\int_0^t ds \, e^{\imath (\la_1-\xi) s} \chi(s)\bigg| \bigg| \int dt' \, e^{-\imath \la' t'}\chi(t')\int_0^{t'} ds' \, e^{\imath (\la_1'+\xi) s'} \chi(s')\bigg|\\
&\lesssim \frac{1}{\langle \la\rangle \langle \la'\rangle} \int_{\R}d\xi\, \frac{1}{\langle \xi+\la-\la_1\rangle } \frac{1}{\langle \xi-\la'+\la_1' \rangle}\lesssim \frac{1}{\langle \la\rangle \langle \la'\rangle} \frac{1}{\langle (\la_1+\la_1')-(\la+\la') \rangle^{1-\varepsilon}},
\end{align*}
for any $\varepsilon >0$. Injecting this bound into $III^{(n)}_b$, we obtain the uniform estimate
\begin{align*}
\sup_{n\geq 1}\big| III^{(n)}_b \big| \lesssim \int_{\R} \frac{d\la }{\langle \la\rangle^{2-2b}}  \int_{\R} \frac{d\la' }{ \langle \la'\rangle^{2-2b}} \int_{\R} \frac{d\la_1}{\langle \la_1\rangle^{2b}}\int_{\R} \frac{d\la_1'}{\langle \la_1'\rangle^{2b}} \frac{1}{\langle (\la_1+\la_1')-(\la+\la') \rangle^{2-2\varepsilon}},
\end{align*}
and by applying Lemma \ref{lem:gtv} twice, we get
\begin{align*}
\sup_{n\geq 1}\big| III^{(n)}_b \big| &\lesssim \int_{\R} \frac{d\la }{\langle \la\rangle^{2-2b}}  \int_{\R} \frac{d\la' }{ \langle \la'\rangle^{2-2b}} \int_{\R} \frac{d\la_1}{\langle \la_1\rangle^{2b}} \frac{1}{\langle \la_1-(\la+\la') \rangle^{2b}}\\
&\lesssim \int_{\R} \frac{d\la }{\langle \la\rangle^{2-2b}}  \int_{\R} \frac{d\la' }{ \langle \la'\rangle^{2-2b}} \frac{1}{\langle \la+\la' \rangle^{4b-1}} <\infty,
\end{align*}
where we have used the fact that $\frac14<b<\frac12$.

\smallskip

Finally, to handle the quantity $II^{(n)}_b$, we can use \eqref{intr-lim-f} and assert that
\begin{align*}
&\liminf_{n\to \infty} II^{(n)}_b \geq \bigg(\int_1^\infty \frac{d\la_1}{|\la_1|^{2b}}\int_1^\infty \frac{d\la_1'}{|\la_1'|^{2b}}\bigg) \bigg( \int_1^\infty d\la \, |\la|^{2b} \int_1^\infty d\la' \, |\la'|^{2b}\, \big| F(\la,\la',0)\big|^2\bigg).
\end{align*}
Of course the latter limit is positive and can only be finite for $b>\frac12$, which, going back to \eqref{iin-iiin}, achieves to prove the contradiction. In other words, we cannot find any $b\in \R$ such that \eqref{futur-contradic} holds true.

\end{proof}

\

\color{blue}

\noindent
{ \textbf{Acknowledgements.} I am deeply grateful to the two anonymous reviewers who have meticulously overseen this study. Their numerous comments have significantly contributed to improving both the presentation and the substance of the article. I also wish to thank them for bringing to my attention several additional publications related to this topic, which have enriched the bibliography accordingly}.

\color{black}

\end{document}